\documentclass[a4paper,11pt]{article}

\usepackage[top=1in, bottom=1.25in, left=1.25in, right=1.25in]{geometry}
\usepackage{amsfonts,amssymb,amsmath,amsthm} 
\usepackage{srcltx,tabularx}
\usepackage[utf8]{inputenc} 
\usepackage[cyr]{aeguill}
\usepackage[alphabetic, nobysame]{amsrefs}
\usepackage{enumerate}
\usepackage[english]{babel}
\usepackage{authblk}
\usepackage[hidelinks, pagebackref=false]{hyperref}
\usepackage{graphicx}
\usepackage{pstricks,pst-plot,pst-node}
\usepackage{pstricks-add}
\usepackage[bottom]{footmisc}
\usepackage[all]{hypcap}
\usepackage[toc,page]{appendix}
\usepackage{breqn}
\usepackage{arydshln}
\usepackage{stfloats}
\usepackage{caption}
\usepackage{subcaption}

\newcommand\Tstrut{\rule{0pt}{2.6ex}}         
\newcommand\Bstrut{\rule[-0.9ex]{0pt}{0pt}}   

\newtheorem{innerprimetheorem}{Theorem}
\newenvironment{primetheorem}[1]
  {\renewcommand\theinnerprimetheorem{\ref*{#1}$'$}\innerprimetheorem}
  {\endinnerprimetheorem}
  
\newtheorem{innerrepeattheorem}{Theorem}
\newenvironment{repeattheorem}[1]
  {\renewcommand\theinnerrepeattheorem{\ref*{#1}}\innerrepeattheorem}
  {\endinnerrepeattheorem}
  
\newtheorem{innerprimecorollary}{Corollary}
\newenvironment{primecorollary}[1]
  {\renewcommand\theinnerprimecorollary{\ref*{#1}$'$}\innerprimecorollary}
  {\endinnerprimecorollary}

\newtheorem{theorem}{Theorem}[section]
\newtheorem{lemma}[theorem]{Lemma}
\newtheorem{proposition}[theorem]{Proposition}
\newtheorem{corollary}[theorem]{Corollary}
\newtheorem{claim}{Claim}
\newtheorem*{claim*}{Claim}

\newenvironment{claimproof}[1]{\par\noindent\textit{Proof of the claim:}\space#1}{\hfill $\blacksquare$}

\newtheorem{maintheorem}{Theorem}
\newtheorem{maincorollary}[maintheorem]{Corollary}

\theoremstyle{definition}
\newtheorem{definition}[theorem]{Definition}
\newtheorem*{definition*}{Definition}
\newtheorem*{remark*}{Remark}

\newcommand{\suchthat}{\;\ifnum\currentgrouptype=16 \middle\fi|\;}

\newcommand{\bigslant}[2]{{\raisebox{.2em}{$#1$}\left/\raisebox{-.2em}{$#2$}\right.}}
\DeclareMathOperator{\Aut}{\mathrm{Aut}}
\DeclareMathOperator{\PSL}{\mathrm{PSL}}
\DeclareMathOperator{\SL}{\mathrm{SL}}
\DeclareMathOperator{\Fix}{\mathrm{Fix}}

\DeclareMathOperator{\id}{\mathrm{id}}
\DeclareMathOperator{\Sym}{\mathrm{Sym}}
\DeclareMathOperator{\Alt}{\mathrm{Alt}}
\DeclareMathOperator{\proj}{\mathrm{proj}}
\DeclareMathOperator{\Inn}{\mathrm{Inn}}
\DeclareMathOperator{\Rist}{\mathrm{Rist}}
\DeclareMathOperator{\Out}{\mathrm{Out}}
\DeclareMathOperator{\im}{\mathrm{im}}
\DeclareMathOperator{\C}{\mathbf{C}}
\DeclareMathOperator{\F}{\mathbf{F}}
\DeclareMathOperator{\D}{\mathbf{D}}
\DeclareMathOperator{\N}{\mathbf{Z}_{\geq 0}}
\DeclareMathOperator{\Nz}{\mathbf{Z}_{> 0}}
\DeclareMathOperator{\sgn}{\mathrm{sgn}}
\DeclareMathOperator{\Sgn}{\mathrm{Sgn}}
\DeclareMathOperator{\Diag}{\Delta}

\title{A classification theorem for boundary $2$-transitive automorphism groups of trees}
\author{Nicolas Radu\thanks{F.R.S.-FNRS Research Fellow.}}

\affil{UCLouvain, 1348 Louvain-la-Neuve, Belgium}
\date{November 30, 2016}

\begin{document}


\maketitle


\begin{abstract}
Let $T$ be a locally finite tree all of whose vertices have valency at least~$6$. We classify, up to isomorphism, the closed subgroups of $\Aut(T)$ acting $2$-transitively on the set of ends of $T$ and whose local action at each vertex contains the alternating group. The outcome of the classification for a fixed tree $T$ is a countable family of groups, all containing two remarkable subgroups: a simple subgroup of index $\leq 8$ and (the semiregular analog of) the universal locally alternating group of Burger--Mozes (with possibly infinite index). We also provide an explicit example showing that the statement of this classification fails for trees of smaller degree.
\end{abstract}


\tableofcontents


\section{Introduction}

Let $T$ be the $(d_0, d_1)$-semiregular tree, with $d_0, d_1 \geq 6$. The goal of this paper is to provide a complete classification of the closed subgroups of $\Aut(T)$ acting $2$-transitively on the set of ends $\partial T$, and whose local action at each vertex $v$ contains the alternating group on the set $E(v)$ of all edges incident to $v$. As we shall see, it is a consequence of the Classification of the Finite Simple Groups that this condition on the local action is automatic for almost all values of $d_0$ and $d_1$. Typical examples of such groups are the full group $\Aut(T)$, its subgroup $\Aut(T)^+$ preserving the canonical bipartition of $V(T)$ (which is of index $2$ when $d_0 = d_1$) and, when $d_0=d_1$, the universal groups $U(\Alt(d_0))$ and $U(\Alt(d_0))^+$ first defined and studied by Burger and Mozes in~\cite{Burger}. The outcome of our classification is a countably infinite family of locally compact groups, most of which are new, each containing an open abstractly simple subgroup of index at most $8$. Our work provides in particular what seems to be the first instance of a classification of a significant family of compactly generated simple non-discrete locally compact groups beyond the case of linear groups.

The assumption that $T$ is semiregular is not restrictive. Indeed, given a (locally finite) tree whose vertices have valency at least~$3$, the existence of a group acting on it such that the induced action on the set of its ends is $2$-transitive already implies that the tree must be semiregular (see Lemma~\ref{corollary:2-transitive} below).

The closed subgroups of $\Aut(T)$, and especially those whose action on $\partial T$ is $2$-transitive, are known to be subjected to a rich structure theory which parallels to a large extent the theory of semisimple Lie groups, although many of those tree automorphism groups are non-linear. This paradigm received its most spectacular illustration in the groundbreaking work of Burger and Mozes, see \cite{Burger} and \cite{Burger2}. Since then, locally compact groups acting on trees have constantly served as a pool of examples and test cases in the framework of the development of a general structure theory of totally disconnected locally compact groups (see~\cite{Willis} and references therein). This became even more true in recent years, during which the case of compactly generated topologically simple groups received special attention (see~\cite{CRW} and references therein).


\subsection*{Main results}

In order to present a precise formulation of our main results, let us henceforth allow the $(d_0,d_1)$-semiregular tree $T$ to have degrees $d_0,d_1 \geq 4$. Let $V(T) = V_0(T) \sqcup V_1(T)$ be the canonical bipartition of $V(T)$ so that each vertex \textbf{of type $t \in \{0,1\}$} (i.e. in $V_t(T)$) is incident to $d_t$ edges. We write $H \leq_{cl} G$ to mean that $H$ is a \textit{closed subgroup} of $G$ and define the sets
$$\mathcal{H}_T := \{H \leq_{cl} \Aut(T) \suchthat H \text{ is $2$-transitive on $\partial T$}\}$$
and
$$\mathcal{H}_T^+ := \{H \in \mathcal{H}_T \suchthat H \leq \Aut(T)^+\}.$$
Note that when $d_0 \neq d_1$, all automorphisms of $T$ are type-preserving so that $\mathcal{H}_T^+ = \mathcal{H}_T$.

Consider a group $H \in \mathcal{H}_T$. For each vertex $v \in V(T)$, one can look at the action of the stabilizer $H(v)$ of $v$ in $H$ on the set $E(v)$. The image of $H(v)$ in $\Sym(E(v))$ is denoted by $\underline{H}(v)$. Since $H$ is $2$-transitive on $\partial T$, it is transitive on $V_0(T)$ and on $V_1(T)$ (see Lemma~\ref{corollary:2-transitive}). Hence, all the groups $\underline{H}(v)$ with $v \in V_0(T)$ (resp. $v \in V_1(T)$) are permutation isomorphic to the same group $F_0 \leq \Sym(d_0)$ (resp. $F_1 \leq \Sym(d_1)$). In this context of finite permutation groups, we use the symbol $\cong$ to mean \textit{permutation isomorphic}. The goal of this paper is to provide a full classification of the groups $H \in \mathcal{H}_T$ such that $F_0 \geq \Alt(d_0)$ and $F_1 \geq \Alt(d_1)$, under the assumption that $d_0,d_1 \geq 6$.

\begin{figure}[b]
\centering
\begin{pspicture*}(-4.5,-2.9)(4.5,2.3)
\fontsize{10pt}{10pt}\selectfont
\psset{unit=1.3cm}

\psline(0,0)(0,-1)
\psline(0,0)(0.866,0.5)
\psline(0,0)(-0.866,0.5)

\psline(0,-1)(0,-1.7)
\psline(0,-1)(0.606,-1.35)
\psline(0,-1)(-0.606,-1.35)
\psline(0.866,0.5)(1.472,0.85)
\psline(0.866,0.5)(1.472,0.15)
\psline(0.866,0.5)(0.866,1.2)
\psline(-0.866,0.5)(-1.472,0.85)
\psline(-0.866,0.5)(-1.472,0.15)
\psline(-0.866,0.5)(-0.866,1.2)

\psline(0,-1.7)(0.28,-1.98)
\psline(0,-1.7)(-0.28,-1.98)
\psline(0.606,-1.35)(0.986,-1.25)
\psline(0.606,-1.35)(0.706,-1.73)
\psline(-0.606,-1.35)(-0.986,-1.25)
\psline(-0.606,-1.35)(-0.706,-1.73)

\psline(1.472,0.85)(1.852,0.75)
\psline(1.472,0.85)(1.572,1.23)
\psline(0.866,1.2)(1.146,1.48)
\psline(0.866,1.2)(0.586,1.48)
\psline(1.472,0.15)(1.852,0.25)
\psline(1.472,0.15)(1.572,-0.23)

\psline(-1.472,0.85)(-1.852,0.75)
\psline(-1.472,0.85)(-1.572,1.23)
\psline(-0.866,1.2)(-1.146,1.48)
\psline(-0.866,1.2)(-0.586,1.48)
\psline(-1.472,0.15)(-1.852,0.25)
\psline(-1.472,0.15)(-1.572,-0.23)

\psdot[dotstyle=*,dotsize=5pt](0,0)
\rput(0,0.2){$1$}
\rput(0.2,-0.1){$v$}

\psdot[dotstyle=square,dotsize=5pt,fillcolor=black](0,-1)
\psdot[dotstyle=square,dotsize=5pt,fillcolor=black](0.866,0.5)
\psdot[dotstyle=square,dotsize=5pt,fillcolor=black](-0.866,0.5)
\rput(0.15,-0.85){$1$}
\rput(0.816,0.3){$2$}
\rput(-0.816,0.3){$3$}

\psdot[dotstyle=*,dotsize=5pt](0,-1.7)
\psdot[dotstyle=*,dotsize=5pt](0.606,-1.35)
\psdot[dotstyle=*,dotsize=5pt](-0.606,-1.35)
\psdot[dotstyle=*,dotsize=5pt](1.472,0.85)
\psdot[dotstyle=*,dotsize=5pt](1.472,0.15)
\psdot[dotstyle=*,dotsize=5pt](0.866,1.2)
\psdot[dotstyle=*,dotsize=5pt](-1.472,0.85)
\psdot[dotstyle=*,dotsize=5pt](-1.472,0.15)
\psdot[dotstyle=*,dotsize=5pt](-0.866,1.2)
\rput(0.15,-1.55){$3$}
\rput(0.6,-1.15){$2$}
\rput(-0.6,-1.15){$4$}
\rput(1.472,0.65){$3$}
\rput(1.322,0.05){$4$}
\rput(0.716,1.05){$2$}
\rput(-1.472,0.65){$3$}
\rput(-1.322,0.05){$2$}
\rput(-0.716,1.05){$4$}

\psdot[dotstyle=square,dotsize=5pt,fillcolor=black](0.28,-1.98)
\psdot[dotstyle=square,dotsize=5pt,fillcolor=black](-0.28,-1.98)
\psdot[dotstyle=square,dotsize=5pt,fillcolor=black](0.986,-1.25)
\psdot[dotstyle=square,dotsize=5pt,fillcolor=black](0.706,-1.73)
\psdot[dotstyle=square,dotsize=5pt,fillcolor=black](-0.986,-1.25)
\psdot[dotstyle=square,dotsize=5pt,fillcolor=black](-0.706,-1.73)
\rput(0.43,-2.13){$2$}
\rput(-0.43,-2.13){$3$}
\rput(1.17,-1.15){$2$}
\rput(0.85,-1.85){$3$}
\rput(-1.17,-1.15){$3$}
\rput(-0.85,-1.85){$2$}

\psdot[dotstyle=square,dotsize=5pt,fillcolor=black](1.852,0.75)
\psdot[dotstyle=square,dotsize=5pt,fillcolor=black](1.572,1.23)
\psdot[dotstyle=square,dotsize=5pt,fillcolor=black](1.146,1.48)
\psdot[dotstyle=square,dotsize=5pt,fillcolor=black](0.586,1.48)
\psdot[dotstyle=square,dotsize=5pt,fillcolor=black](1.852,0.25)
\psdot[dotstyle=square,dotsize=5pt,fillcolor=black](1.572,-0.23)
\rput(2.002,0.6){$3$}
\rput(1.722,1.38){$1$}
\rput(1.296,1.63){$3$}
\rput(0.436,1.63){$1$}
\rput(1.712,-0.38){$3$}
\rput(2.022,0.15){$1$}

\psdot[dotstyle=square,dotsize=5pt,fillcolor=black](-1.852,0.75)
\psdot[dotstyle=square,dotsize=5pt,fillcolor=black](-1.572,1.23)
\psdot[dotstyle=square,dotsize=5pt,fillcolor=black](-1.146,1.48)
\psdot[dotstyle=square,dotsize=5pt,fillcolor=black](-0.586,1.48)
\psdot[dotstyle=square,dotsize=5pt,fillcolor=black](-1.852,0.25)
\psdot[dotstyle=square,dotsize=5pt,fillcolor=black](-1.572,-0.23)
\rput(-2.002,0.6){$1$}
\rput(-1.722,1.38){$2$}
\rput(-1.296,1.63){$1$}
\rput(-0.436,1.63){$2$}
\rput(-1.712,-0.38){$1$}
\rput(-2.022,0.15){$2$}

\psdot[dotstyle=*,dotsize=5pt](2.5,-1.5)
\psdot[dotstyle=square,dotsize=5pt,fillcolor=black](2.5,-2)
\rput(3,-1.5){$V_0(T)$}
\rput(3,-2){$V_1(T)$}
\end{pspicture*}
\caption{A legal coloring of $B(v,3)$ in the $(3,4)$-semiregular tree.}\label{picture:coloring}
\end{figure}

Let us first describe some key examples of groups in $\mathcal{H}_T^+$. In~\cite{Burger}*{Section~3.2}, the notion of a \textit{legal coloring} of a $d$-regular tree is defined, and consists in coloring the edges of the tree with $d$ colors. For our purposes, we need to generalize this notion to a $(d_0,d_1)$-semiregular tree, and a way to do so is to color the vertices instead of the edges. A \textbf{legal coloring} $i$ of $T$ consists of two maps $i_0 \colon V_0(T) \to \{1, \ldots, d_1\}$ and $i_1 \colon V_1(T) \to \{1, \ldots, d_0\}$ such that $i_0|_{S(v, 1)} \colon S(v,1) \to \{1, \ldots, d_1\}$ is a bijection for each $v \in V_1(T)$ and $i_1|_{S(v,1)} \colon S(v,1) \to \{1, \ldots, d_0\}$ is a bijection for each $v \in V_0(T)$. Here, $S(v,r)$ is the set of vertices of $T$ at distance $r$ from $v$. The map $i$ is defined on $V(T)$ by $i|_{V_0(T)} = i_0$ and $i|_{V_1(T)} = i_1$ (see Figure~\ref{picture:coloring}). Given $g \in \Aut(T)$ and $v \in V(T)$, one can look at the \textbf{local action} of $g$ at the vertex $v$ by defining
$$\sigma_{(i)}(g, v) := i|_{S(g(v), 1)} \circ g \circ i|_{S(v,1)}^{-1} \in \begin{cases}
\Sym(d_0) & \text{if $v \in V_0(T),$}\\
\Sym(d_1) & \text{if $v \in V_1(T).$}
\end{cases}$$
In the particular case where $d_0=d_1$, there is a natural correspondence between our definition of a legal coloring and the definition given in \cite{Burger}. One should however note that, with our definition, the group of all automorphisms $g \in \Aut(T)$ such that $\sigma_{(i)}(g,v) = \id$ for each $v \in V(T)$ is \textit{not} vertex-transitive (and even not transitive on $V_0(T)$), while the universal group $U(\id)$ defined in the same way in \cite{Burger}*{Section~3.2} is vertex-transitive. One must therefore be careful when comparing \cite{Burger} with the present paper. Another definition of legal colorings for semiregular trees was given by Smith in \cite{Smith}*{Section~3}: it is equivalent to ours.

The notion of a legal coloring allows us to define the following groups.



\begin{definition*}
Let $T$ be the $(d_0,d_1)$-semiregular tree and let $i$ be a legal coloring of $T$. When $v \in V(T)$ and $Y$ is a subset of $\N$, we set $S_Y(v) := \bigcup_{r \in Y}S(v,r)$. For all (possibly empty) finite subsets $Y_0$ and $Y_1$ of $\N$, define the group
$$G_{(i)}^+(Y_0, Y_1) := \left\{g \in \Aut(T)^+ \suchthat \begin{array}{c} \prod_{w \in S_{Y_0}(v)} \sgn( \sigma_{(i)}(g,w)) = 1 \text{ for each $v \in V_{t_0}(T)$,} \\ \prod_{w \in S_{Y_1}(v)} \sgn( \sigma_{(i)}(g,w)) = 1 \text{ for each $v \in V_{t_1}(T)$} \end{array}
\right\},$$
where $t_0 := (\max Y_0) \bmod 2$, $t_1 := (1 + \max Y_1) \bmod 2$ and $\max(\varnothing) := 0$.
\end{definition*}

The choice of $t_0$ and $t_1$ in this definition is made in such a way that, in each set $S_{Y_t}(v)$ under consideration, the vertices at maximal distance from $v$ are of type $t$ (i.e. $S(v, \max Y_t) \subseteq V_t(T)$), for $t \in \{0,1\}$.

Remark that $G_{(i)}^+(\varnothing,\varnothing) = \Aut(T)^+$ and that all the groups $G_{(i)}^+(Y_0, Y_1)$ contain the group $G_{(i)}^+(\{0\}, \{0\})$, which we also denote by $\Alt_{(i)}(T)^+$ and satisfies
$$\Alt_{(i)}(T)^+ = \{g \in \Aut(T)^+ \suchthat \sigma_{(i)}(g,v) \text{ is even for each $v \in V(T)$}\}.$$
When $T$ is the $d$-regular tree, i.e. when $d_0=d_1=d$, it can be seen that $\Alt_{(i)}(T)^+$ is conjugate to the universal group $U(\Alt(d))^+$ of Burger-Mozes.

Our first result describes various properties of the groups defined above. We denote by $N_G(H)$ the normalizer of $H$ in $G$ and write $\C_2$ and $\D_8$ for the cyclic group of order~$2$ and the dihedral group of order $8$, respectively.


\begin{maintheorem}\label{maintheorem:simple}
Let $T$ be the $(d_0,d_1)$-semiregular tree with $d_0,d_1 \geq 4$ and let $i$ be a legal coloring of~$T$. Let $Y_0$ and $Y_1$ be finite subsets of $\N$.
\begin{enumerate}[(i)]
\item $G_{(i)}^+(Y_0,Y_1)$ belongs to $\mathcal{H}_T^+$.
\item $G_{(i)}^+(Y_0,Y_1)$ is abstractly simple.
\item We have
$$\bigslant{N_{\Aut(T)^+}(G_{(i)}^+(Y_0,Y_1))}{G_{(i)}^+(Y_0,Y_1)} \cong (\C_2)^k \text{ with $k = |\{t \in \{0,1\} \suchthat Y_t \neq \varnothing\}|$}.$$
If $d_0 = d_1$ and $Y_0 = Y_1 =: Y$ with $Y \neq \varnothing$, then $$\bigslant{N_{\Aut(T)}(G_{(i)}^+(Y,Y))}{G_{(i)}^+(Y,Y)} \cong \D_8.$$
\end{enumerate}
\end{maintheorem}

Using the fact that the pointwise stabilizers of half-trees are non-trivial in these groups $G_{(i)}^+(Y_0,Y_1)$, one can also show that they are not linear over a local field, and even not locally linear (as defined in~\cite{Stulemeijer}).

For any group $H \in \mathcal{H}_T$, Burger and Mozes proved that the subgroup $H^{(\infty)}$ of $H$ defined as the intersection of all normal cocompact closed subgroups of $H$ is such that $H^{(\infty)} \in \mathcal{H}_T^+$ and $H^{(\infty)}$ is topologically simple (see~\cite{Burger}*{Proposition~3.1.2}). Our main classification theorem reads as follows. Note that two groups in $\mathcal{H}_T$ are topologically isomorphic if and only if they are conjugate in $\Aut(T)$ (see Proposition~\ref{proposition:isomorphic} in Appendix~\ref{appendix:isomorphic}), so this is a classification up to topological isomorphism.


\begin{maintheorem}[Classification]\label{maintheorem:classification}
Let $T$ be the $(d_0,d_1)$-semiregular tree with $d_0,d_1 \geq 4$ and let $i$ be a legal coloring of $T$. Let $\underline{\mathcal{S}}_{(i)}$ be the set of groups $G_{(i)}^+(Y_0,Y_1)$ where $Y_0$ and $Y_1$ are finite subsets of $\N$ satisfying the following condition: if $Y_0$ and $Y_1$ are both non-empty, then for each $y \in Y_t$ (with $t \in \{0,1\}$), if $y \geq \max Y_{1-t}$ then $y \equiv \max Y_t \bmod 2$.
\begin{enumerate}[(i)]
\item Two groups $G_{(i)}^+(Y_0,Y_1)$ and $G_{(i)}^+(Y'_0,Y'_1)$ belonging to $\underline{\mathcal{S}}_{(i)}$ are conjugate in $\Aut(T)$ if and only if $(Y_0,Y_1) = (Y'_0,Y'_1)$ or $d_0 = d_1$ and $(Y_0,Y_1) = (Y'_1,Y'_0)$.
\item Suppose that $d_0,d_1 \geq 6$. Let $H \in \mathcal{H}_T^+$ be such that  $\underline{H}(x) \cong F_0 \geq \Alt(d_0)$ for each $x \in V_0(T)$ and $\underline{H}(y) \cong F_1 \geq \Alt(d_1)$ for each $y \in V_1(T)$. Then $[H : H^{(\infty)}] \in \{1,2,4\}$ and $H^{(\infty)}$ is conjugate in $\Aut(T)^+$ to a group belonging to $\underline{\mathcal{S}}_{(i)}$.
\end{enumerate}
\end{maintheorem}

We actually give, in the text, the exact description of all groups $H \in \mathcal{H}_T^+$ satisfying the hypotheses of Theorem~\ref{maintheorem:classification} (ii) (see Theorem~\ref{maintheorem:classification'}). The condition $d_0,d_1 \geq 6$ is used several times in our proof and is actually necessary. Indeed, due to the exceptional isomorphisms $\PSL_2(\F_3) \cong \Alt(4)$ and $\PSL_2(\F_4) = \SL_2(\F_4) \cong \Alt(5)$, the linear groups $\PSL_2(\F_3(\!(X)\!))$ and $\PSL_2(\F_4(\!(X)\!))$, which act on their respective Bruhat-Tits trees $T_4$ and $T_5$ (where $T_d$ is the $d$-regular tree), are elements of $\mathcal{H}_{T_4}^+$ and $\mathcal{H}_{T_5}^+$ respectively whose local action at each vertex is the alternating group.
This shows that Theorem~\ref{maintheorem:classification} (ii) fails when $d_0=d_1 \in \{4,5\}$. In Section~\ref{section:counterexample}, we also give a non-linear counterexample when $d_0 = 4$ and $d_1 \geq 4$.

As a corollary of Theorem~\ref{maintheorem:classification}, we find the corresponding result for $H \in \mathcal{H}_T \setminus \mathcal{H}_T^+$ (when $d_0=d_1$, so that $\mathcal{H}_T^+ \subsetneq \mathcal{H}_T$). In this case, $H$ is automatically transitive on $V(T)$.


\begin{maincorollary}\label{maincorollary:classification}
Let $T$ be the $d$-regular tree with $d \geq 6$ and let $i$ be a legal coloring of~$T$. Let $H \in \mathcal{H}_T \setminus \mathcal{H}_T^+$ be such that $\underline{H}(v) \cong F \geq \Alt(d)$ for each $v \in V(T)$. Then $[H : H^{(\infty)}] \in \{2,4,8\}$ and $H^{(\infty)}$ is conjugate in $\Aut(T)^+$ to $G_{(i)}^+(Y,Y)$ for some finite subset $Y$ of $\N$.
\end{maincorollary}

Here again, a full description of all groups $H \in \mathcal{H}_T \setminus \mathcal{H}_T^+$ satisfying the hypotheses of Corollary~\ref{maincorollary:classification} is given in the text (see Corollary~\ref{maincorollary:classification'}).

When $H \in \mathcal{H}_T$, the fact that $H$ is $2$-transitive on $\partial T$ implies that $\underline{H}(v)$ is a $2$-transitive permutation group for each $v \in V(T)$ (see Lemma~\ref{corollary:2-transitive}). The finite $2$-transitive permutation groups have been classified, using the Classification of the Finite Simple Groups, and the set of integers
$$\Theta := \{m \geq 6 \suchthat \text{each finite $2$-transitive group on $\{1,\ldots, m\}$ contains $\Alt(m)$}\}$$
is known (see Proposition~\ref{proposition:theta} in Appendix~\ref{appendix:theta}). The ten smallest numbers in $\Theta$ are $34$, $35$, $39$, $45$, $46$, $51$, $52$, $55$, $56$ and $58$. Moreover, $\Theta$ is asymptotically dense in $\Nz$ (see Corollary~\ref{corollary:dense}). When $d_0, d_1 \in \Theta$, the hypotheses of Theorem~\ref{maintheorem:classification} (ii) and Corollary~\ref{maincorollary:classification} (if $d_0=d_1$) are always satisfied (by definition) and we get the following result.


\begin{maincorollary}\label{maincorollary:Theta}
Let $T$ be the $(d_0,d_1)$-semiregular tree with $d_0,d_1 \in \Theta$, let $i$ be a legal coloring of $T$ and let $H \in \mathcal{H}_T$. If $H \in \mathcal{H}_T^+$, then $[H : H^{(\infty)}] \in \{1,2,4\}$ and $H^{(\infty)}$ is conjugate in $\Aut(T)^+$ to a group belonging to $\underline{\mathcal{S}}_{(i)}$ (defined in Theorem~\ref{maintheorem:classification}). If $d_0 = d_1$ and $H \not \in \mathcal{H}_T^+$, then $[H : H^{(\infty)}] \in \{2,4,8\}$ and $H^{(\infty)}$ is conjugate in $\Aut(T)^+$ to $G_{(i)}^+(Y,Y)$ for some finite subset $Y$ of $\N$.
\end{maincorollary}

It has also been proven by Burger and Mozes in \cite{Burger}*{Propositions~3.3.1 and~3.3.2} that if $H \leq_{cl} \Aut(T)$ is vertex-transitive and if $\underline{H}(v) \cong F \geq \Alt(d)$ with $d \geq 6$, then $H$ is either discrete or $2$-transitive on $\partial T$. We can therefore combine this result with Corollary~\ref{maincorollary:classification} to obtain the following.


\begin{maincorollary}\label{maincorollary:vertextransitive}
Let $T$ be the $d$-regular tree with $d \geq 6$, let $i$ be a legal coloring of $T$ and let $H$ be a vertex-transitive closed subgroup of $\Aut(T)$. If $\underline{H}(v) \cong F \geq \Alt(d)$ for each $v \in V(T)$, then either $H$ is discrete or $[H : H^{(\infty)}] \in \{2,4,8\}$ and $H^{(\infty)}$ is conjugate in $\Aut(T)^+$ to $G_{(i)}^+(Y,Y)$ for some finite subset $Y$ of $\N$.
\end{maincorollary}

For $d \in \Theta$, the condition $\underline{H}(v) \cong F \geq \Alt(d)$  can be replaced by requiring $\underline{H}(v)$ to be $2$-transitive. Note that the result of Burger and Mozes stated above is not true if we replace vertex-transitivity by edge-transitivity. Indeed, the group
$$H = \left\{g \in \Aut(T)^+ \suchthat \text{$\forall v \in V_0(T)$, $\forall x,y \in S(v,2)$: $i(x)=i(y) \Rightarrow i(g(x))=i(g(y))$}\right\}$$
where $i$ is a legal coloring of $T$ is an example of a closed subgroup of $\Aut(T)$ with $\underline{H}(v) \cong \Sym(d)$ for each $v \in V(T)$ and which is edge-transitive but such that $H$ is non-discrete and the action of $H$ on $\partial T$ is not $2$-transitive.

\begin{remark*}
Vladimir Trofimov pointed out to me that the hypothesis that $T$ is a tree is not even necessary in Corollary~\ref{maincorollary:vertextransitive}. Indeed, Trofimov proved that if $\Gamma$ is a connected $d$-regular graph with $d \geq 6$ and $G \leq \Aut(\Gamma)$ is vertex-transitive with $\underline{G}(v) \geq \Alt(d)$ for each $v \in V(\Gamma)$, then either $G$ is discrete or $\Gamma$ is the $d$-regular tree (see \cite{Trofimov}*{Proposition~3.1}).
\end{remark*}

In order to prove the classification, we first needed to generalize some results of \cite{Burger} to the case of non-vertex-transitive groups. This led us to the following side result, which is an analog of \cite{Burger}*{Proposition~3.3.1}.

\begin{maintheorem}\label{maintheorem:S}
Let $T$ be the $(d_0,d_1)$-semiregular tree and let $F_0 \leq \Sym(d_0)$ and $F_1 \leq \Sym(d_1)$. Let $H \in \mathcal{H}_T^+$ be such that $\underline{H}(x) \cong F_0$ for each $x \in V_0(T)$ and $\underline{H}(y) \cong F_1$ for each $y \in V_1(T)$. Suppose that, for each $t \in \{0,1\}$, the stabilizer $F_t(1)$ of $1$ in $F_t$ is simple non-abelian. Then there exists a legal coloring $i$ of $T$ such that $H$ is equal to the group
$$U_{(i)}^+(F_0,F_1) := \left\{g \in \Aut(T)^+ \suchthat \begin{array}{c}\sigma_{(i)}(g,x) \in F_0 \text{ for each $x \in V_0(T)$,} \\ \sigma_{(i)}(g,y) \in F_1 \text{ for each $y \in V_1(T)$}\end{array} \right\}.$$
\end{maintheorem}
The attentive reader will have noticed that $U_{(i)}^+(\Alt(d_0),\Alt(d_1)) = \Alt_{(i)}(T)^+$.



\subsection*{Structure of the paper}
 
The proof of the classification is divided into different main steps. The first step, which is the subject of Section~\ref{section:localtoglobal} (where Theorem~\ref{maintheorem:S} is also proved) and Section~\ref{section:commonsubgroup}, consists in showing that the groups satisfying the hypotheses of Theorem~\ref{maintheorem:classification} (ii) all contain, up to conjugation, the group $\Alt_{(i)}(T)^+$:


\begin{maintheorem}\label{maintheorem:Alt}
Let $T$ be the $(d_0,d_1)$-semiregular tree with $d_0, d_1 \geq 6$. Let $H \in \mathcal{H}_T^+$ be such that $\underline{H}(x) \cong F_0 \geq \Alt(d_0)$ for each $x \in V_0(T)$ and $\underline{H}(y) \cong F_1 \geq \Alt(d_1)$ for each $y \in V_1(T)$. Then there exists a legal coloring $i$ of $T$ such that $H \supseteq \Alt_{(i)}(T)^+$.
\end{maintheorem}

Note that Theorem~\ref{maintheorem:Alt} is already sufficient to obtain meaningful information on the groups $H$ satisfying the hypotheses. For instance, it follows from Theorem~\ref{maintheorem:Alt} that the pointwise stabilizer of a half-tree in such a group $H$ is never trivial.

For a fixed legal coloring $i$, we then find in Section~\ref{section:classification} all the groups $H \in \mathcal{H}_T^+$ containing $\Alt_{(i)}(T)^+$. The strategy adopted to do so is somewhat involved and what follows is a rough description of it. Recall that, following~\cite{Banks}, the \textbf{$n$-closure} $J^{(n)}$ of an arbitrary group $J \leq \Aut(T)$ is defined by $$J^{(n)} := \{g \in \Aut(T) \suchthat \forall v \in V(T), \exists h \in J : g|_{B(v,n)} = h|_{B(v,n)}\}.$$
The next important step in our proof then reads as follows.


\begin{maintheorem}\label{maintheorem:completeinvariants}
Let $T$ be the $(d_0,d_1)$-semiregular tree with $d_0, d_1 \geq 6$, let $i$ be a legal coloring of $T$ and let $H \in \mathcal{H}_T^+$ be such that $H \supseteq \Alt_{(i)}(T)^+$. Then there exists $K \in \N$ such that $H = H^{(K)}$.
\end{maintheorem}

Theorem~\ref{maintheorem:completeinvariants} is crucial, since it means that $H$ is completely determined by its local action on $T$ on a sufficiently large scale. In particular, observe that for each $K \in \N$ there is only a finite number of groups $H \in \mathcal{H}_T^+$ with $H \supseteq \Alt_{(i)}(T)^+$ and such that $H = H^{(K)}$. This already implies that the classification will lead to a countable family of groups. The idea to complete the classification is finally to fix $K$, to find an upper bound to the number of groups $H$ satisfying the hypotheses and such that $H = H^{(K)}$, and to show that this upper bound is achieved by the various groups from the explicit list described beforehand. These groups are all defined in Section~\ref{section:simple}, where Theorem~\ref{maintheorem:simple} is also proved (see Lemma~\ref{lemma:U(Alt)transitive}, Theorem~\ref{theorem:simple} and Lemma~\ref{lemma:normalizer}).


\subsection*{Acknowledgement}

I warmly thank Pierre-Emmanuel Caprace for suggesting me this problem and for all his valuable advice during my research and the writing of this paper. I also thank Adrien Le Boudec and especially Matthias Grüninger for their helpful comments on a previous version of this manuscript. I am finally grateful to the anonymous referee, who made suggestions to improve the reader's experience.


\section{From local to global structure}\label{section:localtoglobal}

In this section, we consider a group $H \in \mathcal{H}_T^+$ and analyze how the knowledge of $\underline{H}(v)$ for each $v \in V(T)$ has an impact on the global structure of $H$. This section is largely inspired from the work of Burger and Mozes \cite{Burger}. Our goal is to generalize several of their results to the situation where the groups are not vertex-transitive.

Most of our notations come from \cite{Burger}. If $x \in V(T)$, then $S(x,n)$ (resp. $B(x,n)$) is the set of vertices of $T$ at distance exactly $n$ (resp. at most $n$) from $x$. We also set $c(x,n) := |S(x,n)|$. If $n \geq 0$ and $x_1, \ldots, x_k \in V(T)$, then define $H_n(x_1, \ldots, x_k)$ to be the pointwise stabilizer of $\bigcup_{i=1}^k B(x_i,n)$. In the particular case where $n = 0$, we write $H(x_1, \ldots, x_k)$ instead of $H_0(x_1, \ldots, x_k)$ as it is simply the stabilizer of vertices $x_1, \ldots, x_k$. For $x \in V(T)$, set
$$\underline{H}_n(x) := \bigslant{H_n(x)}{H_{n+1}(x)}.$$
Once again, for $n = 0$ we write $\underline{H}(x)$ instead of $\underline{H}_0(x)$ and this exactly corresponds to the definition of $\underline{H}(x)$ given in the introduction.

We start by giving the following results which will be used throughout this paper.


\begin{lemma}\label{lemma:2-transitive}
Let $T$ be a locally finite tree whose vertices have valency at least~$3$ and let $H$ be a closed subgroup of $\Aut(T)$. Then $H$ is $2$-transitive on $\partial T$ if and only if $H(v)$ is transitive on $\partial T$ for each $v \in V(T)$.
\end{lemma}

\begin{proof}
See~\cite{Burger}*{Lemma~3.1.1}.
\end{proof}


\begin{lemma}\label{corollary:2-transitive}
Let $T$ be a locally finite tree whose vertices have valency at least~$3$ and let $H$ be a closed subgroup of $\Aut(T)$ acting $2$-transitively on $\partial T$. Then $T$ is semiregular and, for each $x,x',y,y' \in V(T)$ such that $x$ and $x'$ have the same type and $d(x,y)=d(x',y')$, there exists $h \in H$ such that $h(x) = x'$ and $h(y) = y'$.
\end{lemma}

\begin{proof}
This is a direct consequence of Lemma~\ref{lemma:2-transitive}.
\end{proof}


\subsection{Subgroups of products of finite simple non-abelian groups}\label{subsection:finite}

Lemma~\ref{lemma:subdiagonal} below is a basic result about finite groups and will play a fundamental role in the sequel. Its statement comes from~\cite{Burger}*{Lemma~3.4.3}, but the proof therein requires supplementary details because the definition of a \textit{product of subdiagonals} needs to be amended (probably due to a misnomer). The result could be deduced from Goursat's Lemma (see~\cite{Goursat}*{Sections 11--12} and \cite{Lang}*{Chapter I, Exercise 5}) ; we provide a self-contained proof for the reader's convenience.

Given a product of groups $G_1 \times \cdots \times G_n$ and $i_1, \ldots, i_m \in \{1, \ldots, n\}$, we write $\proj_{i_1, \ldots, i_m} \colon G_1 \times \cdots \times G_n \to G_{i_1} \times \cdots \times G_{i_m}$ for the projection on factors $G_{i_1}, \ldots, G_{i_m}$.


\begin{lemma}\label{lemma:finite1}
Let $S$ be a finite simple non-abelian group and let $G \leq S^n$ (where $n \geq 1$). If $\proj_{i,j}(G) = S^2$ for each $1 \leq i < j \leq n$, then $G = S^n$.
\end{lemma}

\begin{proof}
We prove by induction on $m$ that $\proj_{i_1, \ldots, i_m}(G) = S^m$ for each $1 \leq i_1 < \cdots < i_m \leq n$. By hypothesis this is true for $m = 2$. Now let $m \geq 3$ and suppose it is true for $m-1$. Given $1 \leq i_1 < \cdots < i_m \leq n$, we need to show that $\proj_{i_1, \ldots, i_m}(G) = S^m$. For any $k \in \{1, \ldots, m\}$, we have that $\proj_{i_k}(\ker(\proj_{i_1, \ldots, \widehat{i_k}, \ldots i_m})) \unlhd S$, so it is either trivial or equal to $S$ (since $S$ is simple). In the latter case, since $\proj_{i_1, \ldots, \widehat{i_k}, \ldots i_m}(G) = S^{m-1}$ by induction hypothesis, we directly get that $\proj_{i_1, \ldots, i_m}(G) = S^m$ and we are done. Now we assume that
$\proj_{i_k}(\ker(\proj_{i_1, \ldots, \widehat{i_k}, \ldots i_m}))$ is trivial for all $k$ $(*)$. Putting $k = m$ in $(*)$, we get that there exists $\alpha \colon S^{m-1} \to S$ such that $\proj_{i_m}(g) = \alpha(\proj_{i_1, \ldots, i_{m-1}}(g))$ for all $g \in G$. Moreover, $k = 1$ in $(*)$ implies that the map $\beta \colon S \to S$ defined by $\beta(s) = \alpha(s,1,\ldots,1)$ is injective, and hence surjective. For the same reason with $k = 2$, the map $\gamma \colon S \to S$ defined by $\gamma(s) = \alpha(1,s,1,\ldots,1)$ is surjective. Since $\beta(s) \cdot \gamma(s') = \alpha(s,s',1,\ldots,1) = \gamma(s') \cdot \beta(s)$ for all $s, s' \in S$, we get that $S$ is abelian, a contradiction.
\end{proof}

Given a group $S$ and a positive integer $n$, a \textbf{product of subdiagonals} of $S^n$ is a subgroup of $S^n$ of the form $(\alpha_1 \times \cdots \times \alpha_n)(\Delta_{I_1} \cdots \Delta_{I_r})$, where $\{I_j \suchthat 1 \leq j \leq r\}$ is a partition of $\{1, \ldots, n\}$, $\Delta_J$ is defined by $\Delta_J := \{ (s_1,\ldots,s_n) \in S^n \suchthat s_i = 1 \ \ \forall i \not\in J \text{ and } s_{\ell} = s_k \ \ \forall \ell, k \in J\}$ for each subset $J \subseteq \{1, \ldots, n\}$, and $\alpha_1, \ldots, \alpha_n \in \Aut(S)$. Here, $\alpha_1 \times \cdots \times \alpha_n \in \Aut(S^n)$ is the Cartesian product of $\alpha_1, \ldots, \alpha_n$.


\begin{lemma}\label{lemma:finite2}
Let $S$ be a finite simple non-abelian group and let $G \leq S^n$ (where $n \geq 1$). If $\proj_i(G) = S$ for each $i \in \{1,\ldots,n\}$, then $G$ is a product of subdiagonals of $S^n$.
\end{lemma}

\begin{proof}
For any $i, j \in \{1,\ldots,n\}$, we have $\proj_j(\ker(\proj_i)) \unlhd S$, so it is either trivial or equal to $S$. Let us define the relation $\sim$ on $\{1,\ldots,n\}$ by $i \sim j$ if and only if $\proj_j(\ker(\proj_i))$ is trivial. We claim that $\sim$ is an equivalence relation. Reflexivity and transitivity are clear. Let us prove that it is also symmetric. For $i, j \in \{1,\ldots,n\}$, write $G_{i,j} := \proj_{i,j}(G)$. Then $\proj_i|_{G_{i,j}} \colon G_{i,j} \to S$ has image $S$ by hypothesis, and its kernel is trivial if and only if $i \sim j$. So $|G_{i,j}|\ = |S|$ if and only if $i \sim j$. It follows directly that $\sim$ is symmetric and hence an equivalence relation.

Now let $I_1, \ldots, I_r$ be the equivalence classes of $\sim$ : they form a partition of $\{1,\ldots,n\}$. For each $1 \leq j \leq r$, choose $x_j \in I_j$. For such a $j$ and for $y \in I_j$, we have $x_j \sim y$ and thus $\ker(\proj_{x_j}) = \ker(\proj_y)$. As a consequence, there exists $\alpha_y \in \Aut(S)$ such that $\proj_y(g) = \alpha_y(\proj_{x_j}(g))$ for all $g \in G$. Combined with the fact that $\proj_{x_1, \ldots, x_r}(G) = S^r$ (by Lemma~\ref{lemma:finite1}, because $x_i \not \sim x_j$ implies $\proj_{x_i, x_j}(G) = S^2$), we obtain that $G$ is a product of subdiagonals of $S^n$ whose underlying partition is $\{I_j \suchthat 1 \leq j \leq r\}$.
\end{proof}


\begin{lemma}\label{lemma:subdiagonal}
Let $S \unlhd L$ be finite groups, where $L/S$ is solvable and $S$ is simple non-abelian. Let $G \leq L^n$ (where $n \geq 1$) be such that $\proj_i(G) \geq S$ for all $i \in \{1, \ldots, n\}$. Then $G \cap S^n$ is a product of subdiagonals of $S^n$.
\end{lemma}

\begin{proof}
In view of Lemma~\ref{lemma:finite2}, it suffices to show that $\proj_i(G \cap S^n) = S$ for each $i \in \{1, \ldots, n\}$. Given a group $H$, we write $H^{(0)} = H$ and $H^{(k)} = [H^{(k-1)}, H^{(k-1)}]$ for each $k \geq 1$. Since $L/S$ is solvable, there exists $k$ such that $(L/S)^{(k)}$ is trivial. This implies that $L^{(k)} \leq S$. Hence, we obtain $G^{(k)} \leq (L^n)^{(k)} = (L^{(k)})^n \leq S^n$ and \[\proj_i(G \cap S^n) \geq \proj_i(G^{(k)}) = \proj_i(G)^{(k)} \geq S^{(k)} = S. \qedhere\]
\end{proof}

 
\subsection{Kernel of the action on balls}

We can now start to adapt the results \cite{Burger}*{Lemmas~3.4.2, 3.5.1 and~3.5.3} to the case of groups that are not vertex-transitive. Note that the proofs of some of our results are significantly more complicated because of this missing hypothesis.


\begin{lemma}\label{lemma:trivial}
Let $T$ be the $(d_0,d_1)$-semiregular tree with $d_0,d_1 \geq 3$ and let $H \in \mathcal{H}_T^+$. Let $x$ and $y$ be adjacent vertices of $T$ and let $k \geq 1$. Then $H_{k}(x) \neq H_{k}(y)$. In particular, $\underline{H}_{k-1}(x)$ or $\underline{H}_{k-1}(y)$ is non-trivial.
\end{lemma}

\begin{proof}
Assume for a contradiction that $H_{k}(x) = H_{k}(y)$. Since $H_{k}(x) \unlhd H(x)$ and $H_{k}(y) \unlhd H(y)$, we get $H_{k}(x) \unlhd \langle H(x), H(y) \rangle = H$. As $H$ is transitive on $V_0(T)$ and $V_1(T)$, this means that $H_{k}(x) = H_{k}(x')$ for each $x' \in V(T)$, implying that $H_{k}(x)$ is trivial. This is impossible as $H$ would then be countable, which contradicts its $2$-transitivity on $\partial T$.

In particular, $H_k(x) \setminus H_k(y)$ or $H_k(y) \setminus H_k(x)$ is non-empty. If $H_k(x) \setminus H_k(y) \neq \varnothing$, then there exists $h \in H_k(x) \setminus H_k(y) \subseteq H_{k-1}(y) \setminus H_k(y)$ and hence $\underline{H}_{k-1}(y)$ is non-trivial. If $H_k(y) \setminus H_k(x) \neq \varnothing$ then we get that $\underline{H}_{k-1}(x)$ is non-trivial.
\end{proof}

Recall that the \textbf{socle} of a group $G$ is the subgroup generated by the minimal non-trivial normal subgroups of $G$. In the next results, we will often use the easy fact that if $G$ is a finite group whose socle $S$ is simple and of index at most $2$ in $G$, then $S$ is the only non-trivial proper normal subgroup of $G$. If, moreover, $S$ is non-abelian, then it follows that the center $Z(G)$ of $G$ is trivial.


\begin{lemma}\label{lemma:two_poss}
Let $T$ be the $(d_0,d_1)$-semiregular tree with $d_0,d_1 \geq 3$ and let $F_1 \leq \Sym(d_1)$. Let $H \in \mathcal{H}_T^+$ be such that $\underline{H}(y) \cong F_1$ for each $y \in V_1(T)$. Suppose that the socle $S_1$ of the stabilizer $F_1(1)$ of $1$ in $F_1$ is simple non-abelian and of index $\leq 2$. Then for each $x \in V_0(T)$, one of the following holds.
\begin{enumerate}[(A)]
\item $H_1(x,y) = H_2(x)$ for each $y \in S(x,1)$.
\item $\underline{H}_1(x) \supseteq (S_1)^{d_0}$, where $\underline{H}_1(x)$ is seen in the natural way as a subgroup of $(F_1(1))^{d_0}$.
\end{enumerate}
\end{lemma}

\begin{proof}
Fix $x \in V_0(T)$. For each vertex $y \in S(x,1)$, the inclusion $H_1(x) \subseteq H(x,y)$ induces a homomorphism $\varphi_y \colon H_1(x) \to \bigslant{H(x,y)}{H_1(y)} =: H_{x,y} \cong F_1(1)$ which is such that $\varphi_y(H_1(x)) \unlhd H_{x,y}$. This also gives rise to an injective homomorphism
$$\varphi \colon \underline{H}_1(x) \to \prod_{y \in S(x,1)} H_{x,y} \cong (F_1(1))^{d_0}.$$
As $\varphi_y(H_1(x)) \unlhd H_{x,y}$ and $H_1(x) \unlhd H(x)$, there are only two possibilities: either $\varphi_y(H_1(x))$ is trivial for each $y \in S(x,1)$, or $\varphi_y(H_1(x)) \supseteq S_1$ (via the isomorphism $H_{x,y} \cong F_1(1)$) for each $y \in S(x,1)$. In the first case, we directly get $H_1(x) = H_1(x,y)$ for each $y \in S(x,1)$, which implies $H_1(x) = H_2(x)$ and in particular $H_1(x,y) = H_2(x)$ for each $y \in S(x,1)$. In the second case, by Lemma~\ref{lemma:subdiagonal} the group $\varphi(\underline{H}_1(x)) \cap (S_1)^{d_0}$ is a product of subdiagonals. These subdiagonals determine a bloc decomposition for the $H(x)$-action on $S(x,1)$. As this action is $2$-transitive (by Lemma~\ref{corollary:2-transitive}), there are two options: it is either the full group $(S_1)^{d_0}$ or a full diagonal $(\alpha_1 \times \cdots \times \alpha_{d_0})(\Delta_{\{1,\ldots,d_0\}})$ (with the notation given in Subsection~\ref{subsection:finite}). If it is the full group, then $\varphi(\underline{H}_1(x)) \supseteq (S_1)^{d_0}$ as wanted. Otherwise, $\bigslant{H_1(x,y)}{H_2(x)}$ is a $2$-group for each $y \in S(x,1)$. In particular, if $z \in S(x,1)$ with $z \neq y$ then the image $I$ of $H_1(x,y)$ in $H_{x,z} = \bigslant{H(x,z)}{H_1(z)} \cong F_1(1)$ is a subnormal $2$-group of $H_{x,z}$ (because $H_1(x,y) \unlhd H_1(x) \unlhd H(x,z)$). Since $S_1$ is not a $2$-group, the only possibility for $I$ is to be trivial. We thus have $H_1(x,y) \subseteq H_1(z)$ for each $z \in S(x,1)$, which means that $H_1(x,y) = H_2(x)$.
\end{proof}


\begin{lemma}\label{lemma:HkHk+1}
Let $T$ be the $(d_0,d_1)$-semiregular tree with $d_0,d_1 \geq 3$ and let $F_0 \leq \Sym(d_0)$ and $F_1 \leq \Sym(d_1)$. Let $H \in \mathcal{H}_T^+$ be such that $\underline{H}(x) \cong F_0$ for each $x \in V_0(T)$ and $\underline{H}(y) \cong F_1$ for each $y \in V_1(T)$. Suppose that, for each $t \in \{0,1\}$, the socle $S_t$ of $F_t(1)$ is simple non-abelian and of index $\leq 2$. Fix two adjacent vertices $x \in V_0(T)$ and $y \in V_1(T)$ and let $k \geq 1$. Assume that $\underline{H}_k(x) \supseteq (S_{k \bmod 2})^{c(x,k)}$ and, if $k \neq 1$, that $\underline{H}_{k-1}(x) \supseteq (S_{(k-1) \bmod 2})^{c(x,k-1)}$. Then $\underline{H}_k(y)$ is non-trivial.
\end{lemma}

\begin{proof}
For $z \in S(x,n)$, let $p(z)$ be the vertex at distance $n-1$ from $x$ which is adjacent to $z$ and $H_{x,z} := \bigslant{H(z, p(z))}{H_1(z)}$. Define also $S_n(x,y)$ to be the set of vertices of $S(x,n)$ that are at distance $n-1$ from $y$ and $a(x,n) := |S_n(x,y)|$.

For simplicity, we set $s := k \bmod 2$ and $t := (k-1) \bmod 2$. We first claim that there exists $g \in H_{k-1}(x,y) \setminus H_k(x)$ whose image $\sigma(g)$ in $\prod_{z \in S_k(x,y)}H_{x,z} \cong (F_s(1))^{a(x,k)}$ is contained in $(S_s)^{a(x,k)}$. First remark that $H_{k-1}(x,y) \setminus H_k(x)$ is non-empty in view of the hypothesis $\underline{H}_{k-1}(x) \supseteq (S_t)^{c(x,k-1)}$ (if $k=1$, use $\bigslant{H(x,y)}{H_1(x)} \cong F_t(1) \supseteq S_t$). Hence, if $F_s(1) = S_s$ the claim is trivially true. On the other hand, if $[F_s(1) : S_s] = 2$ then take $h \in H_{k-1}(x,y)$ such that $h^2 \in H_{k-1}(x,y) \setminus H_k(x)$. Such an element exists as $S_t$ is not a $2$-group. Then $g = h^2$ satisfies the claim.

Now take $g' \in H_k(x)$ such that $\sigma(g') = \sigma(g)$, whose existence is ensured by the fact that $\underline{H}_k(x) \supseteq (S_s)^{c(x,k)}$. Then the element $g'g^{-1}$ is contained in $H_k(y)$ but not in $H_{k+1}(y)$ (by construction), so $\underline{H}_k(y)$ is non-trivial.
\end{proof}

In the proof of the following lemma, we use the Schreier conjecture stating that $\Out(S)$ is solvable for each finite simple group $S$. This conjecture has been proven using the Classification of the Finite Simple Groups. Note however that, except for Theorem~\ref{maintheorem:S}, we will only use Lemma~\ref{lemma:S^d} with $S_0 = \Alt(d_0)$ and $S_1 = \Alt(d_1)$, in which case the solvability of $\Out(S_0)$ and $\Out(S_1)$ is clear.


\begin{lemma}\label{lemma:S^d}
Let $T$ be the $(d_0,d_1)$-semiregular tree with $d_0,d_1 \geq 3$ and let $F_0 \leq \Sym(d_0)$ and $F_1 \leq \Sym(d_1)$. Let $H \in \mathcal{H}_T^+$ be such that $\underline{H}(x) \cong F_0$ for each $x \in V_0(T)$ and $\underline{H}(y) \cong F_1$ for each $y \in V_1(T)$. Suppose that, for each $t \in \{0,1\}$, the socle $S_t$ of $F_t(1)$ is simple non-abelian, of index $\leq 2$ and transitive but not simply transitive on $\{2,\ldots,d_t\}$. Then $\underline{H}_1(x) \supseteq (S_1)^{d_0}$ for each $x \in V_0(T)$ and $\underline{H}_1(y) \supseteq (S_0)^{d_1}$ for each $y \in V_1(T)$.
\end{lemma}

\begin{proof}
For each $x \in V(T)$, we can apply Lemma~\ref{lemma:two_poss}. This gives two possibilities ((A) or (B)) at each vertex of $T$. As $H$ is transitive on $V_0(T)$ and $V_1(T)$, the situation must be identical at all vertices of the same type. In total, there are three possible situations: (A) for all vertices, (A) for one type of vertices and (B) for the other, or (B) for all vertices. To prove the statement, we must show that the only situation that really occurs is the last one. To do so, we prove that the two other situations are impossible.

We already know that we cannot have (A) for all vertices, since it would imply that $H_2(x) = H_1(x,y) = H_2(y)$ for two adjacent vertices $x$ and $y$, contradicting Lemma~\ref{lemma:trivial}.

Now assume for a contradiction that we have (A) for $V_0(T)$ and (B) for $V_1(T)$ (the reverse situation being identical). If $x \in V_0(T)$ and $y \in S(x,1)$, then (A) means that $H_1(x,y) = H_2(x)$. The homomorphism $\varphi_y \colon H_1(x) \to \bigslant{H(x,y)}{H_1(y)} \cong F_1(1)$ has a normal image and its kernel is exactly $H_1(x,y) = H_2(x)$. Hence, $\underline{H}_1(x) = \bigslant{H_1(x)}{H_2(x)}$ is isomorphic to a normal subgroup of $F_1(1)$: it is either trivial or isomorphic to $S_1$ or $F_1(1)$. By Lemma~\ref{lemma:HkHk+1} (with $k=1$), since $\underline{H}_1(y) \supseteq (S_0)^{d_1}$, $\underline{H}_1(x)$ cannot be trivial.

For the sake of brevity, set $\tilde{H} := \underline{H}_1(x)$ and $G := \bigslant{H(x)}{H_2(x)}$. We have shown that $\tilde{H}$ is isomorphic to $S_1$ or $F_1(1)$, which implies that the center $Z(\tilde{H})$ of $\tilde{H}$ is trivial, and $\tilde{H}$ is a normal subgroup of $G$. Hence, $G$ contains the direct product of $\tilde{H}$ and its centralizer $C_G(\tilde{H})$ (the intersection of these two normal subgroups being $Z(\tilde{H})$).

\begin{claim*}
The product $\tilde{H} \cdot C_G(\tilde{H})$ is a subgroup of index at most $2$ of $G$.
\end{claim*}

\begin{claimproof}
Consider the homomorphism
$$\alpha \colon G \to \Out(\tilde{H}) \colon g \mapsto [h \in \tilde{H} \mapsto ghg^{-1} \in \tilde{H}].$$
An element $g \in G$ is in the kernel of $\alpha$ if and only if there exists $k \in \tilde{H}$ such that $ghg^{-1} = khk^{-1}$ for all $h \in \tilde{H}$, which is equivalent to saying that $k^{-1}g \in C_G(\tilde{H})$. Hence, $\ker(\alpha) = \tilde{H} \cdot C_G(\tilde{H})$. We write $K := \tilde{H} \cdot C_G(\tilde{H})$ and want to show that $[G:K]\leq 2$. Since $K = \ker(\alpha)$, the quotient $\bigslant{G}{K}$ can be embedded into $\Out(\tilde{H})$. By the Schreier conjecture (see~\cite{Dixon}*{Appendix~A}, $\Out(S_1)$ is solvable. As $\tilde{H} \cong S_1$ or $F_1(1)$, it implies that $\Out(\tilde{H})$ is solvable. Indeed, if $[F_1(1) : S_1] = 2$ then there is a natural map $j \colon \Aut(F_1(1)) \to \Out(S_1)$, and one can show that $\ker(j) \subseteq \Inn(F_1(1))$, so that $\bigslant{\Aut(F_1(1))}{\ker(j)} \cong \im(j) \leq \Out(S_1)$ surjects onto $\Out(F_1(1))$, making it solvable.

We just proved that $\bigslant{G}{K}$ is solvable. By the third isomorphism theorem, we have
$$\bigslant{\left(\bigslant{G}{\tilde{H}}\right)}{\left(\bigslant{K}{\tilde{H}}\right)} \cong \bigslant{G}{K}.$$
Since $\bigslant{G}{\tilde{H}} \cong F_0$, this means that $\bigslant{G}{K}$ is isomorphic to a quotient of $F_0$, let us say $\bigslant{F_0}{N}$ with $N \unlhd F_0$. There remains to show that $[F_0 : N] \leq 2$, using the fact that $\bigslant{F_0}{N}$ is solvable. Consider the injective map $i \colon \bigslant{F_0(1)}{N(1)} \hookrightarrow \bigslant{F_0}{N}$ where $N(1)$ is the stabilizer of $1$ in $N$. Since $\bigslant{F_0}{N}$ is solvable, $\bigslant{F_0(1)}{N(1)}$ is also solvable. However, $N(1)$ can only be trivial or equal to $F_0(1)$ or $S_0$. It cannot be trivial as $F_0(1)$ is not solvable, so $\left\lvert\bigslant{F_0(1)}{N(1)}\right\rvert \leq 2$. In particular, $N$ is a non-trivial normal subgroup of the $2$-transitive group $F_0$, which implies that $N$ is transitive. Hence, the map $i$ defined above is an isomorphism, and $\left\lvert\bigslant{F_0}{N}\right\rvert = \left\lvert\bigslant{F_0(1)}{N(1)}\right\rvert \leq 2$ as wanted.
\end{claimproof}

\medskip

Using the fact that $\tilde{H} \cdot C_G(\tilde{H})$ is a subgroup of index $1$ or $2$ of $G$, one can find a contradiction. Denote by $v_1, \ldots, v_{d_0}$ the vertices adjacent to $x$ and by $a^{(1)}_1, \ldots, a^{(1)}_{d_1-1}$ the vertices adjacent to $v_1$ different from $x$ (see Figure~\ref{picture:S^d}). As a corollary of the claim, the group $C_G(\tilde{H})$ acts non-trivially and therefore transitively on $S(x,1) = \{v_1, \ldots, v_{d_0}\}$. Hence, there exist $c_2, \ldots, c_{d_0} \in C_G(\tilde{H})$ such that $c_k(v_1) = v_k$ for each $k \in \{2, \ldots, d_0\}$. Define $a^{(k)}_i = c_k(a^{(1)}_i)$ for each $k \in \{2, \ldots, d_0\}$ and $i \in \{1, \ldots, d_1-1\}$. In this way, for each $k$ the vertices $a^{(k)}_1, \ldots, a^{(k)}_{d_1-1}$ are the vertices adjacent to $v_k$ different from $x$. Thanks to this choice, if $h \in \tilde{H}$ satisfies $h(a^{(1)}_i) = a^{(1)}_j$ for some $i$ and $j$ then the fact that $h c_k = c_k h$ directly implies that $h(a^{(k)}_i) = a^{(k)}_j$ for each $k \in \{2, \ldots, d_0\}$. In other words, as soon as the action of $h \in \tilde{H}$ on the vertices adjacent to $v_1$ is known, its action on the vertices adjacent to $v_k$ is also known for each $k \in \{2, \ldots, d_0\}$.

\begin{figure}
\centering
\begin{pspicture*}(-4,-3.05)(4,0.3)
\fontsize{10pt}{10pt}\selectfont
\psset{unit=1.2cm}

\rput(0,0.17){$x$}

\psline(0,0)(-2.4,-1) \rput(-2.55,-0.85){$v_1$}
\psline(0,0)(-0.8,-1) \rput(-0.95,-0.85){$v_2$}
\psline(0,0)(0.8,-1)
\psline(0,0)(2.4,-1)\rput(2.6,-0.9){$v_{d_0}$}

\psline(-2.4,-1)(-3,-2) \rput(-3,-2.26){$a_1^{(1)}$}
\psline(-2.4,-1)(-2.6,-2)
\psline(-2.4,-1)(-2.2,-2)
\psline(-2.4,-1)(-1.8,-2) \rput(-1.75,-2.26){$a_{d_1-1}^{(1)}$}
\psline(-0.8,-1)(-1.4,-2)
\psline(-0.8,-1)(-1,-2)
\psline(-0.8,-1)(-0.6,-2)
\psline(-0.8,-1)(-0.2,-2)
\psline(2.4,-1)(3,-2) \rput(3,-2.26){$a_{d_1-1}^{(d_0)}$}
\psline(2.4,-1)(2.6,-2)
\psline(2.4,-1)(2.2,-2)
\psline(2.4,-1)(1.8,-2) \rput(1.85,-2.26){$a_1^{(d_0)}$}
\psline(0.8,-1)(1.4,-2) 
\psline(0.8,-1)(1,-2)
\psline(0.8,-1)(0.6,-2)
\psline(0.8,-1)(0.2,-2)

\end{pspicture*}
\caption{Illustration of Lemma~\ref{lemma:S^d}.}\label{picture:S^d}
\end{figure}
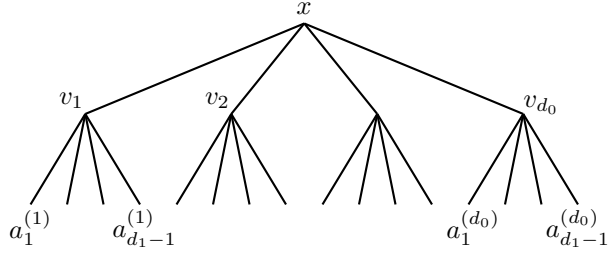

Now consider $c \in C_G(\tilde{H})$ with $c(v_k) = v_{\ell}$ (for some $k, \ell \in \{1,\ldots,d_0\}$). If we write $c(a^{(k)}_i) = a^{(\ell)}_{\sigma(i)}$ (for all $i$) with $\sigma \in \Sym(d_1-1)$, then the fact that $c$ centralizes $\tilde{H}$ implies that $\sigma$ centralizes $S_1$. Denote by $O_1, \ldots, O_r$ the distinct orbits of $C_{\Sym(d_1-1)}(S_1)$, forming a partition of $\{1,\ldots, d_1-1\}$. Since $S_1$ is transitive on $\{1,\ldots, d_1-1\}$, we directly get that $|O_1|\ = \cdots = |O_r|$ and that $S_1$ preserves the partition $O_1 \sqcup \cdots \sqcup O_r$. If $r = 1$, then $C_{\Sym(d_1-1)}(S_1)$ is transitive and hence $S_1$ is simply transitive, which is impossible by hypothesis. If $r = 2$, then $\{s \in S_1 \suchthat s(O_1) = O_1\}$ is a subgroup of index $2$ of $S_1$, which contradicts its simplicity. Hence, we must have $r \geq 3$.

We now explain how this contradicts the $2$-transitivity of $H$. Let us look at the possible images of the ordered pair $(a^{(1)}_1, a^{(2)}_1)$ by elements of $G$. In view of Lemma~\ref{corollary:2-transitive}, for all distinct $k, \ell \in \{1, \ldots, d_0\}$ and all $i,j \in \{1, \ldots, d_1-1\}$ there should exist some element $g \in G$ such that $g((a^{(1)}_1, a^{(2)}_1)) = (a^{(k)}_i, a^{(\ell)}_j)$. This means that $\left\lvert G \cdot (a^{(1)}_1, a^{(2)}_1)\right\rvert = d_0(d_0-1)(d_1-1)^2$. However, in view of what has been observed above, the image of $(a^{(1)}_1, a^{(2)}_1)$ by an element of $\tilde{H}$ is always of the form $(a^{(1)}_i, a^{(2)}_i)$, and the image of $(a^{(1)}_i, a^{(2)}_i)$ by an element of $C_G(\tilde{H})$ is always of the form $(a^{(k)}_j, a^{(\ell)}_{j'})$ with $j$ and $j'$ in the orbit $O_a \ni i$. Consequently, we have $\lvert (\tilde{H} \cdot C_G(\tilde{H})) \cdot (a^{(1)}_1, a^{(2)}_1)\rvert \leq d_0(d_0-1) r \left(\frac{d_1-1}{r}\right)^2$ (because there are $r$ orbits, each of size $\frac{d_1-1}{r}$). Since $[G : (\tilde{H} \cdot C_G(\tilde{H}))] \leq 2$, this implies that $\left\lvert G \cdot (a^{(1)}_1, a^{(2)}_1)\right\rvert \leq \frac{2}{r}\cdot d_0(d_0-1)(d_1-1)^2$, which contradicts the fact that $r \geq 3$.
\end{proof}


\begin{proposition}\label{proposition:S^c(k)}
Under the assumptions of Lemma~\ref{lemma:S^d} and for each $x \in V(T)$ and each $k \in \Nz$, we have
$$\underline{H}_k(x) \supseteq (S_{(t+k)\bmod 2})^{c(x,k)},$$
where $t \in \{0,1\}$ is the type of $x$.
\end{proposition}

\begin{proof}
For $x \in V(T)$ and $z \in S(x,n)$, set $H_{x,z} := \bigslant{H(z, p(z))}{H_1(z)}$ where $p(z)$ is the vertex at distance $n-1$ from $x$ which is adjacent to $z$. For $y \in S(x,1)$, let also $S_n(x,y)$ be the set of vertices of $S(x,n)$ that are at distance $n-1$ from $y$ and $a(x,n) := |S_n(x,y)|$.

We prove the result by induction on $k$. For $k = 1$, this is exactly Lemma~\ref{lemma:S^d}. Now let $k \geq 2$ and assume the result is proven for $k-1$ (and for all vertices). We show that it is therefore also true for $k$. By Lemma~\ref{lemma:trivial} and since $H$ is transitive on $V_0(T)$ and $V_1(T)$, $\underline{H}_k(x)$ is non-trivial for each $x \in V_0(T)$ or $\underline{H}_k(y)$ is non-trivial for each $y \in V_1(T)$. Assume without loss of generality that $\underline{H}_k(x)$ is non-trivial for each $x \in V_0(T)$. We first prove that $\underline{H}_k(x) \supseteq (S_s)^{c(x,k)}$ for each $x \in V_0(T)$, where $s := k \bmod 2$.

Fix $x \in V_0(T)$. For each $y \in S(x,1)$, let $I_y$ be the image of $H_{k-1}(y)$ in the product $\prod_{z \in S_k(x,y)}H_{x,z}$. By the induction hypothesis, we have $I_y \supseteq (S_s)^{a(x,k)}$. But $H_k(x) \unlhd H_{k-1}(y)$, so if $I'_y$ is the image of $H_k(x)$ in this product, then $I'_y \unlhd I_y$ and $I'_y \cap (S_s)^{a(x,k)} \unlhd (S_s)^{a(x,k)}$. The only normal subgroups of $(S_s)^{a(x,k)}$ are the products made from the trivial group and the full group $S_s$. By transitivity of $H(x)$ on $S(x,k)$ (see Lemma~\ref{lemma:2-transitive}), $I'_y \cap (S_s)^{a(x,k)}$ must be either trivial or equal to $(S_s)^{a(x,k)}$. Suppose that $I'_y \cap (S_s)^{a(x,k)}$ is trivial. Then $I'_y$ is trivial, since the contrary and the fact that $I'_y \unlhd I_y$ would imply that $F_s(1)$ has a normal subgroup of order~$2$, which is not the case. Then, by transitivity of $H(x)$ on $S(x,1)$, $I'_y$ must be trivial for each $y \in S(x,1)$. This is impossible as $\underline{H}_k(x)$ is non-trivial. Hence, $I'_y$ contains $(S_s)^{a(x,k)}$. 

Now $\underline{H}_k(x)$ is exactly the image of $H_k(x)$ in $\prod_{y \in S(x,1)} \prod_{z \in S_k(x,y)} H_{x,z}$, so $\underline{H}_k(x) \cap (S_s)^{c(x,k)}$ is a product of subdiagonals in $(S_s)^{c(x,k)}$, by Lemma~\ref{lemma:subdiagonal}. We claim that it must be the full group $(S_s)^{c(x,k)}$. By contradiction, suppose it is not the case. Then the product of subdiagonals induces a bloc decomposition $\{B_i\}_{1 \leq i \leq r}$ for the $H(x)$-action on $S(x,k)$ with $\lvert B_{i_0}\rvert \geq 2$ for some $i_0$ and $\lvert B_i \cap S_k(x,y)\rvert \leq 1$ for all $i$ and $y \in S(x,1)$ (because $I'_y \supseteq (S_s)^{a(x,k)}$). Choose $y \neq y'$ in $S(x,1)$ such that $B_{i_0} \cap S_k(x,y) = \{z\}$ and $B_{i_0} \cap S_k(x,y') = \{z'\}$. Take $w \in S_k(x,y')$ with $w \neq z'$. By Lemma~\ref{corollary:2-transitive}, there exists $g \in H(x)$ such that $g(z) = z$ and $g(z') = w$, which is a contradiction with the bloc decomposition. Therefore, we have $\underline{H}_k(x) \supseteq (S_s)^{c(x,k)}$ as wanted.

We are done for each $x \in V_0(T)$. Now if we try to do the same reasoning for $y \in V_1(T)$, the only issue is that $\underline{H}_k(y)$ could a priori be trivial. However, since $\underline{H}_k(x) \supseteq (S_s)^{c(x,k)}$ for each $x \in V_0(T)$ and as $\underline{H}_{k-1}(x) \supseteq (S_{1-s})^{c(x,k-1)}$ by induction hypothesis, Lemma~\ref{lemma:HkHk+1} precisely tells us that $\underline{H}_k(y)$ is non-trivial. Hence, we also get $\underline{H}_k(y) \supseteq (S_{1-s})^{c(y,k)}$ in the same way.
\end{proof}


\subsection{A global result}

In the particular case where $F_0(1)$ and $F_1(1)$ are simple non-abelian, we can deduce from Proposition~\ref{proposition:S^c(k)} that there is, up to conjugation, only one group $H \in \mathcal{H}_T^+$ such that $\underline{H}(x) \cong F_0$ for each $x \in V_0(T)$ and $\underline{H}(y) \cong F_1$ for each $y \in V_1(T)$. This is the subject of Theorem~\ref{maintheorem:S} whose statement is recalled below.

Recall that a legal coloring $i$ of $T$ is a map defined piecewise by $i|_{V_0(T)} = i_0$ and $i|_{V_1(T)} = i_1$ where, for each $t \in \{0,1\}$, the map $i_t \colon V_t(T) \to \{1,\ldots,d_{1-t}\}$ is such that $i_t|_{S(v,1)} \colon S(v,1) \to \{1,\ldots, d_{1-t}\}$ is a bijection for each $v \in V_{1-t}(T)$. For $g \in \Aut(T)$ and $v \in V(T)$, the local action of $g$ at $v$ is $\sigma_{(i)}(g,v) := i|_{S(g(v),1)} \circ g \circ i|_{S(v,1)}^{-1}$. Given $F_0 \leq \Sym(d_0)$ and $F_1 \leq \Sym(d_1)$, the group $U_{(i)}^+(F_0,F_1)$ is defined by
$$U_{(i)}^+(F_0,F_1) := \left\{g \in \Aut(T)^+ \suchthat \begin{array}{c}\sigma_{(i)}(g,x) \in F_0 \text{ for each $x \in V_0(T)$,} \\ \sigma_{(i)}(g,y) \in F_1 \text{ for each $y \in V_1(T)$}\end{array} \right\}.$$

The following basic result will be used constantly in the rest of this paper.

\begin{lemma}\label{lemma:sigma}
Let $T$ be the $(d_0,d_1)$-semiregular tree and let $i$ be a legal coloring of $T$.
\begin{itemize}
\item If $g, h \in \Aut(T)$ and $v \in V(T)$, then $\sigma_{(i)}(gh, v) = \sigma_{(i)}(g, h(v)) \circ \sigma_{(i)}(h,v)$.
\item If $g \in \Aut(T)$ and $v \in V(T)$, then $\sigma_{(i)}(g^{-1}, v) = \sigma_{(i)}(g, g^{-1}(v))^{-1}$.
\end{itemize}
\end{lemma}

\begin{proof}
This directly follows from the definition of $\sigma_{(i)}(g,v)$.
\end{proof}

 The next result is the edge-transitive version of~\cite{Burger}*{Proposition~3.2.2}.

\begin{lemma}\label{lemma:S}
Let $T$ be the $(d_0,d_1)$-semiregular tree with $d_0,d_1 \geq 3$ and let $F_0 \leq \Sym(d_0)$ and $F_1 \leq \Sym(d_1)$. Let $H \in \mathcal{H}_T^+$ be such that $\underline{H}(x) \cong F_0$ for each $x \in V_0(T)$ and $\underline{H}(y) \cong F_1$ for each $y \in V_1(T)$. Then there exists a legal coloring $i$ of $T$ such that $H \subseteq U_{(i)}^+(F_0,F_1)$.
\end{lemma}

\begin{proof}
Fix $x \in V_0(T)$ and, for each $v \in V_0(T)$ different from $x$, let $p(v)$ be the vertex of $S(v,2)$ the closest to $x$. For each such $v$, choose $h_v \in H$ such that $h$ interchanges $v$ and $p(v)$. We define an appropriate map $i_1 \colon V_1(T) \to \{1,\ldots, d_0\}$ inductively on $X_n := V_1(T) \cap B(x,2n-1)$. For $n = 1$, we choose a bijection $i_x \colon X_1 = S(x,1) \to \{1,\ldots,d_0\}$ such that $i_x \underline{H}(x) i_x^{-1} = F_0$ and set $i_1|_{X_1} = i_x$. Now assume that $i_1$ is defined on $X_n$. To extend $i_1$ to $X_{n+1}$, we set $i_1|_{S(v,1)} = i_1|_{S(p(v),1)} h_v|_{S(v,1)}$ for each $v \in S(x,2n)$. The map $i_0 \colon V_0(T) \to \{1,\ldots, d_1\}$ is defined in the same way by fixing $y \in V_1(T)$ and choosing $h_v \in H$ for each $v \in V_1(T)$ as above. Define finally $i$ by $i|_{V_0(T)} = i_0$ and $i|_{V_1(T)} = i_1$.

Given $v \in V_0(T)$ different from $x$, our construction is such that $\sigma_{(i)}(h_v,v) = \id$. Hence, if $v$ is at distance $2n$ from $x$, the element $\tilde{h}_v = h_{p^{n-1}(v)} \cdots h_{p(v)} h_v \in H$ satisfies $\tilde{h}_v(v) = x$ and $\sigma_{(i)}(\tilde{h}_v,v) = \id$ (by Lemma~\ref{lemma:sigma}). Now if we consider $g \in H$ and $v \in V_0(T)$, the element $\tilde{h}_{g(v)} g \tilde{h}_v^{-1} \in H$ fixes $x$ and is therefore such that $\sigma_{(i)}(\tilde{h}_{g(v)} g \tilde{h}_v^{-1}, x) \in F_0$. Using Lemma~\ref{lemma:sigma}, we obtain that $\sigma_{(i)}(g,v) \in F_0$. In the same way, for $v \in V_1(T)$ we get $\sigma_{(i)}(g,v) \in F_1$. We thus have $g \in U_{(i)}^+(F_0,F_1)$ and hence $H \subseteq U_{(i)}^+(F_0,F_1)$.
\end{proof}

Let us now prove Theorem~\ref{maintheorem:S}. Note that the fact that $F_t(1)$ is simple non-abelian implies that $\lvert F_t(1)\rvert \geq 60$ and hence that $d_t \geq 6$ for each $t \in \{0,1\}$.

\begin{repeattheorem}{maintheorem:S}
Let $T$ be the $(d_0,d_1)$-semiregular tree and let $F_0 \leq \Sym(d_0)$ and $F_1 \leq \Sym(d_1)$. Let $H \in \mathcal{H}_T^+$ be such that $\underline{H}(x) \cong F_0$ for each $x \in V_0(T)$ and $\underline{H}(y) \cong F_1$ for each $y \in V_1(T)$. Suppose that, for each $t \in \{0,1\}$, $F_t(1)$ is simple non-abelian. Then there exists a legal coloring $i$ of $T$ such that $H = U_{(i)}^+(F_0,F_1)$.
\end{repeattheorem}

\begin{proof}
By Lemma~\ref{lemma:S}, there exists a legal coloring $i$ of $T$ such that $H \subseteq U_{(i)}^+(F_0,F_1)$. For each $t \in \{0,1\}$, $F_t$ is $2$-transitive and hence $F_t(1)$ is transitive on $\{2,\ldots,d_t\}$. Moreover, $F_t(1)$ is never simply transitive. Indeed, if it was the case then $F_t$ would be sharply $2$-transitive, but the finite sharply $2$-transitive permutation groups have been classified and they never have a simple non-abelian point stabilizer (see \cite{Zassenhaus}, \cite{Dixon}*{Section~7.6}). We can therefore apply Proposition~\ref{proposition:S^c(k)} and directly obtain, since $H$ is closed in $\Aut(T)$, that for each $v \in V(T)$ the stabilizer $H(v)$ is equal to $U_{(i)}^+(F_0,F_1)(v)$. As $H$ is generated by its vertex stabilizers, the conclusion follows.
\end{proof}


\section{A common subgroup}\label{section:commonsubgroup}

We assume in this section that $H \in \mathcal{H}_T^+$ satisfies $\underline{H}(x) \cong F_0 \geq \Alt(d_0)$ for each $x \in V_0(T)$ and $\underline{H}(y) \cong F_1 \geq \Alt(d_1)$ for each $y \in V_1(T)$. Our goal is to prove, under this hypothesis and when $d_0, d_1 \geq 6$, that there always exists a legal coloring $i$ of $T$ such that $H \supseteq \Alt_{(i)}(T)^+$. Recall that $\Alt_{(i)}(T)^+ = U_{(i)}^+(\Alt(d_0),\Alt(d_1))$, i.e.
$$\Alt_{(i)}(T)^+ = \{g \in \Aut(T)^+ \suchthat \sigma_{(i)}(g,v) \text{ is even for each $v \in V(T)$}\}.$$
Under these assumptions, we will apply Proposition~\ref{proposition:S^c(k)}. Indeed, when $F_t \supseteq \Alt(d_t)$ with $d_t \geq 6$ (for $t \in \{0,1\}$), the socle $S_t$ of $F_t(1)$ is $\Alt(d_t-1)$ which is simple non-abelian, of index at most $2$ in $F_t(1)$, and transitive but not simply transitive on $\{2,\ldots,d_t\}$.

Remark that, if $F_0 = \Alt(d_0)$ and $F_1 = \Alt(d_1)$, then we already know by Theorem~\ref{maintheorem:S} that $H = \Alt_{(i)}(T)^+$ for some legal coloring $i$. The task is however surprisingly more difficult when $F_0 = \Sym(d_0)$ or $F_1 = \Sym(d_1)$.


\subsection{Finding good colorings of rooted trees}

For our next results, we denote by $T_{d_0,d_1,n}$ the rooted tree of depth $n$ where the root $v_0$ has $d_0$ children, the vertices at positive even distance from $v_0$ (except the leaves) have $d_0-1$ children, and the vertices at odd distance from $v_0$ (except the leaves) have $d_1-1$ children. Similarly, $T'_{d_0,d_1,n}$ is the rooted tree of depth $n$ where $v_0$ and all the vertices at even distance from $v_0$ have $d_0-1$ children while the vertices at odd distance from $v_0$ have $d_1-1$ children.
Remark that, in the $(d_0,d_1)$-semiregular tree $T$, a ball $B(v,n)$ around a vertex $v$ of type $0$ is isomorphic to $T_{d_0,d_1,n}$. The intersection of $B(v,n)$ with a half-tree of $T$ rooted in $v$ is isomorphic to $T'_{d_0,d_1,n}$.

The notion of a legal coloring of $T_{d_0,d_1,n}$, as well as the permutations $\sigma_{(i)}(g,v)$ for $g \in \Aut(T_{d_0,d_1,n})$ and $v \not \in \partial T_{d_0,d_1,n}$ (i.e. $v$ is not a leaf), are defined as for semiregular trees. We can also define a legal coloring of $T'_{d_0,d_1,n}$: it suffices to precise that only $d_0-1$ colors are used for the vertices adjacent to $v_0$. The notation $\sigma_{(i)}(g,v)$ has also a meaning, but $\sigma_{(i)}(g,v_0) \in \Sym(d_0-1)$ instead of $\Sym(d_0)$.
Given $\tilde{T} = T_{d_0,d_1,n}$ or $T'_{d_0,d_1,n}$ with a legal coloring $i$, we finally define
$$\Alt_{(i)}(\tilde{T}) := \{g \in \Aut(\tilde{T}) \suchthat \sigma_{(i)}(g, v) \text{ is even for each $v \not \in \partial \tilde{T}$}\}.$$
In the rest of this section and for the sake of brevity, we will sometimes forget the word \textit{legal} and write \textit{coloring} instead of \textit{legal coloring}.


\begin{lemma}\label{lemma:square}
Let $\tilde{T} = T'_{d_0,d_1,n}$ with $d_0,d_1 \geq 3$ and let $i$ be a legal coloring of $\tilde{T}$. Then $\Alt_{(i)}(\tilde{T})$ is generated by the set $\{g^2 \suchthat g \in \Alt_{(i)}(\tilde{T})\}$.
\end{lemma}

\begin{proof}
We proceed by induction on $n$. For $n = 0$, the tree $T'_{d_0,d_1,0}$ has only one vertex and there is nothing to prove. Now let $n \geq 1$ and assume the result is proven for $n-1$. This means that $\left\{g|_{B(v_0,n-1)}^2 \suchthat g \in \Alt_{(i)}(\tilde{T})\right\}$ generates $\Alt_{(i)}(B(v_0,n-1))$, where $v_0$ is the root of $\tilde{T}$. Hence, it suffices to show that $\left\{g^2 \suchthat g \in \Fix_{\Alt_{(i)}(\tilde{T})}(B(v_0,n-1))\right\}$ generates $\Fix_{\Alt_{(i)}(\tilde{T})}(B(v_0,n-1))$. Since alternating groups are generated by $3$-cycles, the group $\Fix_{\Alt_{(i)}(\tilde{T})}(B(v_0,n-1))$ is generated by the elements $f \in \Alt_{(i)}(\tilde{T})$ fixing $\tilde{T} \setminus \{a,b,c\}$ and such that $f(a) = b$, $f(b) = c$ and $f(c) = a$ where $a,b,c \in S(v_0,n)$ have the same parent. The conclusion simply follows from the fact that each such element $f$ is the square of $f^{-1} \in \Alt_{(i)}(\tilde{T})$.
\end{proof}

In the following, if $v$ is a vertex in a tree $\tilde{T}$ with root $v_0$, then $X_v$ is the branch of $v$, i.e. the subtree of $\tilde{T}$ spanned by $v$ and all its descendants. For $G \leq \Aut(\tilde{T})$, $\Rist_G(v)$ is the pointwise stabilizer in $G$ of $\tilde{T} \setminus X_v$. We will generally see $\Rist_G(v)$ as a subgroup of $\Aut(X_v)$. Finally, $G_k$ is the pointwise stabilizer in $G$ of $B(v_0, k)$ for $k \geq 0$.


\begin{lemma}\label{lemma:rooted}
Let $\tilde{T} = T_{d_0,d_1,n}$ or $T'_{d_0,d_1,n}$ with $d_0, d_1 \geq 6$ (and $n \geq 1$), let $v_0$ be the root of $\tilde{T}$ and let $i$ be a legal coloring of $B(v_0,n-1)$. Let $G \leq \Aut(\tilde{T})$ be such that $G_{n-1} \supseteq \Alt(d_0-1)^{c(v_0,n-1)}$ (or $\Alt(d_1-1)^{c(v_0,n-1)}$ or $\Alt(d_0)$, depending on $n$) and $G|_{B(v_0,n-1)} \supseteq \Alt_{(i)}(B(v_0,n-1))$. Then there exists a legal coloring $\overline{i}$ of $\tilde{T}$ extending $i$ such that $G \supseteq \Alt_{(\overline{i})}(\tilde{T})$. Moreover, if for some vertex $y_0 \in S(v_0,1)$ we already had a legal coloring $i'$ of $X_{y_0}$ coinciding with $i$ on $X_{y_0} \cap B(v_0,n-1)$ and such that $\Rist_G(y_0) \supseteq \Alt_{(i')}(X_{y_0})$, then $\overline{i}$ can be chosen to extend $i'$ too.
\end{lemma}

\begin{proof}
Define $e = d_0$ if $\tilde{T} = T_{d_0,d_1,n}$ and $e = d_0-1$ if $\tilde{T} = T'_{d_0,d_1,n}$, so that the root $v_0$ of $\tilde{T}$ has exactly $e$ neighbors. We proceed by induction on $n$. For $n = 1$, we have $G = G_0 \supseteq \Alt(e)$ by hypothesis, and thus any coloring $\overline{i}$ of $\tilde{T}$ is such that $G \supseteq \Alt_{(\overline{i})}(\tilde{T})$. Now let $n \geq 2$ and assume the lemma is true for $n-1$. We show it is also true for $n$.

Let $y_1, \ldots, y_e$ be the vertices of $S(v_0, 1)$. By hypothesis, $G_{n-1} \supseteq \Alt(\tilde{d}-1)^{c(v_0,n-1)}$ where $\tilde{d} = d_0$ if $n$ is odd and $\tilde{d} = d_1$ if $n$ is even. This implies in particular that $\Aut(X_{y_1}) \geq \Rist_G(y_1)_{n-2} \supseteq \Alt(\tilde{d}-1)^{c(y_1,n-2)}$ (where $c(y_1,n-2)$ counts the vertices of $X_{y_1}$ at distance $n-2$ from $y_1$ and $\Rist_G(y_1)$ is seen as a subgroup of $\Aut(X_{y_1})$).
We also claim that $\Rist_G(y_1)|_{B({y_1},n-2)} \supseteq \Alt_{(i)}(X_{y_1} \cap B(y_1,n-2))$. Indeed, since $G|_{B(v_0,n-1)} \supseteq \Alt_{(i)}(B(v_0,n-1))$, for each $h \in \Alt_{(i)}(X_{y_1} \cap B(y_1,n-2))$ there exists $g \in G$ fixing $(\tilde{T} \setminus X_{y_1}) \cap B(v_0,n-1)$ and acting as $h$ on $X_{y_1} \cap B(y_1,n-2)$. Then $g^2 \in G$ acts as $h^2$ on this set, and has the advantage that $g^2|_{E(x)}$ is an even permutation of $E(x)$ for each $x \in (\tilde{T} \setminus X_{y_1}) \cap S(v_0,n-1)$. As $G_{n-1} \supseteq \Alt(\tilde{d}-1)^{c(v_0,n-1)}$, there exists $f \in G_{n-1}$ such that $f|_{E(x)} = g^2|_{E(x)}$ for all those $x$. Then $f^{-1}g^2$ acts as $h^2$ on $X_{y_1} \cap B(y_1,n-2)$ and belongs to $\Rist_G(y_1)$. This means that $\Rist_G(y_1)|_{B(y_1,n-2)}$ contains $\{h^2 \suchthat h \in \Alt_{(i)}(X_{y_1}\cap B(y_1,n-2))\}$. By Lemma~\ref{lemma:square}, we obtain $\Rist_G(y_1)|_{B({y_1},n-2)} \supseteq \Alt_{(i)}(X_{y_1} \cap B(y_1,n-2))$. We can now use our induction hypothesis on $\Rist_G(y_1) \leq \Aut(X_{y_1})$ to get a coloring $i_1$ of $X_{y_1}$ extending $i$ and such that $\Rist_G(y_1) \supseteq \Alt_{(i_1)}(X_{y_1})$. In the particular case where we are given a vertex $y_0$ and a coloring $i'$ of $X_{y_0}$ as in the statement, we set $y_1 = y_0$ and rather define $i_1 = i'$.

Now take $g_1 \in G$ with $g_1|_{B(v_0, n-1)} \in \Alt_{(i)}(B(v_0, n-1))$ such that the induced action of $g_1$ on $S(v_0, 1)$ is the $3$-cycle $(y_1 \ y_3 \ y_2)$, $g_1$ fixes $X_y \cap B(v_0, n-1)$ for each $y \in S(v_0, 1) \setminus \{y_1, y_2, y_3\}$, and $g_1^3|_{B(v_0, n-1)} = \id|_{B(v_0, n-1)}$. Such an element $g_1$ exists as $G|_{B(v_0, n-1)} \supseteq \Alt_{(i)}(B(v_0,n-1))$.
The element $h_1 = g_1^2$ acts as the $3$-cycle $(y_1 \ y_2 \ y_3)$ on $S(v_0, 1)$, fixes $X_y \cap B(v_0, n-1)$ for each $y \in S(v_0, 1) \setminus \{y_1, y_2, y_3\}$ and also satisfies $h_1^3|_{B(v_0, n-1)} = \id|_{B(v_0, n-1)}$. In addition, $h_1|_{E(x)}$ is an even permutation of $E(x)$ for each $x \in (\tilde{T} \setminus (X_{y_1} \cup X_{y_2} \cup X_{y_3})) \cap S(v_0, n-1)$ (because $h_1 = g_1^2$). From $i_1$, construct a coloring $i_2$ of $X_{y_2}$ (coinciding with~$i$) such that $i_2|_{S(h_1(x),1)} \circ h_1 \circ i_1|_{S(x,1)}^{-1}$ is even for each $x \in X_{y_1} \cap S(v_0, n-1)$. In the same way, from $i_2$, construct a coloring $i_3$ of $X_{y_3}$ (coinciding with $i$) such that $i_3|_{S(h_1(x),1)} \circ h_1 \circ i_2|_{S(x,1)}^{-1}$ is even for each $x \in X_{y_2} \cap S(v_0, n-1)$. As $h_1 = g_1^2$, we also obtain that $i_1|_{S(h_1(x),1)} \circ h_1 \circ i_3|_{S(x,1)}^{-1}$ is even for each $x \in X_{y_3} \cap S(v_0, n-1)$. This exactly means that, for any coloring $\overline{i}$ of $\tilde{T}$ extending $i$, $i_1$, $i_2$ and $i_3$, it will be true that $h_1 \in \Alt_{(\overline{i})}(\tilde{T})$.

In the case where $e$ is odd, the proof is almost finished. Indeed, repeat this process to get $h_3 \in G$ inducing $(y_3 \ y_4 \ y_5)$ on $S(v_0, 1)$ and colorings $i_4$ of $X_{y_4}$ and $i_5$ of $X_{y_5}$, and so on until $h_{e-2} \in G$ inducing $(y_{e-2} \ y_{e-1} \ y_e)$ on $S(v_0, 1)$ and colorings $i_{e-1}$ of $X_{y_{e-1}}$ and $i_e$ of $X_{y_e}$. Then define $\overline{i}$ as the unique coloring extending $i, i_1, \ldots, i_e$. In view of our construction, $\overline{i}$ is such that $h_1, h_3, \ldots, h_{e-2} \in \Alt_{(\overline{i})}(\tilde{T})$. 
What is interesting about $h_1, h_3, \ldots, h_{e-2}$ is the fact that the permutations $(y_1 \ y_2 \ y_3), (y_3 \ y_4 \ y_5), \ldots, (y_{e-2} \ y_{e-1} \ y_e)$ generate $\Alt(e)$. In particular, as $\Rist_G(y_1) \supseteq \Alt_{(i_1)}(X_{y_1})$ we see by conjugating this inclusion with an element of $\langle h_1, h_3, \ldots, h_{e-2} \rangle$ sending $y_1$ on $y_k$ that $\Rist_G(y_k) \supseteq \Alt_{(i_k)}(X_{y_k})$ for each $k \in \{1, \ldots, e\}$. This means that $G$ contains all elements of $\Alt_{(\overline{i})}(\tilde{T})$ fixing $S(v_0, 1)$. Since it also contains $h_1, h_3, \ldots, h_{e-2} \in \Alt_{(\overline{i})}(\tilde{T})$ whose induced actions on $S(v_0, 1)$ generate $\Alt(e)$, we finally get $G \supseteq \Alt_{(\overline{i})}(\tilde{T})$.
 
If $e$ is even, then the exact same reasoning gives us $h_3, h_5,\ldots, h_{e-3}$ and colorings $i_4, \ldots, i_{e-1}$. At the end, there is no coloring of $X_{y_e}$ yet and the permutations $(y_1 \ y_2 \ y_3)$, $(y_3 \ y_4 \ y_5), \ldots$, $(y_{e-3} \ y_{e-2} \ y_{e-1})$ only generate the even permutations of $S(v_0, 1)$ fixing $y_e$. So as to conclude, take $g_{e-2} \in G$ as before so that the induced action on $S(v_0, 1)$ is $(y_{e-2} \ y_e \ y_{e-1})$ and define $h_{e-2} = g_{e-2}^2$. For simplicity, we write $h := h_{e-2}$. Here, the colorings $i_{e-2}$ and $i_{e-1}$ are already fixed and we can only choose a coloring $i_e$ of $X_{y_e}$. Choose $i_e$ so that $i_e|_{S(h(x),1)} \circ h \circ i_{e-1}|_{S(x,1)}^{-1}$ is even for each $x \in X_{y_{e-1}} \cap S(v_0, n-1)$, and define $\overline{i}$ as the unique coloring extending $i, i_1, \ldots, i_e$. The only issue preventing us from concluding as above is that it is not sure if $h \in \Alt_{(\overline{i})}(\tilde{T})$. The permutation $\sigma_{(\overline{i})}(h, x)$ could indeed be odd for some $x \in (X_{y_{e-2}} \cup X_{y_e}) \cap S(v_0, n-1)$. More precisely, these are the only vertices for which $\sigma_{(\overline{i})}(h, x)$ could be odd and we even know (because $h = g_{e-2}^2$) that $\sigma_{(\overline{i})}(h, x)$ with $x \in X_{y_e} \cap S(v_0, n-1)$ is odd if and only if $\sigma_{(\overline{i})}(h, h(x))$ is odd. We therefore define
$$O := \{x \in X_{y_e} \cap S(v_0, n-1) \suchthat \sigma_{(\overline{i})}(h, x) \text{ is odd}\},$$
so that $O \cup h(O)$ is exactly the set of vertices at which there is an odd permutation.

To finish the proof, we show that there exists $h' \in G \cap \Alt_{(\overline{i})}(\tilde{T})$ with $h'|_{B(v_0, n-1)} = h|_{B(v_0, n-1)}$. Denote by $a^{(e-2)}_1, \ldots, a^{(e-2)}_m$ the vertices of $X_{y_{e-2}} \cap S(v_0,n-1)$. Then define $a^{(e-1)}_j = h(a^{(e-2)}_j)$ and $a^{(e)}_j = h(a^{(e-1)}_j)$ for each $j \in \{1, \ldots, m\}$. Finally, for each $k \in \{1, \ldots, e-3\}$ choose $r_k \in G \cap \Alt_{(\overline{i})}(\tilde{T})$ such that $r_k(y_{e-2}) = y_k$ and define $a^{(k)}_j = r_k(a^{(e-2)}_j)$ for all $j$. We say that $f \in \Aut(\tilde{T})$ \textit{preserves the labelling} if $f(y_k) = y_{\ell}$ implies $f(a^{(k)}_j) = a^{(\ell)}_j$ for all $j$. One sees that if $f$ preserves the labelling and if $\sigma_{(\overline{i})}(f,v_0)$ is even, then $f|_{B(v_0,n-1)} \in \Alt_{(\overline{i})}(B(v_0, n-1))$.

Choose $f_1, f_2 \in \Alt_{(\overline{i})}(\tilde{T})$ preserving the labelling, fixing $X_{y_e}$ and such that the induced action of $f_1$ (resp. $f_2$) on $S(v_0, 1)$ is the permutation $(y_1 \ y_2 \  y_{e-1})$ (resp. $(y_1 \ y_{e-2} \ y_{e-1})$). Such elements exist by the previous remark, and they are contained in $G$. Note that $d_0 \geq 6$, so $e \geq 5$ and $2 < e-2$. Let us look at the element $\tau = (f_1 \circ h \circ f_2)^2 \in G$. Clearly, $\tau$ preserves the labelling and it suffices to look at its action on $S(v_0,1)$ to know its action on $S(v_0, n-1)$. The action of $\tau$ on $S(v_0,1)$ is given by
$$[(y_1 \ y_2 \  y_{e-1}) (y_{e-2} \ y_{e-1} \ y_e) (y_1 \ y_{e-2} \ y_{e-1})]^2$$
which is exactly the trivial permutation. Hence, $\tau$ acts trivially on $B(v_0,n-1)$. We should now observe with the help of Lemma~\ref{lemma:sigma} if $\sigma_{(\overline{i})}(\tau, x)$ is even or odd, for each $x \in S(v_0,n-1)$. As $f_1, f_2 \in \Alt_{(\overline{i})}(\tilde{T})$, all the permutations they induce are even. Using that $\sigma_{(\overline{i})}(h, x)$ is odd if and only if $x \in O \cup h(O)$, we actually obtain that $\sigma_{(\overline{i})}(\tau, x)$ is odd if and only if $x \in O \cup h(O)$. This means that $h' = h \circ \tau \in G$, which acts as $h$ on $B(v_0, n-1)$, is such that $\sigma_{(\overline{i})}(h', x)$ is always even, i.e. $h' \in \Alt_{(\overline{i})}(\tilde{T})$.
\end{proof}


\subsection{The common subgroup \texorpdfstring{$\Alt_{(i)}(T)^+$}{Alt_(i)(T)^+}}

We are now ready to complete the proof of Theorem~\ref{maintheorem:Alt} from the introduction. For the reader's convenience we reproduce its statement.


\begin{repeattheorem}{maintheorem:Alt}\label{maintheorem:Alt'}
Let $T$ be the $(d_0,d_1)$-semiregular tree with $d_0, d_1 \geq 6$. Let $H \in \mathcal{H}_T^+$ be such that $\underline{H}(x) \cong F_0 \geq \Alt(d_0)$ for each $x \in V_0(T)$ and $\underline{H}(y) \cong F_1 \geq \Alt(d_1)$ for each $y \in V_1(T)$. Then there exists a legal coloring $i$ of $T$ such that $H \supseteq \Alt_{(i)}(T)^+$.
\end{repeattheorem}

\begin{proof}
Given $v \in V(T)$ and a coloring $i$ of $T$, we say that $i$ is \textbf{$n$-valid at $v$} (with $n \in \Nz$) if the natural image of $H(v)$ in $\Aut(B(v,n))$ contains $\Alt_{(i)}(B(v, n))$. If $H(v) \supseteq \Alt_{(i)}(T)^+(v)$, $i$ is said to be \textbf{$\infty$-valid at $v$}. As $H$ is closed in $\Aut(T)$, a coloring is $\infty$-valid at $v$ if and only if it is $n$-valid at $v$ for all $n \in \Nz$.

We first claim that if $i$ is a coloring of $T$ which is $\infty$-valid at $v_1$ and $n$-valid at $v_2$ where $v_1$ and $v_2$ are adjacent vertices (with $n \in \Nz$), then there exists a coloring $\tilde{i}$ of $T$ such that $\tilde{i}|_{B(v_1,n) \cup B(v_2,n)} = i|_{B(v_1,n) \cup B(v_2,n)}$ and which is $(n+1)$-valid at $v_1$ and $\infty$-valid at $v_2$. To prove the claim, first define $\tilde{i}$ on $B(v_2, n) \cup T_{v_1}$ by $\tilde{i}|_{B(v_2, n) \cup T_{v_1}} = i|_{B(v_2, n) \cup T_{v_1}}$, where $T_{v_1}$ is the subtree of $T$ spanned by the vertices which are closer to $v_1$ than to $v_2$. This is already sufficient for $\tilde{i}$ to be $(n+1)$-valid at $v_1$ and $n$-valid at $v_2$. Now suppose that $\tilde{i}$ is defined on $B(v_2, k) \cup T_{v_1}$ for some $k \geq n$. We explain how to extend it to $B(v_2, k+1) \cup T_{v_1}$ so that it becomes $(k+1)$-valid at $v_2$. Define $\tilde{T} = B(v_2, k+1)$ and denote by $G$ the image of $H(v_2)$ in $\Aut(\tilde{T})$. We have $G_{k} \supseteq \Alt(\tilde{d}-1)^{c(v_2, k)}$ (where $\tilde{d} = d_0$ or $d_1$) in view of Proposition~\ref{proposition:S^c(k)} and $G|_{B(v_2, k)} \supseteq \Alt_{(\tilde{i})}(B(v_2, k))$ since $\tilde{i}$ is $k$-valid at $v_2$. Moreover, $X_{v_1}$ is already colored (by $\tilde{i}$ too) and $\Rist_G(v_1) \supseteq \Alt_{(\tilde{i})}(X_{v_1})$ (because $i$ is $\infty$-valid at $v_1$ and $\tilde{i}|_{X_{v_1}} = i|_{X_{v_1}}$). Lemma~\ref{lemma:rooted} thus gives us an extension of $\tilde{i}$ to $\tilde{T}$ making it $(k+1)$-valid at $v_2$. The coloring $\tilde{i}$ of $T$ defined in this way by induction is $(n+1)$-valid at $v_1$ and $\infty$-valid at $v_2$.

To prove the theorem, fix $x \in V_0(T)$ and $y \in V_1(T)$ two adjacent vertices of $T$. As $\Alt_{(i)}(T)^+(x)$ and $\Alt_{(i)}(T)^+(y)$ generate $\Alt_{(i)}(T)^+$, a coloring $i$ of $T$ is such that $H \supseteq \Alt_{(i)}(T)^+$ if and only if $i$ is $\infty$-valid at $x$ and $y$. Let us construct such a coloring. By Proposition~\ref{proposition:S^c(k)} and Lemma~\ref{lemma:rooted}, there exists a coloring $i_1$ of $T$ which is $\infty$-valid at $x$. As all colorings, $i_1$ is $1$-valid at $y$. Using the claim, construct $i_{n+1}$ from $i_n$ with $i_{n+1}|_{B(x, n) \cup B(y,n)} = i_n|_{B(x, n) \cup B(y,n)}$ for each $n \geq 1$. For $n$ odd, $i_n$ is $\infty$-valid at $x$ and $n$-valid at $y$; while for $n$ even, $i_n$ is $n$-valid at $x$ and $\infty$-valid at $y$. There is now a natural way to define our coloring $i$ of $T$: for each $v \in V(T)$, set $i(v) = i_n(v)$ where $n$ is such that $v \in B(x,n) \cup B(y,n)$. By construction, $i$ is $\infty$-valid at $x$ and $y$.
\end{proof}


\section{A list of examples, simplicity and normalizers}\label{section:simple}

In this section, we define all the groups that will appear in our classification theorems and analyze some of their properties, for instance their simplicity.


\subsection{Definition of the examples}\label{subsection:definitions}

We first recall the definitions of the groups appearing in the introduction and also define new similar groups. The fact that they are indeed groups follows from Lemma~\ref{lemma:sigma}.


\begin{definition}\label{definition:groups++}
Let $T$ be the $(d_0,d_1)$-semiregular tree with $d_0,d_1 \geq 4$ and let $i$ be a legal coloring of $T$. When $v \in V(T)$ and $X$ is a subset of $\N$, we set $S_X(v) := \bigcup_{r \in X}S(v,r)$. The notation $X \subset_f \N$ means that $X$ is a \textit{non-empty finite subset} of $\N$. We also write $\Sgn_{(i)}(g,A) := \prod_{w \in A} \sgn(\sigma_{(i)}(g,w))$ when $A$ is a finite subset of $V(T)$ and $g \in \Aut(T)$.

First set $G_{(i)}^+(\varnothing, \varnothing) := \Aut(T)^+$. Then, for $X \subset_f \N$, define
$$G_{(i)}^+(X, \varnothing) := \left\{g \in \Aut(T)^+ \suchthat \Sgn_{(i)}(g,S_X(v)) = 1 \text{ for each $v \in V_{t}(T)$}\right\}$$
and
$$G_{(i)}^+(X^*, \varnothing) := \left\{g \in \Aut(T)^+ \suchthat 
\text{All } \Sgn_{(i)}(g,S_X(v)) \text{ with $v \in V_t(T)$ are equal}\right\},$$
where $t = (\max X) \bmod 2$. The groups $G_{(i)}^+(\varnothing, X)$ and $G_{(i)}^+(\varnothing, X^*)$ are defined in the same way but with $t = (1 + \max X) \bmod 2$. For $X_0,X_1 \subset_f \N$ and $Y_0 \in \{X_0, X_0^*\}$, $Y_1 \in \{X_1,X_1^*\}$, define
$$G_{(i)}^+(Y_0,Y_1) := G_{(i)}^+(Y_0,\varnothing) \cap G_{(i)}^+(\varnothing,Y_1).$$
Finally, for $X_0,X_1 \subset_f \N$, set
$$G_{(i)}^+(X_0,X_1)^* := \left\{g \in \Aut(T)^+ \suchthat \begin{array}{c}
\text{All } \Sgn_{(i)}(g,S_{X_0}(v)) \text{ with $v \in V_{t_0}(T)$ and}\\
\Sgn_{(i)}(g,S_{X_1}(v)) \text{ with $v \in V_{t_1}(T)$ are equal}
\end{array}\right\},$$
where $t_0 = (\max X_0) \bmod 2$ and $t_1 = (1 + \max X_1) \bmod 2$.

We write $\mathcal{G}_{(i)}$ for the set of all these groups. Two groups are considered as different in this definition as soon as they have a different name, but two different groups may have exactly the same elements. We also define the following subsets $\mathcal{S}_{(i)}$ and $\mathcal{N}_{(i)}$ of $\mathcal{G}_{(i)}$, so that $\mathcal{G}_{(i)} = \mathcal{S}_{(i)} \sqcup \mathcal{N}_{(i)}$:
$$\mathcal{S}_{(i)} := \left\{G_{(i)}^+(\varnothing, \varnothing), G_{(i)}^+(X_0,\varnothing), G_{(i)}^+(\varnothing,X_1), G_{(i)}^+(X_0,X_1) \suchthat X_0, X_1 \subset_f \N\right\},$$
\vspace{-0.6cm}
\begin{multline*}
\mathcal{N}_{(i)} := \left\{G_{(i)}^+(X_0^*, \varnothing), G_{(i)}^+(\varnothing,X_1^*), G_{(i)}^+(X_0^*,X_1), G_{(i)}^+(X_0,X_1^*),\right. \\ \left. G_{(i)}^+(X_0^*,X_1^*), G_{(i)}^+(X_0,X_1)^* \suchthat X_0, X_1 \subset_f \N\right\}.
\end{multline*}
Finally, denote by $s \colon \mathcal{N}_{(i)} \to \mathcal{S}_{(i)}$ the map which simply erases the stars $^*$. Our remark on the groups which are considered as different in $\mathcal{G}_{(i)}$ is essential for $s$ to be well-defined.
\end{definition}


\begin{lemma}[Theorem~\ref{maintheorem:simple} (i)]\label{lemma:U(Alt)transitive}
Let $H \in \mathcal{G}_{(i)}$. Then $H$ belongs to $\mathcal{H}_T^+$.
\end{lemma}

\begin{proof}
All the groups $H \in \mathcal{G}_{(i)}$ contain $\Alt_{(i)}(T)^+ = G_{(i)}^+(\{0\},\{0\})$ and are closed in $\Aut(T)$, so it suffices to prove that $\Alt_{(i)}(T)^+$ is $2$-transitive on $\partial T$. By Lemma~\ref{lemma:2-transitive}, it is equivalent to showing that $\Alt_{(i)}(T)^+(v)$ is transitive on $\partial T$ for each $v \in V(T)$. As $\Alt_{(i)}(T)^+$ is closed, we can just show that the fixator in $\Alt_{(i)}(T)^+$ of a geodesic $(v,w)$ with $v,w \in V(T)$ always acts transitively on $E(w) \setminus \{e\}$, where $e$ is the edge of $(v,w)$ adjacent to $w$. This is immediate, since $\Alt(d-1)$ is transitive when $d \geq 4$.
\end{proof}

Given $H \in \mathcal{G}_{(i)}$ and $h \in \Aut(T)^+$, it is not hard to determine whether $h$ belongs to $H$. Indeed, one can simply draw the tree $T$ and label each vertex $v$ of $T$ with the letter $e$ (for \textit{even}) or $o$ (for \textit{odd}) depending on the parity of $\sigma_{(i)}(h,v)$. A condition on the value of $\Sgn_{(i)}(h, S_X(v))$ then translates in a condition on the parity of the number of vertices labelled by $o$ in $S_X(v)$.

Using this observation, we can easily construct elements of $H$ step by step. For example, consider $H = G_{(i)}^+(X_0,X_1^*)$. Let us observe how one can construct any labelling of $T$ that satisfies the condition of being in $H$, i.e. such that if $h \in \Aut(T)^+$ realizes this labelling, then $h \in H$. First fix a vertex $v_0 \in V(T)$. For $n \in \N$ and given a labelling of $B(v_0, n-1)$ (if $n \neq 0$), we look at how it can be extended to a labelling of $B(v_0, n)$ while satisfying the conditions for being in $H$.  Suppose we already have a labelling of $B(v_0, n-1)$ not contradicting any of the conditions. Let $t \in \{0,1\}$ be the type of the vertices of $S(v_0,n)$. If $n < \max X_t$, then there is no set $S_{X_0}(v)$ or $S_{X_1}(v)$ contained in $B(v_0,n)$ but not already contained in $B(v_0,n-1)$, so the labelling can be extended with no restriction. On the contrary, if $n \geq \max X_t$, then our new labelling must satisfy some additional conditions: the ones on the set $S_{X_t}(v)$ where $v$ is a vertex at distance $n-\max X_t$ from $v_0$. But $\{S_{X_t}(v) \cap S(v_0,n) \suchthat d(v,v_0) = n-\max X_t\}$ is a partition of $S(v_0,n)$, so there is only one condition on the parity of the number of labels $o$ on each set $S_{X_t}(v) \cap S(v_0,n)$. If $t = 0$ (recall that we consider $H = G_{(i)}^+(X_0,X_1^*)$), we just have to make sure that there is an even number of vertices labelled by $o$ in $S_{X_0}(v)$. If $t = 1$, then we distinguish the following two cases. If this is the first time (of the whole process) that we observe a set of the form $S_{X_1}(v)$, then we can still make the choice of the parity of the number of labels $o$ in $S_{X_1}(v)$. Otherwise, this parity must be the same as for this first choice. In all cases, we still have a lot of freedom in our choice of the new labelling. A labelling of $T$ constructed in this way will always be suitable, since everything was made for the conditions to be met.


\subsection{Simplicity}

It is clear that each group $H \in \mathcal{N}_{(i)}$ has $s(H)$ as a proper normal subgroup, and is therefore not simple. Our next goal is to prove that the groups in $\mathcal{S}_{(i)}$ are simple. Banks, Elder and Willis \cite{Banks} provided tools to show that a group of automorphisms of trees is simple. Those happen to be exactly what we need. Note that their work is based on a generalization of Tits' Property~P (see~\cite{Titsarbres}). For $G \leq_{cl} \Aut(T)$ and $k \in \Nz$, define
$$G^{+_k} := \langle G_{k-1}(v,w) \suchthat [v,w] \in E(T)\rangle.$$
The next proposition is a combination of results of \cite{Banks}.


\begin{proposition}\label{proposition:BEW}
Let $Y_0$ and $Y_1$ be (possibly empty) finite subsets of $\N$, let $H = G_{(i)}^+(Y_0,Y_1) \in \mathcal{S}_{(i)}$ and let $M = \max(\max Y_0, \max Y_1) + 1$, where we set $\max (\varnothing) = 0$ by convention. Then $H^{+_M}$ is abstractly simple.
\end{proposition}

\begin{proof}
Recall the following definition for $n \in \N$:
$$H^{(n)} := \{g \in \Aut(T) \suchthat \forall v \in V(T), \exists h \in H : g|_{B(v,n)} = h|_{B(v,n)}\}.$$
In our case, it is clear from the definition of $H$ that $H^{(M)} = H$. Hence, by \cite{Banks}*{Proposition~5.2}, $H$ has Property $IP_M$ (as defined in \cite{Banks}*{Definition~5.1}). Since $H$ is a closed subgroup of $\Aut(T)$, we deduce from \cite{Banks}*{Corollary~6.4} that $H$ has Property $P_M$ (as defined in \cite{Banks}*{Definition~6.2}). We can therefore apply \cite{Banks}*{Theorem~7.3} that asserts that $H^{+_M}$ is abstractly simple or trivial. Since there exist non-trivial elements in $\Alt_{(i)}(T)^+ \subseteq H$ fixing arbitrarily large balls, we conclude that $H^{+_M}$ is abstractly simple.
\end{proof}

In order to prove that a group $H \in \mathcal{S}_{(i)}$ is simple, we therefore only need to prove that $H = H^{+_M}$, where $M = \max(\max Y_0, \max Y_1) + 1$. We first assert that $H^{+_1} = H$. Note that $H^{+_1}$ is the subgroup of $H$ generated by the elements fixing an edge of $T$.


\begin{lemma}\label{lemma:H+1}
Let $H \in \mathcal{H}_T^+$ (with $d_0,d_1 \geq 3$). Then $H^{+_1} = H$.
\end{lemma}

\begin{proof}
The result readily follows from the fact that the fixator of an edge $e = [v,w]$ in $H$ is transitive on $E(v) \setminus \{e\}$ (by Lemma~\ref{corollary:2-transitive}).
\end{proof}

For $[v,w] \in E(T)$, we write $T_{v,w}$ for the subtree of $T$ spanned by the vertices that are closer to $v$ than to $w$. Such a subtree is called a \textbf{half-tree}.


\begin{lemma}\label{lemma:simple}
Let $H \in \mathcal{S}_{(i)}$. Then $H$ is generated by the elements of $H$ fixing a half-tree of~$T$. In particular, $H = H^{+_k}$ for any $k \in \Nz$.
\end{lemma}

\begin{proof}
We already know by Lemma~\ref{lemma:H+1} that $H = \langle H(v,w) \suchthat [v,w] \in E(T)\rangle$. Let us now prove that each $h \in H(v,w)$ (for some $[v,w] \in E(T)$) is generated by elements of $H$ fixing a half-tree of $T$. We construct an element $g \in H$ such that $g|_{T_{v,w}} = h|_{T_{v,w}}$ and $g$ fixes some half-tree of $T$. This will prove the statement as $h = (hg^{-1})g$.

First define $g$ on $T_{v,w}$ by declaring that $g|_{T_{v,w}} = h|_{T_{v,w}}$. Now look at the labelled tree associated to $g$: for the moment, all the vertices of $T_{v,w}$ are labelled by $e$ or $o$. We also label all the vertices of $T_{w,v} \cap B(w,M-1)$ where $M = \max(\max Y_0, \max Y_1)+1$ exactly as in the labelled tree associated to $h$. Since $h \in H$, all the conditions to be in $H$ which concern $T_{v,w} \cup B(w,M-1)$ are satisfied.

We now want to put new labellings on $S(w,n) \cap T_{w,v}$ for each $n \geq M$. Before doing so, we number the edges of $T_{w,v}$ in the following way: if $x$ is a vertex at distance $D$ from $w$, the edges from $x$ to a vertex at distance $D+1$ from $w$ are numbered with $1, 2, \ldots, d_0-1$ (or $d_1-1$). Let us now label the whole tree with $e$ and $o$. As already explained in Subsection~\ref{subsection:definitions}, at each step there will be conditions on the parity of the number of labels $o$ in sets of the form $S_X(x)$. More precisely, if we look at $S(w,n)$ (for $n \geq M$), then either there is no new condition to satisfy (because of the symbol $\varnothing$ in $H$), or there is a condition on each set of the form $S(w,n) \cap S(x,\max X)$ (where $X = X_0$ or $X = X_1$) with $x$ is at distance $n-\max X$ from $w$. If there is no condition then we label all the vertices of $S(w,n) \cap T_{w,v}$ by $e$. Otherwise, in each set $S(w,n) \cap S(x,\max X)$ with $x$ at distance $n-\max X$ from $w$, the number of vertices labelled by $o$ must be either even or odd (depending on the previous labellings). If it must be even, we label all the vertices of $S(w,n) \cap S(x,\max X)$ by $e$. If it must be odd, we label by $o$ the vertex $z$ of $S(w,n) \cap S(x,\max X)$ such that the path from $x$ to $z$ only contains edges numbered $1$. All the other vertices are labelled by $e$ (see Figure~\ref{picture:labelling} where $n = 3$ and $\max X = 2$).

\begin{figure}
\centering
\begin{pspicture*}(-6,-3.3)(6,1.9)
\fontsize{10pt}{10pt}\selectfont
\psset{unit=1cm}

\psline(0,1)(-2,2)
\psline(0,1)(0,2)
\psline(0,1)(2,2)

\psline(0,0)(0,1)

\psline(0,0)(-3,-1) \rput(-1.9,-0.4){$1$}
\psline(0,0)(0,-1) \rput(0.17,-0.6){$2$}
\psline(0,0)(3,-1) \rput(1.9,-0.4){$3$}

\psline(-3,-1)(-4,-2) \rput(-3.7,-1.4){$1$}
\psline(-3,-1)(-3,-2) \rput(-2.83,-1.7){$2$}
\psline(-3,-1)(-2,-2) \rput(-2.3,-1.4){$3$}

\psline(0,-1)(-1,-2) \rput(-0.7,-1.4){$1$}
\psline(0,-1)(0,-2) \rput(0.17,-1.7){$2$}
\psline(0,-1)(1,-2) \rput(0.7,-1.4){$3$}

\psline(3,-1)(2,-2) \rput(2.3,-1.4){$1$}
\psline(3,-1)(3,-2) \rput(3.17,-1.7){$2$}
\psline(3,-1)(4,-2) \rput(3.7,-1.4){$3$}

\psline(-4,-2)(-4.3,-3) \rput(-4.33,-2.5){$1$} \rput(-4.3,-3.2){$o$}
\psline(-4,-2)(-4,-3) \rput(-4,-3.2){$e$}
\psline(-4,-2)(-3.7,-3) \rput(-3.67,-2.5){$3$} \rput(-3.7,-3.2){$e$}
\psline(-3,-2)(-3.3,-3) \rput(-3.33,-2.5){$1$} \rput(-3.3,-3.2){$e$}
\psline(-3,-2)(-3,-3) \rput(-3,-3.2){$e$}
\psline(-3,-2)(-2.7,-3) \rput(-2.67,-2.5){$3$} \rput(-2.7,-3.2){$e$}
\psline(-2,-2)(-2.3,-3) \rput(-2.33,-2.5){$1$} \rput(-2.3,-3.2){$e$}
\psline(-2,-2)(-2,-3) \rput(-2,-3.2){$e$}
\psline(-2,-2)(-1.7,-3) \rput(-1.67,-2.5){$3$} \rput(-1.7,-3.2){$e$}

\psline(-1,-2)(-1.3,-3) \rput(-1.33,-2.5){$1$} \rput(-1.3,-3.2){$e$}
\psline(-1,-2)(-1,-3) \rput(-1,-3.2){$e$}
\psline(-1,-2)(-0.7,-3) \rput(-0.67,-2.5){$3$} \rput(-0.7,-3.2){$e$}
\psline(0,-2)(-0.3,-3) \rput(-0.33,-2.5){$1$} \rput(-0.3,-3.2){$e$}
\psline(0,-2)(0,-3) \rput(0,-3.2){$e$}
\psline(0,-2)(0.3,-3) \rput(0.33,-2.5){$3$} \rput(0.3,-3.2){$e$}
\psline(1,-2)(0.7,-3) \rput(0.67,-2.5){$1$} \rput(0.7,-3.2){$e$}
\psline(1,-2)(1,-3) \rput(1,-3.2){$e$}
\psline(1,-2)(1.3,-3) \rput(1.33,-2.5){$3$} \rput(1.3,-3.2){$e$}

\psline(2,-2)(1.7,-3) \rput(1.67,-2.5){$1$} \rput(1.7,-3.2){$o$}
\psline(2,-2)(2,-3) \rput(2,-3.2){$e$}
\psline(2,-2)(2.3,-3) \rput(2.33,-2.5){$3$} \rput(2.3,-3.2){$e$}
\psline(3,-2)(2.7,-3) \rput(2.67,-2.5){$1$} \rput(2.7,-3.2){$e$}
\psline(3,-2)(3,-3) \rput(3,-3.2){$e$}
\psline(3,-2)(3.3,-3) \rput(3.33,-2.5){$3$} \rput(3.3,-3.2){$e$}
\psline(4,-2)(3.7,-3) \rput(3.67,-2.5){$1$} \rput(3.7,-3.2){$e$}
\psline(4,-2)(4,-3) \rput(4,-3.2){$e$}
\psline(4,-2)(4.3,-3) \rput(4.33,-2.5){$3$} \rput(4.3,-3.2){$e$}

\rput(0.22,0.85){$v$}
\rput(-0.22,0.85){$e$}
\rput(0.22,0.15){$w$}

\rput(-3.1,-0.8){$e$}
\rput(-0.22,-0.9){$o$}
\rput(3.1,-0.8){$e$}

\end{pspicture*}
\caption{Illustration of Lemma~\ref{lemma:simple} for $H = G_{(i)}^+(\{2\},\{2\})$.}\label{picture:labelling}
\end{figure}

We claim that, after having followed these rules to label the whole tree, there will always exists a half-tree $T_{s,t}$ whose vertices are all labelled by $e$ (with $s, t \in T_{w,v}$ and $t$ closer to $w$ than $s$). This will complete the proof, since it is always possible to define $g$ on $T_{w,v}$ such that $g$ fixes the whole path from $w$ to $t$, fixes $T_{s,t}$, and realizes the labelled tree that we just constructed. (Note that we need $d_0,d_1 \geq 4$ to have sufficient freedom.)

Let us prove the claim. Let $s_0$ be a vertex of $T_{w,v} \cap S(w,M)$ labelled by $e$. Define $(s_n)_{n \in \N}$ by saying that $s_j$ is the vertex adjacent to $s_{j-1}$ farther from $w$ than $s_j$ and such that $[s_{j-1}, s_j]$ is numbered $2$. We show by induction that, for each $j \in \N$, the ball $B(s_j, j)$ only contains vertices labelled by $e$. For $j = 0$ this is clear. Now assume that all the vertices of $B(s_j,j)$ are labelled by $e$ and look at the ball $B(s_{j+1}, j+1)$. All the vertices of $B(s_{j+1}, j+1) \cap B(s_j,j)$ are labelled by $e$, so we only need to observe $B(s_{j+1}, j+1) \setminus B(s_j,j) = T_{s_{j+1},s_j} \cap (S(s_{j+1},j) \cup S(s_{j+1},j+1))$. The labels of the vertices of $T_{s_{j+1},s_j} \cap S(s_{j+1},j) =: A$ were determined according to some eventual conditions on sets of the form $S_X(x)$. If there are no such conditions, then all the vertices of $A$ were labelled by $e$ as wanted. Otherwise, there are two cases: either $\max X \leq j$ or $\max X > j$. If $\max X \leq j$, then $S_X(x) \setminus A \subseteq B(s_{j},j)$ so all the vertices of $S_X(x) \setminus A$ are labelled by $e$ and the vertices of $A$ were therefore also labelled by $e$. If $\max X > j$, then the condition on $S_X(x)$ may have been to put a label $o$ somewhere, but in any case this label $o$ was not put in $A$ since $[s_{j},s_{j+1}]$ is numbered~$2$ (and not~$1$). So all the vertices of $A$ were labelled by $e$. The reasoning is exactly the same for $T_{s_{j+1},s_j} \cap S(s_{j+1},j+1) =: A'$. This means that $B(s_M, M)$ only contains vertices labelled by $e$. Hence, in $T_{s_M, s_{M-1}}$ there is no condition on a set $S_X(x)$ asking to label a vertex by $o$ (because $\max X < M$). All the vertices of the half-tree $T_{s_M, s_{M-1}}$ are thus labelled by $e$.
\end{proof}


\begin{theorem}[Theorem~\ref{maintheorem:simple} (ii)]\label{theorem:simple}
Let $H \in \mathcal{S}_{(i)}$. Then $H$ is abstractly simple.
\end{theorem}

\begin{proof}
This follows from Proposition~\ref{proposition:BEW} and Lemma~\ref{lemma:simple}.
\end{proof}


\subsection{Are these examples pairwise distinct?}

As highlighted in Definition~\ref{definition:groups++}, it is not clear for the moment if the members of $\mathcal{G}_{(i)}$ are pairwise different. One can actually remark that this is not the case: for instance, if $X_0, X_1 \subset_f \N$ are such that $\max X_0 \not \equiv \max X_1 \bmod 2$ and $\max X_0 < \max X_1$, then $G_{(i)}^+(X_0,X_1) = G_{(i)}^+(X_0,X_1 \triangle X_0)$ and $G_{(i)}^+(X_0,X_1)^* = G_{(i)}^+(X_0^*,X_1 \triangle X_0)$, where $\triangle$ denotes the symmetric difference. For this reason, we introduce the following definition.


\begin{definition}\label{definition:classification}
Let $T$ be the $(d_0,d_1)$-semiregular tree with $d_0,d_1 \geq 4$ and let $i$ be a legal coloring of~$T$. Say that $X_0, X_1 \subset_f \N$ are \textbf{compatible} if for each $x \in X_t$ (with $t \in \{0,1\}$), if $x \geq \max X_{1-t}$ then $x \equiv \max X_t \bmod 2$. Define $\underline{\mathcal{G}}_{(i)}$ to be the set containing the following groups:
\begin{itemize}
\item $G_{(i)}^+(Y_0, Y_1)$, where $Y_0 \in \{\varnothing, X_0, X_0^*\}$, $Y_1 \in \{\varnothing, X_1, X_1^*\}$, $X_0, X_1 \subset_f \N$ and, if $Y_0 \neq \varnothing$ and $Y_1 \neq \varnothing$, then $X_0$ and $X_1$ are compatible;
\item $G_{(i)}^+(X_0, X_1)^*$, where $X_0, X_1 \subset_f \N$ are compatible.
\end{itemize}
\end{definition}

We then have the following result.


\begin{proposition}\label{proposition:different}
The members of $\underline{\mathcal{G}}_{(i)}$ are pairwise different.
\end{proposition}

\begin{proof}
The groups in $\mathcal{S}_{(i)}$ are simple (Theorem~\ref{theorem:simple}) while those in $\mathcal{N}_{(i)}$ are not, so a group in $\mathcal{S}_{(i)}$ is never equal to a group in $\mathcal{N}_{(i)}$.

Let us now prove that two groups $G_{(i)}^+(Y_0,Y_1)$ and $G_{(i)}^+(Y'_0,Y'_1)$ in $\mathcal{S}_{(i)} \cap \underline{\mathcal{G}}_{(i)}$ with $(Y_0,Y_1) \neq (Y'_0,Y'_1)$ are always different. If $Y_0 = \varnothing$ but $Y'_0 \neq \varnothing$, then $G_{(i)}^+(\varnothing,Y_1) \not \subseteq G_{(i)}^+(Y'_0,Y'_1)$. Indeed, for each ball $B(v,n)$ in $T$ such that $S(v,n) \subseteq V_0(T)$, the fixator of $B(v,n)$ in $G_{(i)}^+(\varnothing, Y_1)$ can act in any manner on $B(v,n+1)$. This is not true for $G_{(i)}^+(Y'_0,Y'_1)$ when $n \geq \max Y'_0$. This reasoning works whenever exactly one of the two sets $Y_t$ and $Y'_t$ is empty for some $t \in \{0,1\}$.

Now consider $X_0 \neq X'_0$ and let us show that $G_{(i)}^+(X_0,\varnothing) \neq G_{(i)}^+(X'_0,\varnothing)$ (the proof is exactly the same for $G_{(i)}^+(\varnothing, X_1) \neq G_{(i)}^+(\varnothing, X'_1)$ with $X_1 \neq X'_1$). If $\max X_0 < \max X'_0$, then fix $v$ a vertex of type $(\max X_0) \bmod 2$ and construct (as explained in Subsection~\ref{subsection:definitions} with the labellings, starting from $v$) an element $h \in G_{(i)}^+(X'_0,\varnothing)$ such that there is exactly one vertex labelled by $o$ in $S_{X_0}(v)$. This is possible because $\max X_0 < \max X'_0$. By definition, $h \not \in G_{(i)}^+(X_0,\varnothing)$. The reasoning is the same when $\max X_0 > \max X'_0$. Now assume that $\max X_0 = \max X'_0$. Suppose without loss of generality that $X'_0 \not \subseteq X_0$ and take $r \in X'_0 \setminus X_0$. Then, if $v$ is a vertex of type $(\max X_0) \bmod 2$, there exists $h \in G_{(i)}^+(X_0, \varnothing)$ such that there is exactly one vertex labelled by $o$ in $S(v,r)$ and all the other vertices of $B(v,\max X_0)$ are labelled by $e$. This element $h$ is not in $G_{(i)}^+(X'_0, \varnothing)$ because it does not satisfy the condition on $S_{X'_0}(v)$ (since $r \in X'_0$).

Finally (for $\mathcal{S}_{(i)} \cap \underline{\mathcal{G}}_{(i)}$), let $(X_0,X_1) \neq (X'_0,X'_1)$ be such that $X_0$ and $X_1$ (resp. $X'_0$ and $X'_1$) are compatible and let us show that $G_{(i)}^+(X_0,X_1) \neq G_{(i)}^+(X'_0,X'_1)$. As in the previous case, if $\max X_0 < \max X'_0$ then we can construct an element $h \in G_{(i)}^+(X'_0,X'_1)$ which is not in $G_{(i)}^+(X_0,X_1)$. The same reasoning works when $\max X_0 > \max X'_0$ or $\max X_1 \neq \max X'_1$. Now assume that $\max X_0 = \max X'_0$ and $\max X_1 = \max X'_1$, and without loss of generality that $\max X_0 \leq \max X_1$. If $X_0 \neq X'_0$ then as before we obtain an element that is in exactly one of the two groups $G_{(i)}^+(X_0,X_1)$ and $G_{(i)}^+(X'_0,X'_1)$. Now suppose that $X_0 = X'_0$ and $X_1 \neq X'_1$. Once again, assume without loss of generality that $X'_1 \not \subseteq X_1$. Let $r$ be the greatest element of $X'_1 \setminus X_1$. Since $X'_0$ and $X'_1$ are compatible, we have $r < \max X_0$ or $r \equiv \max X_1 \bmod 2$. If $v$ is a vertex of type $(1 + \max X_1) \bmod 2$, this means that there is no set of the form $S_{X_0}(x)$ (with $x$ of type $(\max X_0) \bmod 2$) which is contained in $B(v,r)$ but not already in $B(v,r-1)$. Hence, there exists $h \in G_{(i)}^+(X_0, X_1)$ with exactly one vertex labelled by $o$ in $S(v,r)$ and all the other vertices of $B(v,r)$ labelled by $e$. Our choice for $r$ is such that $h$ cannot also be an element of $G_{(i)}^+(X'_0, X'_1)$.

We proved that the groups in $\mathcal{S}_{(i)} \cap \underline{\mathcal{G}}_{(i)}$ are pairwise different. Let us now do it for $\mathcal{N}_{(i)} \cap \underline{\mathcal{G}}_{(i)}$. Take $H, H' \in \mathcal{N}_{(i)} \cap \underline{\mathcal{G}}_{(i)}$ with different names. If $H$ and $H'$ have exactly the same sets $X_0$ and/or $X_1$ in their name, i.e. if $s(H) = s(H')$, then it is clear from the definitions and the constructions with the labellings (see Subsection~\ref{subsection:definitions}) that $H \neq H'$. We can therefore assume that $s(H) \neq s(H')$ (and these groups are really different because of our work for $\mathcal{S}_{(i)} \cap \underline{\mathcal{G}}_{(i)}$). Suppose for a contradiction that $H = H'$. Recall that $s(H) \unlhd H$, $s(H') \unlhd H'$ and $s(H)$ and $s(H')$ are simple. Hence, the group $H = H'$ has two different simple normal subgroups having a non-trivial intersection ($s(H)$ and $s(H')$ both contain $\Alt_{(i)}(T)^+$), which is impossible.
\end{proof}

As a corollary of the classification, we will also see that each group in $\mathcal{G}_{(i)}$ is equal to a group in $\underline{\mathcal{G}}_{(i)}$, which makes the definition of $\underline{\mathcal{G}}_{(i)}$ completely natural.


\subsection{Normalizers}

We are now interested in the normalizers of all our groups. Before giving them, we must define the groups that will appear in the classification for groups in $\mathcal{H}_T \setminus \mathcal{H}_T^+$.


\begin{definition}\label{definition:groups}
Let $T$ be the $d$-regular tree with $d \geq 4$ and let $i$ be a legal coloring of $T$. With the same notation as in Definition~\ref{definition:groups++}, set $G_{(i)}(\varnothing,\varnothing) := \Aut(T)$ and define, for $X \subset_f \N$,
$$G_{(i)}(X,X) := \left\{g \in \Aut(T) \suchthat \Sgn_{(i)}(g, S_X(v)) = 1 \text{ for each $v \in V(T)$}\right\},$$
$$G_{(i)}(X^*,X^*) := \left\{g \in \Aut(T) \suchthat \begin{array}{c}
\text{All } \Sgn_{(i)}(g, S_X(v)) \text{ with $v \in V_0(T)$ are equal} \\
\text{and all } \Sgn_{(i)}(g, S_X(v)) \text{ with $v \in V_1(T)$ are equal}
\end{array}\right\},$$
$$G_{(i)}(X,X)^* := \left\{g \in \Aut(T) \suchthat
\text{All } \Sgn_{(i)}(g, S_X(v)) \text{ with $v \in V(T)$ are equal}\right\},$$
and
$$G_{(i)}'(X,X)^* := \left\{g \in \Aut(T) \suchthat \begin{array}{c}
\text{All } \Sgn_{(i)}(g, S_X(v)) \text{ with $v \in V_0(T)$ are equal to} \\
\text{$p_0$, all } \Sgn_{(i)}(g, S_X(v)) \text{ with $v \in V_1(T)$ are equal} \\
\text{to $p_1$, and $p_0 = p_1$ if and only if $g \in \Aut(T)^+$}
\end{array}\right\}.$$
We write $\mathcal{G}'_{(i)}$ for the set of all these groups.
\end{definition}

The normalizers are then given in the following lemma.


\begin{lemma}[Theorem~\ref{maintheorem:simple} (iii)]\label{lemma:normalizer}
Let $T$ be the $(d_0,d_1)$-semiregular tree with $d_0,d_1 \geq 4$  and let $i$ be a legal coloring of $T$.
\begin{enumerate}[(i)]
\item Define the map $n^+ \colon \mathcal{G}_{(i)} \to \mathcal{G}_{(i)}$ by $n^+(G_{(i)}^+(\varnothing,\varnothing)) = G_{(i)}^+(\varnothing,\varnothing)$, $n^+(G_{(i)}^+(X_0,\varnothing)) = G_{(i)}^+(X_0^*,\varnothing)$, 
$n^+(G_{(i)}^+(\varnothing,X_1)) = G_{(i)}^+(\varnothing,X_1^*)$,
$n^+(G_{(i)}^+(X_0,X_1)) = G_{(i)}^+(X_0^*,X_1^*)$ and $n^+(H) = n^+(s(H))$ for $H \in \mathcal{N}_{(i)}$. Let $H \in \mathcal{G}_{(i)}$.
\begin{enumerate}[(a)]
\item If $\tau \in \Aut(T)^+$ is such that $\tau H \tau^{-1} \supseteq \Alt_{(i)}(T)^+$, then $\tau \in n^+(H)$.
\item $n^+(H)$ is the normalizer of $H$ in $\Aut(T)^+$.
\end{enumerate}
\item Suppose that $d_0 = d_1$. For each $X \subset_f \N$, the normalizer of $G_{(i)}^+(X,X)$ (resp. $G_{(i)}^+(X^*,X^*)$ and $G_{(i)}^+(X,X)^*$) in $\Aut(T)$ is $G_{(i)}(X^*,X^*)$.
\end{enumerate}
\end{lemma}

\begin{proof}\hspace{1cm}
\begin{enumerate}[(i)]
\item Let us first prove (a). Since $n^+(H) \supseteq H$, having $\tau H \tau^{-1} \supseteq \Alt_{(i)}(T)^+$ implies $\tau n^+(H) \tau^{-1} \supseteq \Alt_{(i)}(T)^+$. As $n^+(n^+(H)) = n^+(H)$, this means that we can just prove the statement for $H = G_{(i)}^+(\varnothing,\varnothing)$, $G_{(i)}^+(X_0^*,\varnothing)$, $ G_{(i)}^+(\varnothing,X_1^*)$ and $G_{(i)}^+(X_0^*,X_1^*)$. If $H = G_{(i)}^+(\varnothing,\varnothing) = \Aut(T)^+$ then there is nothing to prove.

Now consider $H = G_{(i)}^+(X_0^*,\varnothing)$. Let $\tau \in \Aut(T)^+$ be such that $\tau G_{(i)}^+(X_0^*,\varnothing) \tau^{-1} \supseteq \Alt_{(i)}(T)^+$. Remind that
$$G_{(i)}^+(X_0^*, \varnothing) := \left\{g \in \Aut(T)^+ \suchthat 
\text{All } \Sgn_{(i)}(g,S_{X_0}(v)) \text{ with $v \in V_t(T)$ are equal}\right\},$$
where $t = (\max X_0) \bmod 2$. We therefore directly obtain
$$\tau G_{(i)}^+(X_0^*,\varnothing) \tau^{-1} = \left\{g \in \Aut(T)^+ \suchthat \begin{array}{c}
\text{All } \Sgn_{(i)}(\tau^{-1}g \tau,S_{X_0}(v))\\ \text{ with $v \in V_t(T)$ are equal}\end{array}\right\}.$$
By Lemma~\ref{lemma:sigma}, we have $\sigma_{(i)}(\tau^{-1}g\tau,w) = \sigma_{(i)}(\tau^{-1},g\tau(w)) \circ \sigma_{(i)}(g,\tau(w)) \circ \sigma_{(i)}(\tau,w)$ and $\sigma_{(i)}(\tau^{-1},g\tau(w)) = \sigma_{(i)}(\tau, \tau^{-1}g\tau(w))^{-1}$, so $\Sgn_{(i)}(\tau^{-1}g \tau,S_{X_0}(v))$ is equal to
$$\Sgn_{(i)}(\tau,S_{X_0}(\tau^{-1}g\tau(v))) \cdot \Sgn_{(i)}(g,S_{X_0}(\tau(v))) \cdot \Sgn_{(i)}(\tau,S_{X_0}(v)).$$
We want to prove that $\tau \in G_{(i)}^+(X_0^*, \varnothing)$, i.e. that all $\Sgn_{(i)}(\tau,S_{X_0}(v))$ with $v \in V_t(T)$ are equal. It suffices to show that $\Sgn_{(i)}(\tau,S_{X_0}(x)) = \Sgn_{(i)}(\tau,S_{X_0}(y))$ when $x,y \in V_t(T)$ and $d(x,y) = 2$. Fix such $x$ and $y$ and consider $z \in V_t(T)$ such that $d(x,z) = d(y,z) = 2$. Take $g \in \Alt_{(i)}(T)^+$ such that $g(\tau(x)) = \tau(z)$ and $g(\tau(z)) = \tau(y)$. By hypothesis, we have $g \in \tau G_{(i)}^+(X_0^*,\varnothing) \tau^{-1}$ so the two values
$$\Sgn_{(i)}(\tau,S_{X_0}(z)) \cdot \Sgn_{(i)}(g,S_{X_0}(\tau(x))) \cdot \Sgn_{(i)}(\tau,S_{X_0}(x))$$
and
$$\Sgn_{(i)}(\tau,S_{X_0}(y)) \cdot \Sgn_{(i)}(g,S_{X_0}(\tau(z))) \cdot \Sgn_{(i)}(\tau,S_{X_0}(z))$$
are equal. As $g \in \Alt_{(i)}(T)^+$, we have $\Sgn_{(i)}(g, A) = 1$ for each finite set $A \subseteq V(T)$ and hence we get $\Sgn_{(i)}(\tau,S_{X_0}(x)) = \Sgn_{(i)}(\tau,S_{X_0}(y))$ as wanted.

For $H = G_{(i)}^+(\varnothing, X_1^*)$, the reasoning is the same.

For $H = G_{(i)}^+(X_0^*, X_1^*)$, the inclusion $\tau H \tau^{-1} \supseteq \Alt_{(i)}(T)^+$ implies in particular that $\tau G_{(i)}^+(X_0^*,\varnothing) \tau^{-1} \supseteq \Alt_{(i)}(T)^+$ and that $\tau G_{(i)}^+(\varnothing, X_1^*) \tau^{-1} \supseteq \Alt_{(i)}(T)^+$. By the previous work, we therefore obtain $\tau \in G_{(i)}^+(X_0^*,\varnothing) \cap G_{(i)}^+(\varnothing, X_1^*) = G_{(i)}^+(X_0^*,X_1^*)$.

Part (b) follows from (a). Indeed, the normalizer of $H$ in $\Aut(T)^+$ is contained in $n^+(H)$ by (a), and one readily checks that $n^+(H)$ normalizes $H$ for each $H \in \mathcal{G}_{(i)}$.

\item Let $H$ be one of the groups $G_{(i)}^+(X,X)$, $G_{(i)}^+(X^*,X^*)$ and $G_{(i)}^+(X,X)^*$. By (i), the normalizer of $H$ in $\Aut(T)^+$ is $n^+(H) = G_{(i)}^+(X^*,X^*)$. Consider $\nu \in \Aut(T) \setminus \Aut(T)^+$ not preserving the types but preserving the colors, i.e. such that $i \circ \nu = i$. It is clear that $\nu$ normalizes $H$, and hence the normalizer of $H$ in $\Aut(T)$ is exactly $n^+(H) \cup n^+(H) \nu = G_{(i)}^+(X^*,X^*) \cup G_{(i)}^+(X^*,X^*) \nu = G_{(i)}(X^*,X^*)$. \qedhere
\end{enumerate}
\end{proof}


\section{The classification}\label{section:classification}

Throughout this section, we let $i$ be a legal coloring of $T$ and fix $d_0,d_1 \geq 6$. Our goal is to find all groups $H \in \mathcal{H}_T^+$ satisfying $H \supseteq \Alt_{(i)}(T)^+$. Our strategy consists in first observing the groups $H \in \mathcal{H}_T^+$ with this property and in defining some invariants (namely $c(t)$, $K(t)$ and $f^t_v$). We will then see that these invariants form a complete set of invariants, i.e. that they completely characterize the group $H$. This is the subject of Theorem~\ref{maintheorem:completeinvariants'}, which is the precise formulation of Theorem~\ref{maintheorem:completeinvariants} mentioned in the introduction. The idea is then simple: compute these invariants for the groups in $\underline{\mathcal{G}}_{(i)}$ and prove that these are the only values that our invariants can take. This task will turn out to be lengthy and technical because of the algebraic invariants $f^t_v$ which are not easy to manipulate.


\subsection{Evens and odds diagrams}

Let us first fix some $v \in V(T)$ and $k \in \N$. The colored rooted tree $B(v,k)$ where each vertex is additionally labelled by $e$ or $o$ is called a \textbf{diagram supported by $B(v,k)$}. We write $\Diag_{v,k}$ for the set of all these diagrams. Remark that $|\Diag_{v,k}| = 2^{|B(v,k)|}$. There is now a natural way to define the surjective map
$$\mathcal{D} \colon \Aut(B(v,k+1)) \to \Diag_{v,k}$$
where $B(v,k+1)$ is seen as a colored rooted tree. Indeed, given $\tilde{h} \in \Aut(B(v,k+1))$ we can define $\mathcal{D}(\tilde{h})$ (which we call the \textbf{diagram of $\tilde{h}$}) to be the rooted tree $B(v,k)$ where each vertex $w$ is labelled by the parity ($e$ for \textit{even} or $o$ for \textit{odd}) of $\sigma_{(i)}(\tilde{h}, w)$. We highlight the fact that $\mathcal{D}$ associates a diagram supported by $B(v,k)$ to an automorphism of the larger ball $B(v,k+1)$. In this section, we will often deal with such diagrams. For this reason, the next lemma must be well understood.

\begin{lemma}\label{lemma:diag}
Let $\tilde{g}, \tilde{h} \in \Aut(B(v,k+1))$ and let $w$ be a vertex of $B(v,k)$.
\begin{itemize}
\item The label of $w$ in $\mathcal{D}(\tilde{g}\tilde{h})$ is $e$ if and only if the label of $w$ in $\mathcal{D}(\tilde{h})$ and the label of $\tilde{h}(w)$ in $\mathcal{D}(\tilde{g})$ are identical.
\item The label of $w$ in $\mathcal{D}(\tilde{g}^{-1})$ is equal to the label of $\tilde{g}^{-1}(w)$ in $\mathcal{D}(\tilde{g})$.
\end{itemize}
\end{lemma}

\begin{proof}
This is a corollary of Lemma~\ref{lemma:sigma}.
\end{proof}

We now fix $H \in \mathcal{H}_T^+$ such that $H \supseteq \Alt_{(i)}(T)^+$ and denote by $\tilde{H}^k(v)$ the natural image of $H(v)$ in $\Aut(B(v,k+1))$. Since $H$ is closed in $\Aut(T)$ and generated by its vertex stabilizers, it is entirely described by the groups $\tilde{H}^k(v)$ with $v \in V(T)$ and $k \in \N$. The next lemma shows that knowing the diagrams of elements of $\tilde{H}^k(v)$, i.e. $\mathcal{D}(\tilde{H}^k(v))$, suffices to fully know $\tilde{H}^k(v)$.


\begin{lemma}\label{lemma:diagram}
We have $\tilde{H}^k(v) = \mathcal{D}^{-1}(\mathcal{D}(\tilde{H}^k(v)))$.
\end{lemma}

\begin{proof}
Take $\tilde{h} \in \tilde{H}^k(v)$ and $\tilde{g} \in \Aut(B(v, k+1))$ such that $\mathcal{D}(\tilde{g}) = \mathcal{D}(\tilde{h})$. We want to show that $\tilde{g} \in \tilde{H}^k(v)$. As $\mathcal{D}(\tilde{g}) = \mathcal{D}(\tilde{h})$, Lemma~\ref{lemma:diag} directly implies that all the vertices of $\mathcal{D}(\tilde{g}\tilde{h}^{-1})$ are labelled by $e$, i.e. $\tilde{g}\tilde{h}^{-1}$ is an element of $\Alt_{(i)}(B(v,k+1))$. Since $H \supseteq \Alt_{(i)}(T)^+$, we have $\tilde{H}^k(v) \supseteq \Alt_{(i)}(B(v,k+1))$ and hence $\tilde{g} = (\tilde{g}\tilde{h}^{-1})\tilde{h} \in \tilde{H}^k(v)$.
\end{proof}

In view of the previous lemma, we only need to describe $\mathcal{D}(\tilde{H}^k(v))$ to entirely describe $\tilde{H}^k(v)$. We are first interested in the diagrams of $\mathcal{D}(\tilde{H}^k(v))$ where all the vertices of $B(v, k-1)$ are labelled by $e$. Let us call \textbf{$e$-diagram} a diagram in $\Diag_{v,k}$ with this property, and remark that $\tilde{g} \in \Aut(B(v,k+1))$ is such that $\mathcal{D}(\tilde{g})$ is an $e$-diagram if and only if $\tilde{g}|_{B(v,k)} \in \Alt_{(i)}(B(v,k))$. We denote by $\tilde{H}^k(v)_e$ the subgroup of $\tilde{H}^k(v)$ consisting of elements whose diagram is an $e$-diagram. If $\delta \in \Diag_{v,k}$ and if $w$ is a vertex of $\delta$ then the subtree of $\delta$ spanned by $w$ and all its descendants is called the \textbf{branch of $w$}. For $0 \leq r \leq k$, an \textbf{$r$-branch} of $\delta$ is a branch of a vertex at distance $k-r$ from $v$. The only $k$-branch is therefore the full tree $\delta$ and the $0$-branches all consist of a single leaf of $\delta$.


\begin{lemma}\label{lemma:alpha}
Let $v \in V(T)$ and $k \in \N$. Exactly one of the following holds.
\begin{enumerate}
\item $\mathcal{D}(\tilde{H}^k(v)_e)$ contains all the $e$-diagrams.
\item There exists $0 \leq r \leq k$ such that $\mathcal{D}(\tilde{H}^k(v)_e)$ exactly contains the $e$-diagrams with an even number of labels $o$ in each $r$-branch.
\item $\mathcal{D}(\tilde{H}^k(v)_e)$ exactly contains the $e$-diagrams with an even number of labels $o$ in each $(k-1)$-branch and the $e$-diagrams with an odd number of labels $o$ in each $(k-1)$-branch. (This case only occurs if $k \geq 1$.)
\end{enumerate}
\end{lemma}

\begin{proof}
For each $e$-diagram $\delta$, label each branch of $\delta$ with $E$ or $O$ depending on whether it contains an even or an odd number of vertices labelled by $o$. Denote by $\mathcal{D}_s$ the set of $e$-diagrams whose $s$-branches are all labelled by $E$. By definition, $\mathcal{D}(\Alt_{(i)}(B(v,k+1))) = \mathcal{D}_0 \subset \mathcal{D}_1 \subset \cdots \subset \mathcal{D}_k$. Since $\tilde{H}^k(v) \supseteq \Alt_{(i)}(B(v,k+1))$, we have $\mathcal{D}(\tilde{H}^k(v)_e) \supseteq \mathcal{D}_0$.

\begin{claim}\label{claim:branch}
Let $0 \leq s \leq k-1$. If $\mathcal{D}(\tilde{H}^k(v)_e) \supseteq \mathcal{D}_s$, then either $\mathcal{D}(\tilde{H}^k(v)_e) \supseteq \mathcal{D}_{s+1}$ or for every diagram $\delta \in \mathcal{D}(\tilde{H}^k(v)_e)$ and every $(s+1)$-branch $b$ of $\delta$, all the $s$-branches in $b$ have the same label.
\end{claim}

\begin{claimproof}
Suppose there exist a diagram $\mathcal{D}(\tilde{h}) \in \mathcal{D}(\tilde{H}^k(v)_e)$ and an $(s+1)$-branch $b$ of $\mathcal{D}(\tilde{h})$ containing both an $s$-branch $b_1$ labelled by $E$ and an $s$-branch $b_2$ labelled by $O$. Let $b_3$ and $b_4$ be two other $s$-branches in $b$ with the same label (such branches exist because $d_0,d_1 \geq 6$). Consider $\tilde{\tau} \in \Alt_{(i)}(B(v,k+1)) \subseteq \tilde{H}^k(v)$ an element interchanging $b_1$ and $b_2$, interchanging $b_3$ and $b_4$, and stabilizing all the other $s$-branches. In this way, $\tilde{h}\tilde{\tau}\tilde{h}^{-1} \in \tilde{H}^k(v)$ is such that the only $s$-branches of $\mathcal{D}(\tilde{h}\tilde{\tau}\tilde{h}^{-1})$ labelled by $O$ are $b_1$ and $b_2$ (see Lemma~\ref{lemma:diag}). Conjugating this element by adequate elements of $\Alt_{(i)}(B(v,k+1))$ and combining them, we deduce (once again by using Lemma~\ref{lemma:diag}) that $\mathcal{D}(\tilde{H}^k(v)_e)$ contains all the $e$-diagrams where each $(s+1)$-branch contains an even number of $s$-branches labelled by $O$. These are exactly the $e$-diagrams with each $(s+1)$-branch labelled by $E$, so $\mathcal{D}(\tilde{H}^k(v)_e) \supseteq \mathcal{D}_{s+1}$.
\end{claimproof}

\begin{claim}\label{claim:branch2}
Let $0 \leq s \leq k-2$. If $\mathcal{D}(\tilde{H}^k(v)_e) \supseteq \mathcal{D}_s$ but $\mathcal{D}(\tilde{H}^k(v)_e) \not\supseteq \mathcal{D}_{s+1}$, then $\mathcal{D}(\tilde{H}^k(v)_e) = \mathcal{D}_s$.
\end{claim}

\begin{claimproof}
By Claim~\ref{claim:branch}, the fact that $\mathcal{D}(\tilde{H}^k(v)_e) \supseteq \mathcal{D}_s$ but $\mathcal{D}(\tilde{H}^k(v)_e) \not\supseteq \mathcal{D}_{s+1}$ implies that for every diagram $\delta \in \mathcal{D}(\tilde{H}^k(v)_e)$ and every $(s+1)$-branch $b$ of $\delta$, all the $s$-branches in $b$ have the same label $(*)$. In order to show that $\mathcal{D}(\tilde{H}^k(v)_e) = \mathcal{D}_s$, it suffices to prove that it is impossible to have a diagram with an $(s+1)$-branch only containing $s$-branches labelled by $O$. By contradiction, suppose there exist $\tilde{h} \in \tilde{H}^k(v)_e$ and some $(s+1)$-branch $b$ of $\mathcal{D}(\tilde{h})$ all whose $s$-branches are labelled by $O$. In view of Lemma~\ref{lemma:diagram}, we can assume that $\tilde{h}$ fixes $B(v,k)$. Let us say that $b$ is the branch of the vertex $w$. Denote by $x_1, \ldots, x_r$ the children of $w$, by $b_1, \ldots, b_r$ the corresponding $s$-branches, and by $y$ the parent of $w$ (note that $w \neq v$ since $s \leq k-2$), see Figure~\ref{picture:branch2}. Let $h \in H$ be an element whose image in $\tilde{H}^k(v)$ is $\tilde{h}$ and consider an element $g \in \Alt_{(i)}(T)^+$ that fixes $w$ and interchanges $x_1$ and $y$. Then $f = ghg^{-1} \in H$ is an element fixing $B(w,1)$. Now observe the image of $f$ in $\tilde{H}^{s+1}(w)$: it is contained in $\tilde{H}^{s+1}(w)_e$ and the branches $b_2, \ldots, b_r$ are labelled by $O$ while $b_1$ is labelled by $E$ (see Lemma~\ref{lemma:diag}). Consider an element $\tau \in \Alt_{(i)}(T)^+$ that fixes $w$ and all the vertices that are closer to $y$ than to $w$, interchanges $x_1$ and $x_2$ and interchanges $x_3$ and $x_4$. Then $f \tau f^{-1} \in H$ is an element that also fixes $w$ and all the vertices that are closer to $y$ than to $w$ and, if we look at its image in $\tilde{H}^k(v)$, it is contained in $\tilde{H}^k(v)_e$ and the branches $b_1$ and $b_2$ are labelled by $O$ while the branches $b_3, \ldots, b_r$ are labelled by $E$. This is a contradiction with $(*)$.
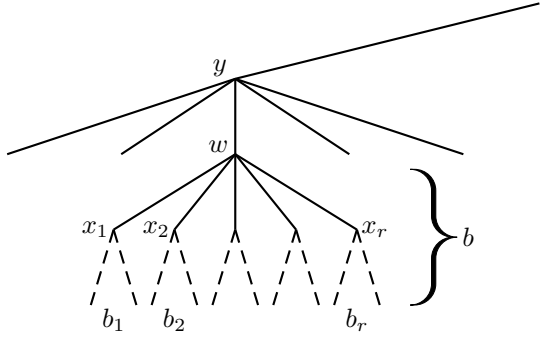
\begin{figure}
\centering
\begin{pspicture*}(-4.1,-3.4)(4.1,1.1)
\fontsize{10pt}{10pt}\selectfont
\psset{unit=1cm}

\psline(0,0)(4,1)

\psline(0,0)(-3,-1)
\psline(0,0)(-1.5,-1)
\psline(0,0)(0,-1)
\psline(0,0)(1.5,-1)
\psline(0,0)(3,-1)

\psline(0,-1)(-1.6,-2)
\psline(0,-1)(-0.8,-2)
\psline(0,-1)(0,-2)
\psline(0,-1)(0.8,-2)
\psline(0,-1)(1.6,-2)

\psline[linestyle=dashed](-1.6,-2)(-1.9,-3)
\psline[linestyle=dashed](-1.6,-2)(-1.3,-3)
\psline[linestyle=dashed](-0.8,-2)(-1.1,-3)
\psline[linestyle=dashed](-0.8,-2)(-0.5,-3)
\psline[linestyle=dashed](0,-2)(-0.3,-3)
\psline[linestyle=dashed](0,-2)(0.3,-3)
\psline[linestyle=dashed](0.8,-2)(0.5,-3)
\psline[linestyle=dashed](0.8,-2)(1.1,-3)
\psline[linestyle=dashed](1.6,-2)(1.9,-3)
\psline[linestyle=dashed](1.6,-2)(1.3,-3)

\rput(-0.2,0.15){$y$}
\rput(-0.22,-0.9){$w$}
\rput(-1.84,-2){$x_1$}
\rput(-1.04,-2){$x_2$}
\rput(1.84,-2){$x_r$}

\rput(-1.6,-3.2){$b_1$}
\rput(-0.8,-3.2){$b_2$}
\rput(1.6,-3.2){$b_r$}

\psbrace[ref=lC,nodesepA=2pt](2.3,-3)(2.3,-1.2){$b$}

\end{pspicture*}
\caption{Illustration of Lemma~\ref{lemma:alpha},  Claim~\ref{claim:branch2}.}\label{picture:branch2}
\end{figure}
\end{claimproof}

\medskip

If $\mathcal{D}(\tilde{H}^k(v)_e) \supseteq \mathcal{D}_{k}$ then there are two possibilities: either $\mathcal{D}(\tilde{H}^k(v)_e) = \mathcal{D}_{k}$ (i.e. we are in the second case with $r = k$) or there exists a diagram in $\mathcal{D}(\tilde{H}^k(v)_e)$ whose $k$-branch is labelled by $O$. In the latter case, $\mathcal{D}(\tilde{H}^k(v)_e)$ contains all the $e$-diagrams.

Suppose now that $\mathcal{D}(\tilde{H}^k(v)_e) \not\supseteq \mathcal{D}_{k}$. Then there exists $0 \leq s \leq k-1$ such that $\mathcal{D}(\tilde{H}^k(v)_e) \supseteq \mathcal{D}_s$ but $\mathcal{D}(\tilde{H}^k(v)_e) \not\supseteq \mathcal{D}_{s+1}$. If $s \neq k-1$ then by Claim~\ref{claim:branch2} we have $\mathcal{D}(\tilde{H}^k(v)_e) = \mathcal{D}_s$, i.e. we are in the second case with $r = s$. If $s = k-1$, then by Claim~\ref{claim:branch} we know that each diagram in $\mathcal{D}(\tilde{H}^k(v)_e)$ either has all its $(k-1)$-branches labelled by $E$ or all its $(k-1)$-branches labelled by $O$. If there is no diagram with labels $O$, then $\mathcal{D}(\tilde{H}^k(v)_e) = \mathcal{D}_{k-1}$ (i.e. we are in the second case with $r = k-1$). On the contrary, if there exists such a diagram, then $\mathcal{D}(\tilde{H}^k(v)_e)$ also contains the $e$-diagrams with an odd number of labels $o$ in each $(k-1)$-branch (i.e. we are in the third case).
\end{proof}


\subsection{Four possible shapes for \texorpdfstring{$H(v)$}{H(v)}}

For $v \in V(T)$ and $k \in \N$, Lemma~\ref{lemma:alpha} gives different shapes that $\mathcal{D}(\tilde{H}^k(v)_e)$ can take. We now associate a symbol $\alpha_k(v)$ to each $v$ and $k$ by defining $\alpha_k(v) = \infty$ in the first case, $\alpha_k(v) = r$ in the second case and $\alpha_k(v) = (k-1)^*$ in the third case. A natural total order on the set of symbols $\{\infty,0,0^*,1,1^*,\ldots\}$ is given by $0 < 0^* < 1 < 1^* < \cdots$ and $x < \infty$ for each $x \neq \infty$.

\begin{lemma}\label{lemma:alphaeasy}
For $x \in \{1,2,\ldots,k\}$, we have $\alpha_k(v) \geq x$ if and only if there exists a diagram in $\mathcal{D}(\tilde{H}^k(v)_e)$ with exactly two vertices labelled by $o$, situated in the same $x$-branch but in different $(x-1)$-branches.
\end{lemma}

\begin{proof}
This is a consequence of the definition of $\alpha_k(v)$.
\end{proof}

Clearly, since $\Alt_{(i)}(T)^+$ is transitive on $V_0(T)$ and $V_1(T)$, we have $\alpha_k(v) = \alpha_k(v')$ when $v$ and $v'$ have the same type. For this reason, for $t \in \{0,1\}$ we define $\alpha_k^t$ to be equal to $\alpha_k(v)$ where $v$ is a vertex of type $(t+k) \bmod 2$. In this way, $\alpha_k^t$ tells us the labels that can appear in $S(v,k)$, which is a sphere containing vertices of type $t$.

We are now interested in how the sequences $(\alpha_k^0)_{k \in \N}$ and $(\alpha_k^1)_{k \in \N}$ can look like.

\setcounter{claim}{0}


\begin{lemma}\label{lemma:sequences}
Let $t \in \{0,1\}$. Either $\alpha_k^t = \infty$ for all $k \in \N$ (case $\#$0), or there exists $K \in \N$ such that the sequence $(\alpha_k^t)_{k \in \N}$ takes one of the following three shapes. (For cases $\#$2 and $\#$3, $K$ cannot be equal to $0$.)
$$\begin{array}{c|ccccccc}
\# & \alpha_0^t & \cdots & \alpha_{K-1}^t & \alpha_K^t & \alpha_{K+1}^t & \alpha_{K+2}^t & \cdots \\
\hline
1 & \infty & \cdots & \infty & K & K & K & \cdots \\
\hline
2 & \infty & \cdots & \infty & K-1 & K-1 & K-1 & \cdots \\
\hline
3 & \infty & \cdots & \infty & (K-1)^* & K-1 & K-1 & \cdots \\
\end{array}$$
\end{lemma}

\begin{proof}
We prove this result by giving two rules that $(\alpha_k^t)_{k \in \N}$ must satisfy.

\begin{claim}\label{claim:decreasing}
The sequence $(\alpha_k^t)_{k \in \N}$ is non-increasing, i.e. $\alpha_k^t \geq \alpha_{k+1}^t$ for all $k \in \N$.
\end{claim}

\begin{claimproof}
Let $k \in \N$, let $v$ be a vertex of type $(t+k+1) \bmod 2$ and let $w$ be a vertex adjacent to $v$. Given a diagram $\delta \in \mathcal{D}(\tilde{H}^{k+1}(v)_e)$, Lemma~\ref{lemma:diagram} tells us that it is realized by an element $\tilde{h} \in \tilde{H}^{k+1}(v)_e$ which fixes $w$. Hence, $\tilde{h}$ has a natural image in $\tilde{H}^k(w)_e$ and the diagram of this image is exactly the restriction of $\delta$ to $B(w,k)$. Hence, $\mathcal{D}(\tilde{H}^{k}(w)_e)$ contains the restriction of each element of $\mathcal{D}(\tilde{H}^{k+1}(v)_e)$ to $B(w,k)$. Observing the different possibilities for $\alpha_{k+1}^t$, this always implies that $\alpha_{k}^t \geq \alpha_{k+1}^t$.
\end{claimproof}

\begin{claim}\label{claim:rule}
If $\alpha_{k}^t \geq x$ with $x \in \{0, 1, \ldots, k\}$, then $\alpha_{k+1}^t \geq x$.
\end{claim}

\begin{claimproof}
If $x = 0$ then the claim is trivial, so suppose that $x > 0$. Let $w$ be a vertex of type $(t+k) \bmod 2$. Since $\alpha_{k}^t \geq x$, there exists $h \in H(w)$ whose image in $\tilde{H}^k(w)$ has a diagram which is an $e$-diagram with exactly two vertices labelled by $o$, say $a$ and $b$, in the same $x$-branch but in different $(x-1)$-branches (see Lemma~\ref{lemma:alphaeasy}). Take $c \in S(w,k)$ a vertex in this same $x$-branch but in a third $(x-1)$-branch and $v$ a vertex adjacent to $w$ such that $a$ is closer to $w$ than to $v$ (see Figure~\ref{picture:rule}). By Lemma~\ref{lemma:diagram}, we can assume that $h$ fixes $v$. Consider $\tau \in \Alt_{(i)}(T)^+$ an element fixing all the vertices closer to $v$ than to $w$, stabilizing $a$ and interchanging $b$ and $c$. Then by Lemma~\ref{lemma:diag} the image of $h \tau h^{-1}$ in $\tilde{H}^{k+1}(v)$ has a diagram which is an $e$-diagram having exactly two vertices labelled by $o$, namely $b$ and $c$. By Lemma~\ref{lemma:alphaeasy}, this implies that $\alpha^t_{k+1} \geq x$.
\begin{figure}[b]
\centering
\begin{pspicture*}(-4.5,-3.4)(4.5,1.25)
\fontsize{10pt}{10pt}\selectfont
\psset{unit=1cm}

\psline(0,0)(4,1)

\psline(0,0)(-3,-1)
\psline(0,0)(-1.5,-1)
\psline(0,0)(0,-1)
\psline(0,0)(1.5,-1)
\psline(0,0)(3,-1)

\psline(0,-1)(-1.6,-2)
\psline(0,-1)(-0.8,-2)
\psline(0,-1)(0,-2)
\psline(0,-1)(0.8,-2)
\psline(0,-1)(1.6,-2)

\psline[linestyle=dashed](-1.6,-2)(-1.9,-3)
\psline[linestyle=dashed](-1.6,-2)(-1.3,-3)
\psline[linestyle=dashed](-0.8,-2)(-1.1,-3)
\psline[linestyle=dashed](-0.8,-2)(-0.5,-3)
\psline[linestyle=dashed](0,-2)(-0.3,-3)
\psline[linestyle=dashed](0,-2)(0.3,-3)
\psline[linestyle=dashed](0.8,-2)(0.5,-3)
\psline[linestyle=dashed](0.8,-2)(1.1,-3)
\psline[linestyle=dashed](1.6,-2)(1.9,-3)
\psline[linestyle=dashed](1.6,-2)(1.3,-3)

\rput(3.9,1.15){$v$}
\rput(-0.2,0.15){$w$}

\psdot[dotsize=3pt](-1.6,-3)
\rput(-1.6,-3.3){$a$}
\psdot[dotsize=3pt](-0.8,-3)
\rput(-0.8,-3.27){$b$}
\psdot[dotsize=3pt](0,-3)
\rput(0,-3.3){$c$}

\psbrace[ref=lC,nodesepA=2pt,nodesepB=-1pt](2.3,-3)(2.3,-1.2){$x$-branch}

\end{pspicture*}
\caption{Illustration of Lemma~\ref{lemma:sequences},  Claim~\ref{claim:rule}.}\label{picture:rule}
\end{figure}
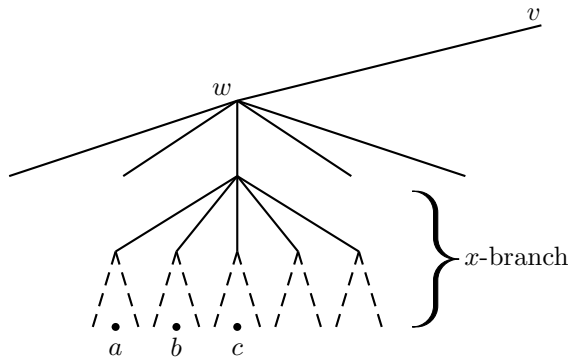
\end{claimproof}

\medskip

These two claims suffice to get the result. Indeed, we either have $\alpha_0^t = 0$ or $\alpha_0^t = \infty$. If $\alpha_0^t = 0$ then by Claim~\ref{claim:decreasing} we get case $\#$1. If $\alpha_0^t = \infty$, then either $\alpha_k^t = \infty$ for all $k \in \N$, or there exists a smallest $K$ such that $\alpha_K^t < \infty$. In the latter case, Claim~\ref{claim:rule} with $k = x = K-1$ gives $\alpha_K^t \geq K-1$, so $\alpha_K^t \in \{K-1, (K-1)^*, K\}$. If $\alpha_K^t \in \{K-1, K\}$, then the two claims imply that $(\alpha_k^t)_{k \geq K}$ is constant and we get cases $\#$1 and $\#$2. If $\alpha_K^t = (K-1)^*$, then Claim~\ref{claim:decreasing} and Claim~\ref{claim:rule} with $k = K$ and $x = K-1$ give $(K-1) \leq \alpha_{K+1}^t \leq (K-1)^*$. Since $\alpha_{K+1}^t$ is never equal to $(K-1)^*$, we must have $\alpha_{K+1}^t = K-1$ and then get the constant sequence as above, which gives case $\#$3.
\end{proof}


\subsection{The numerical invariants \texorpdfstring{$c(t)$}{c(t)} and \texorpdfstring{$K(t)$}{K(t)}}

For $t \in \{0,1\}$, denote by $c(t) \in \{0,1,2,3\}$ the case which was encountered in Lemma~\ref{lemma:sequences} and by $K(t)$ the smallest integer such that $\alpha_{K(t)}^t < \infty$, as in the lemma (if $c(t) = 0$, define $K(t) = \infty$). The limit value
$$K'(t) := \lim_{k \to \infty}\alpha_k^t$$
will also be useful for our proofs. Note that $c(t)$ and $K(t)$ completely determine $K'(t)$. Similarly, $c(t)$ and $K'(t)$ determine $K(t)$.

These invariants can be computed for each of our key examples. To simplify the notations, we define the binary operation $\boxplus \colon \N \times \N \to \N$ by
$$a \boxplus b := a + b - \left\lceil\frac{|a-b|}{2}\right\rceil.$$


\begin{proposition}\label{proposition:table}
The values of $c(0)$, $c(1)$, $K'(0)$ and $K'(1)$ for the members of $\underline{\mathcal{G}}_{(i)}$ are given in Table~\ref{table:invariants}. The last column of Table~\ref{table:invariants} gives, for fixed $c(0)$, $c(1)$, $K'(0)$ and $K'(1)$, the exact number of groups (in $\underline{\mathcal{G}}_{(i)}$) in the corresponding line.
\end{proposition}

\begin{proof}
\begin{table}[b!]
$$
\begin{array}{|c|c|c;{1pt/2pt}c|c;{1pt/2pt}c|c|}
\hline
 & & c(0) & K'(0) & c(1) & K'(1) & \text{\#} \Tstrut\Bstrut\\
\hline
1 & \Aut(T)^+ = G_{(i)}^+(\varnothing,\varnothing) & 0 & \infty & 0 & \infty & 1 \Tstrut\\
2 & G_{(i)}^+(X_0,\varnothing) & 1 & \max X_0 & 0 & \infty & 2^{K'(0)} \\
3 & G_{(i)}^+(\varnothing,X_1) & 0 & \infty & 1 & \max X_1 & 2^{K'(1)} \\
4 & G_{(i)}^+(X_0,X_1) & 1 & \max X_0 & 1 & \max X_1 & 2^{K'(0)\boxplus K'(1)} \Bstrut \\
\hline
5 & G_{(i)}^+(X_0^*,\varnothing) & 3 & \max X_0 & 0 & \infty & 2^{K'(0)} \Tstrut\\
6 & G_{(i)}^+(\varnothing,X_1^*) & 0 & \infty & 3 & \max X_1 & 2^{K'(1)} \\
7 & G_{(i)}^+(X_0,X_1^*) & 1 & \max X_0 & 3 & \max X_1 & 2^{K'(0)\boxplus K'(1)} \\
8 & G_{(i)}^+(X_0^*,X_1) & 3 & \max X_0 & 1 & \max X_1 & 2^{K'(0)\boxplus K'(1)} \\
9 & G_{(i)}^+(X_0^*,X_1^*) & 3 & \max X_0 & 3 & \max X_1 & 2^{K'(0)\boxplus K'(1)} \Bstrut\\
10 & \begin{array}{c}
G_{(i)}^+(X_0,X_1)^*\Tstrut\\
\text{($\max X_0 = \max X_1$)}
\end{array}
& 2 & \max X_0 & 2 & \max X_1 & 2^{K'(0)\boxplus K'(1)} \Tstrut\Bstrut\\
11 & \begin{array}{c}
G_{(i)}^+(X_0,X_1)^*\Tstrut\\
\text{($\max X_0 > \max X_1$)}
\end{array}
& 1 & \max X_0 & 3 & \max X_1 & 2^{K'(0)\boxplus K'(1)} \Tstrut\Bstrut\\
12 & \begin{array}{c}
G_{(i)}^+(X_0,X_1)^*\Tstrut\\
\text{($\max X_0 < \max X_1$)}
\end{array}
& 3 & \max X_0 & 1 & \max X_1 & 2^{K'(0)\boxplus K'(1)} \Tstrut\Bstrut\\
\hline
\end{array}
$$
\caption{Values of the invariants for the groups in $\underline{\mathcal{G}}_{(i)}$.}
\label{table:invariants}
\end{table}
The values of the different invariants can be computed only using the definitions of the groups and the construction explained in Subsection~\ref{subsection:definitions} with the labellings $e$ and~$o$. We suggest the reader to compute the invariants for some of the groups to become familiar with the definitions.

The value $2^{K'(0)}$ in the last column for lines 2 and 5 is simply equal to the number of sets $X_0 \subset_f \N$ such that $\max X_0 = K'(0)$. The reasoning is the same for lines 3 and~6. Concerning line 4 and lines 7--12, the value $2^{K'(0) \boxplus K'(1)}$ corresponds to the number of pairs $(X_0,X_1)$ with $X_0,X_1 \subset_f \N$ such that $X_0$ and $X_1$ are compatible (as defined in Definition~\ref{definition:classification}), $\max X_0 = K'(0)$ and $\max X_1 = K'(1)$. Note that we do not count a group twice as all the groups in $\underline{\mathcal{G}}_{(i)}$ are pairwise different (see Proposition~\ref{proposition:different}).
\end{proof}

In Table~\ref{table:invariants}, it is remarkable that having $c(0) = 2$ also implies $c(1) = 2$ and $K(0) = K(1)$. This is actually a general fact for any $H \in \mathcal{H}_T^+$ such that $H \supseteq \Alt_{(i)}(T)^+$.


\begin{lemma}\label{lemma:2-2}
If $c(t) = 2$ for some $t \in \{0,1\}$, then $c(1-t) = 2$ and $K(0) = K(1)$.
\end{lemma}

\begin{proof}
Assume without loss of generality that $t = 0$ and let $v$ be a vertex of type $(K(0)-1) \bmod 2$. Since $\alpha_{K(0)-1}^0 = \infty$, there exists $h \in H$ fixing $B(v,K(0)-1)$ and such that the diagram of its image in $\tilde{H}^{K(0)-1}(v)$ has exactly one vertex $a$ labelled by $o$. Let $w$ be a vertex adjacent to $v$ such that $a$ is not in the branch of $w$.

We first show that $\alpha^{1}_{K(0)} \leq K(0)-1$, which will in particular imply that $K(1) \leq K(0)$. Suppose for a contradiction that $\alpha^{1}_{K(0)} \geq (K(0)-1)^*$. Then there exists $g \in H$ fixing $B(v,K(0))$ and such that the diagram of its image $\tilde{g}$ in $\tilde{H}^{K(0)}(v)$ and the diagram of the image of $h$ in $\tilde{H}^{K(0)}(v)$ coincide on the branch of $w$. Indeed, the condition $\alpha^{1}_{K(0)} \geq (K(0)-1)^*$ gives us sufficient freedom to choose the labels of $\mathcal{D}(\tilde{g})$ in the branch of $w$. Hence, the diagram of the image of $hg^{-1}$ in $\tilde{H}^{K(0)}(w)$ is an $e$-diagram with a $(K(0)-1)$-branch (the branch of $v$) containing exactly one vertex labelled by $o$, contradicting $\alpha_{K(0)}^0 = K(0)-1$.

We now prove that $K(1) \geq K(0)$, once again by contradiction, assuming that $K(1) < K(0)$. If $\tilde{h}$ is the image of $h$ in $\tilde{H}^{K(0)}(v)$, then since $\alpha^1_{K(1)} \in \{K'(1), K'(1)^*\}$ the $K'(1)$-branches of $\mathcal{D}(\tilde{h})$ contained in the branch of $w$ all contain an even number of vertices labelled by $o$. But $\alpha_{K(0)}^1 = K'(1)$ (because $K(0) > K(1)$), so there exists $g \in H$ fixing $B(v,K(0))$ and such that the diagram of its image $\tilde{g}$ in $\tilde{H}^{K(0)}(v)$ and the diagram of $\tilde{h}$ coincide on the branch of $w$. We therefore have the same contradiction as above by considering the image of $hg^{-1}$ in $\tilde{H}^{K(0)}(w)$.

As a conclusion, $K(1) = K(0)$ and $\alpha_{K(1)}^1 \leq K(1)-1$ so $c(1) = 2$. 
\end{proof}


\subsection{The algebraic invariants \texorpdfstring{$f^t_v$}{f^t_v}}

Our next goal is to understand the relationship between $\mathcal{D}(\tilde{H}^{k-1}(v))$ and $\mathcal{D}(\tilde{H}^k(v))$ (for fixed $v$ and $k$). The first result in this direction is the following. We identify the group $\C_2$ of order $2$ with $\{E,O\}$, where $E$ is the neutral element. By convention, we say that $B(v,-1) = \varnothing$ and that $\mathcal{D}(\tilde{H}^{-1}(v))$ only contains the empty diagram $\varepsilon$.

\begin{lemma}\label{lemma:crucial}
Let $v \in V(T)$, let $k \in \N$ and let $\delta \in \mathcal{D}(\tilde{H}^{k-1}(v))$.
\begin{enumerate}[(i)]
\item If $\alpha_k(v) = \infty$, then $\mathcal{D}(\tilde{H}^k(v))$ contains all the diagrams of $\Diag_{v,k}$ whose intersection with $B(v,k-1)$ is $\delta$.
\item If $\alpha_k(v) = x \in \{0,1,\ldots,k\}$, denote by $b_1,\ldots,b_m$ the $x$-branches of $B(v,k)$. Then there exists a unique element $(p_1,\ldots,p_m) \in \{E,O\}^m$ such that the following holds. For each $\hat{\delta} \in \Diag_{v,k}$ with $\hat{\delta} \cap B(v,k-1) = \delta$, if $q_i \in \{E,O\}$ is the parity of the number of vertices labelled by $o$ in $\hat{\delta} \cap b_i \cap S(v,k)$, then $\hat{\delta}$ is contained in $\mathcal{D}(\tilde{H}^k(v))$ if and only if $(q_1,\ldots,q_m) = (p_1,\ldots,p_m)$.
\item If $\alpha_k(v) = (k-1)^*$, denote by $b_1,\ldots,b_m$ the $(k-1)$-branches of $B(v,k)$. Then there exists a unique element $[(p_1,\ldots,p_m)] \in \bigslant{\{E,O\}^m}{\langle (O,\ldots,O) \rangle}$ such that the following holds. For each $\hat{\delta} \in \Diag_{v,k}$ with $\hat{\delta} \cap B(v,k-1) = \delta$, if $q_i \in \{E,O\}$ is the parity of the number of vertices labelled by $o$ in $\hat{\delta} \cap b_i \cap S(v,k)$, then $\hat{\delta}$ is contained in $\mathcal{D}(\tilde{H}^k(v))$ if and only if $[(q_1,\ldots,q_m)] = [(p_1,\ldots,p_m)]$.
\end{enumerate}
\end{lemma}

\begin{proof}
Let $\tilde{h} \in \tilde{H}^k(v)$ be such that $\mathcal{D}(\tilde{h}) \cap B(v,k-1) = \delta$.
\begin{enumerate}[(i)]
\item Let $\tilde{g} \in \Aut(B(v,k+1))$ be such that $\mathcal{D}(\tilde{g}) \cap B(v,k-1) = \delta$. Using Lemma~\ref{lemma:diag}, we see that having $\mathcal{D}(\tilde{g}) \cap B(v,k-1) = \mathcal{D}(\tilde{h}) \cap B(v,k-1)$ implies that $\mathcal{D}(\tilde{g}\tilde{h}^{-1})$ is an $e$-diagram. As $\alpha_k(v) = \infty$, we get $\tilde{g}\tilde{h}^{-1} \in \tilde{H}^k(v)$ and thus $\tilde{g} = (\tilde{g}\tilde{h}^{-1})\tilde{h} \in \tilde{H}^k(v)$.
\item[(ii),(iii)] For each $i \in \{1,\ldots,m\}$, let $p_i \in \{E,O\}$ be the parity of the number of vertices labelled by $o$ in $\mathcal{D}(\tilde{h}) \cap b_i \cap S(v,k)$. We prove that $(p_1,\ldots,p_m)$ (resp. $[(p_1,\ldots,p_m)]$) satisfies the statement (and it is clear that it is unique). Let $\tilde{g} \in \Aut(B(v,k+1))$ be such that $\mathcal{D}(\tilde{g}) \cap B(v,k-1) = \delta$ and let $q_i \in \{E,O\}$ be the parity of the number of vertices labelled by $o$ in $\mathcal{D}(\tilde{g}) \cap b_i \cap S(v,k)$. We have $\tilde{g} \in \tilde{H}^k(v)$ if and only if $\tilde{g} \tilde{h}^{-1} \in \tilde{H}^k(v)$, and $\mathcal{D}(\tilde{g} \tilde{h}^{-1})$ is an $e$-diagram. The value of $\alpha_k(v)$ and Lemma~\ref{lemma:diag} then imply that $\tilde{g} \tilde{h}^{-1} \in \tilde{H}^k(v)_e$ if and only if $(q_1,\ldots,q_m) = (p_1,\ldots,p_m)$ (resp. $[(q_1,\ldots,q_m)] = [(p_1,\ldots,p_m)]$). \qedhere
\end{enumerate}
\end{proof}

Fix $t \in \{0,1\}$ such that $c(t) \neq 0$. For $k < K(t)$, if $v$ is a vertex of type $(t+k) \bmod 2$ then the fact that $\alpha_k^t = \infty$ implies by Lemma~\ref{lemma:crucial} (i) that $\mathcal{D}(\tilde{H}^k(v))$ exactly contains the diagrams of $\Diag_{v,k}$ whose intersection with $B(v,k-1)$ is a diagram in $\mathcal{D}(\tilde{H}^{k-1}(v))$.

On the other hand, if $v$ is a vertex of type $(t+K(t)) \bmod 2$ then the shape of $\mathcal{D}(\tilde{H}^{K(t)}(v))$ cannot be directly deduced from $\mathcal{D}(\tilde{H}^{K(t)-1}(v))$. In view of Lemma~\ref{lemma:crucial} (ii),(iii), we can however define a map $f^t_v$ to encode this information. The domain of $f^t_v$ will be $\mathcal{D}(\tilde{H}^{K(t)-1}(v))$ while its codomain $J^t$ will depend on the value of $c(t)$. Given a diagram $\delta \in \mathcal{D}(\tilde{H}^{K(t)-1}(v))$, the value of $f^t_v(\delta)$ will exactly give what is the condition on a diagram of $\Diag_{v,K(t)}$ whose intersection with $B(v,K(t)-1)$ is $\delta$ for being contained in $\mathcal{D}(\tilde{H}^{K(t)}(v))$. Let us denote by $b_1, \ldots, b_{\tilde{d}}$ the branches of the vertices adjacent to $v$.

\begin{itemize}
\item If $c(t) = 1$, then $\alpha^t_{K(t)} = K(t)$ and we can apply Lemma~\ref{lemma:crucial} (ii) with $k = x = K(t)$. We set $J^t := \{E, O\}$ and define $f^t_v \colon \mathcal{D}(\tilde{H}^{K(t)-1}(v)) \to J^t$ naturally: $f^t_v(\delta)$ is the unique element $p \in J^t$ given by the lemma (note that $m = 1$).

\item If $c(t) = 2$, then $\alpha^t_{K(t)} = K(t)-1$ and we can apply Lemma~\ref{lemma:crucial} (ii) with $k = K(t)$ and $x = K(t)-1$. We set $J^t := \{E, O\}^{\tilde{d}}$ and define $f^t_v \colon \mathcal{D}(\tilde{H}^{K(t)-1}(v)) \to J^t$ naturally: $f^t_v(\delta)$ is the unique element $(p_1,\ldots,p_{\tilde{d}}) \in J_t$ given by the lemma.

\item If $c(t) = 3$, then $\alpha^t_{K(t)} = (K(t)-1)^*$ and we can apply Lemma~\ref{lemma:crucial} (iii) with $k = K(t)$. We set $J^t := \bigslant{\{E, O\}^{\tilde{d}}}{\langle(O,\ldots,O)\rangle}$ and define $f^t_v \colon \mathcal{D}(\tilde{H}^{K(t)-1}(v)) \to J^t$ naturally: $f^t_v(\delta)$ is the unique element $[(p_1,\ldots,p_{\tilde{d}})] \in J_t$ given by the lemma.
\end{itemize}

The next result directly follows from the definition of $f^t_v$.


\begin{lemma}\label{lemma:ftrivial}
Let $\delta_e \in \Diag_{v,K(t)-1}$ be the diagram with all vertices labelled by $e$. Then $\delta_e$ belongs to $\mathcal{D}(\tilde{H}^{K(t)-1}(v))$ and $f^t_v(\delta_e)$ is the trivial element of $J^t$. Moreover, for $\tilde{g}, \tilde{h} \in \tilde{H}^{K(t)-1}(v)$, we have the following.
\begin{itemize}
\item If $c(t) = 1$, then $f^t_v(\mathcal{D}(\tilde{g}\tilde{h})) = f^t_v(\mathcal{D}(\tilde{g})) \cdot f^t_v(\mathcal{D}(\tilde{h}))$.
\item If $c(t) \in \{2,3\}$, then $f^t_v(\mathcal{D}(\tilde{g}\tilde{h})) = \sigma(f^t_v(\mathcal{D}(\tilde{g}))) \cdot f^t_v(\mathcal{D}(\tilde{h}))$, where $\sigma \colon J^t \to J^t$ permutes the coordinates in the same way as $\tilde{h}$ permutes the branches $b_1, \ldots, b_{\tilde{d}}$.
\end{itemize}
\end{lemma}

\begin{proof}
For each $k \geq 0$, the diagram in $\Diag_{v,k}$ with all vertices labelled by $e$ is always contained in $\mathcal{D}(\tilde{H}^k(v))$ (because $H \supseteq \Alt_{(i)}(T)^+$). In particular, we have $\delta_e \in \mathcal{D}(\tilde{H}^{K(t)-1}(v))$ and $f^t_v(\delta_e)$ must be the trivial element of $J_t$.

The formula for $f^t_v(\mathcal{D}(\tilde{g}\tilde{h}))$ can be directly obtain by using Lemma~\ref{lemma:diag}.
\end{proof}


\subsection{The invariants \texorpdfstring{$c(t)$}{c(t)}, \texorpdfstring{$K(t)$}{K(t)} and \texorpdfstring{$f^t_v$}{f^t_v} form a complete system}

By definition, the map $f^t_v$ defined above fully describes the shape of $\mathcal{D}(\tilde{H}^{K(t)}(v))$ from $\mathcal{D}(\tilde{H}^{K(t)-1}(v))$. A priori, it seems that we also need similar maps to deduce the shape of $\mathcal{D}(\tilde{H}^k(v))$ from $\mathcal{D}(\tilde{H}^{k-1}(v))$ for each $k > K(t)$ (where $v$ is of type $(t+k) \bmod 2$). However, as the following lemma shows, this is not the case.

\begin{lemma}\label{lemma:describe}
Let $t \in \{0,1\}$ be such that $c(t) \neq 0$, let $k > K(t)$ and let $v$ be a vertex of type $(t+k) \bmod 2$. Let also $\delta \in \mathcal{D}(\tilde{H}^{k-1}(v))$. Consider $\hat{\delta} \in \Diag_{v,k}$ with $\hat{\delta} \cap B(v,k-1) = \delta$. Then $\hat{\delta}$ belongs to $\mathcal{D}(\tilde{H}^k(v))$ if and only if, for each vertex $w$ at distance $k-K(t)$ from $v$, the diagram $\hat{\delta} \cap B(w,K(t))$ belongs to $\mathcal{D}(\tilde{H}^{K(t)}(w))$.
\end{lemma}

\begin{proof}
If $\hat{\delta} \in \mathcal{D}(\tilde{H}^k(v))$ and $w$ is a vertex at distance $k-K(t)$ from $v$, then by Lemma~\ref{lemma:diagram} the diagram $\hat{\delta}$ is realized by an element $\tilde{h}$ of $\tilde{H}^k(v)$ fixing $w$. Hence, the diagram of $\tilde{h}|_{B(w,K(t)+1)}$, which is $\hat{\delta} \cap B(w,K(t))$, is contained in $\mathcal{D}(\tilde{H}^{K(t)}(w))$.

Now take $\hat{\delta} \in \Diag_{v,k}$ with $\hat{\delta} \cap B(v,k-1) = \delta$ such that $\hat{\delta} \cap B(w,K(t)) \in \mathcal{D}(\tilde{H}^{K(t)}(w))$ for each vertex $w$ at distance $k-K(t)$ from $v$. Consider also $\hat{\delta}' \in \mathcal{D}(\tilde{H}^k(v))$ with $\hat{\delta}' \cap B(v,k-1) = \delta$. In view of the first part of the proof, we have $\hat{\delta}' \cap B(w,K(t)) \in \mathcal{D}(\tilde{H}^{K(t)}(w))$ for each $w$ at distance $k-K(t)$ from $v$. Denote by $b_1,\ldots,b_m$ the $K'(t)$-branches of $B(v,k)$, and let $p_i$ (resp. $p'_i$) be the parity of the number of vertices labelled by $o$ in $\hat{\delta} \cap b_i \cap S(v,k)$ (resp. in $\hat{\delta}' \cap b_i \cap S(v,k)$). In view of Lemma~\ref{lemma:crucial} (ii), it suffices to show that $(p_1,\ldots,p_m) = (p'_1,\ldots,p'_m)$ in order to prove that $\hat{\delta} \in \mathcal{D}(\tilde{H}^k(v))$. Let $j \in \{1,\ldots,m\}$ and let $w$ be the vertex at distance $k-K(t)$ from $v$ whose branch $b$ contains $b_j$. The diagrams $\delta_0 = \hat{\delta} \cap B(w,K(t))$ and $\delta'_0 = \hat{\delta}' \cap B(w,K(t))$, which both belong to $\mathcal{D}(\tilde{H}^{K(t)}(v))$, coincide on $B(w,K(t)-1) \cup (S(w,K(t)) \setminus b)$. In particular, we have $\delta_0 \cap B(w,K(t)-1) = \delta'_0 \cap B(w,K(t)-1)$ and the two diagrams must therefore satisfy the condition given by $f^t_w(\delta_0 \cap B(w,K(t)-1)) \in J^t$ $(*)$. Given a part $X$ of a diagram, let us write $P(X)$ for the parity of the number of vertices labelled by $o$ in $X$.

If $c(t) = 1$ then $b = b_j$ and $(*)$ means that $P(\delta_0 \cap S(w,K(t))) = P(\delta'_0 \cap S(w,K(t)))$. Since $\delta_0$ and $\delta'_0$ coincide on $S(w,K(t)) \setminus b$, this means that $P(\delta_0 \cap S(w,K(t)) \cap b) = P(\delta'_0 \cap S(w,K(t)) \cap b)$, i.e. we have $p_j = p'_j$.
If $c(t) \in \{2,3\}$ then let $\tilde{b}_1,\ldots,\tilde{b}_{\tilde{d}}$ be the branches (seen in $B(w,K(t))$) of the vertices adjacent to $w$. One of these branches is equal to $b_j$, say $\tilde{b}_1$, and another of these branches is the branch of the parent of $w$ (in $B(v,k)$), say $\tilde{b}_2$. If $c(t) = 2$ then $(*)$ means that $P(\delta_0 \cap \tilde{b}_i \cap S(w,K(t))) = P(\delta'_0 \cap \tilde{b}_i \cap S(w,K(t)))$ for each $i \in \{1,\ldots,\tilde{d}\}$, and $i = 1$ directly gives $p_j = p'_j$. If $c(t) = 3$, then $(*)$ means that either $P(\delta_0 \cap \tilde{b}_i \cap S(w,K(t))) = P(\delta'_0 \cap \tilde{b}_i \cap S(w,K(t)))$ for each $i$ or $P(\delta_0 \cap \tilde{b}_i \cap S(w,K(t))) \neq P(\delta'_0 \cap \tilde{b}_i \cap S(w,K(t)))$ for each $i$. But $\delta_0$ and $\delta'_0$ coincide on $S(w,K(t)) \setminus b = S(w,K(t)) \cap \tilde{b}_2$, so we must have the equality for each $i$. In particular, $i = 1$ gives $p_j = p'_j$.
\end{proof}

As a consequence of the previous lemma, we find that the invariants $c(t)$, $K(t)$ and $f^t_v$ (for $t \in \{0,1\}$ and $v \in V(T)$ such that $f^t_v$ is defined) fully describe the entire group $H$. Note that, since $\Alt_{(i)}(T)^+$ is transitive on $V_0(T)$ and $V_1(T)$, if $c(t) \neq 0$ then knowing $f^t_v$ for a fixed vertex $v$ of type $(t+K(t)) \bmod 2$ suffices to get each $f^t_w$.


\begin{primetheorem}{maintheorem:completeinvariants}\label{maintheorem:completeinvariants'}
If $H, H' \in \mathcal{H}_T^+$ satisfy $H, H' \supseteq \Alt_{(i)}(T)^+$ and have the same invariants $c(t)$, $K(t)$ and $f^t_v$ (for $t \in \{0,1\}$ and $v \in V(T)$ such that $f^t_v$ is defined), then $H = H'$.
\end{primetheorem}

\begin{proof}
We fix one group $H \in \mathcal{H}_T^+$ with $H \supseteq \Alt_{(i)}(T)^+$ and show that, for each $v \in V(T)$ and $k \in \N$, the set $\mathcal{D}(\tilde{H}^k(v))$ can be described only using the invariants $c(t)$, $K(t)$ and $f^t_v$. By Lemma~\ref{lemma:diagram} and the fact that $H$ is generated by point stabilizers, this will prove the statement.

Let us do it by induction on $k$. For $k = -1$, $\mathcal{D}(\tilde{H}^{-1}(v))$ only contains the empty diagram $\varepsilon$ for each $v \in V(T)$. Now fix $k \geq 0$ and assume that $\mathcal{D}(\tilde{H}^{k-1}(v))$ is known for each $v \in V(T)$. Let $v \in V(T)$ and define $t \in \{0,1\}$ to be such that $v$ is of type $(t+k) \bmod 2$. If $k < K(t)$, then $\alpha^t_k = \infty$ and we know that $\mathcal{D}(\tilde{H}^k(v))$ exactly contains the diagrams of $\Diag_{v,k}$ whose intersection with $B(v,k-1)$ is contained in $\mathcal{D}(\tilde{H}^{k-1}(v))$ (see Lemma~\ref{lemma:crucial} (i)). If $k = K(t)$, then $f^t_v$ fully describes $\mathcal{D}(\tilde{H}^{K(t)}(v))$ from $\mathcal{D}(\tilde{H}^{K(t)-1}(v))$. Finally, if $k > K(t)$, then Lemma~\ref{lemma:describe} shows that $\mathcal{D}(\tilde{H}^{k}(v))$ can be deduced from $\mathcal{D}(\tilde{H}^{k-1}(v))$ and each $\mathcal{D}(\tilde{H}^{K(t)}(w))$ with $w$ at distance $k-K(t)$ from $v$.
\end{proof}

Theorem~\ref{maintheorem:completeinvariants} formulated in the introduction is a consequence of Theorem~\ref{maintheorem:completeinvariants'}. Indeed, for $H \in \mathcal{H}_T^+$ satisfying $H \supseteq \Alt_{(i)}(T)^+$, take $K \in \Nz$ strictly greater than the finite numbers among $\{K(0),K(1)\}$. Then the $K$-closure $H^{(K)}$ of $H$ also satisfies $H^{(K)} \in \mathcal{H}_T^+$ and $H^{(K)} \supseteq \Alt_{(i)}(T)^+$ and have the same invariants as $H$, so that $H = H^{(K)}$.


\subsection{Possible shapes for \texorpdfstring{$f^t_v$}{f^t_v} when \texorpdfstring{$K(t) \leq K(1-t)$}{K(t) <= K(1-t)}}

We now observe which shapes $f^t_v$ can take when $c(t) \neq 0$ and $K(t) \leq K(1-t)$. Recall that $S_X(v) := \bigcup_{r \in X}S(v,r)$ when $X \subseteq \N$.


\begin{lemma}\label{lemma:forms}
Suppose that $c(t) \neq 0$ and $K(t) \leq K(1-t)$ and let $v$ be a vertex of type $(t+K(t))\bmod 2$. Then the possible shapes for $f^t_v$ are given as follows. Here, $b_1, \ldots, b_{\tilde{d}}$ denote the branches of the vertices adjacent to $v$, as in the definition of $J^t$.
\begin{itemize}
\item If $c(t) = 1$, then there exists $A \subseteq \{0, 1, \ldots, K(t)-1\}$ such that $f^t_v(\delta)$ is equal to the parity of the number of vertices labelled by $o$ in $\delta \cap S_A(v)$.
\item If $c(t) = 2$, then there exist $A \subseteq \{1, 2, \ldots, K(t)-1\}$ and $B \subseteq \{0, 1, \ldots, K(t)-1\}$ such that $f^t_v(\delta) = (p_1, \ldots, p_{\tilde{d}})$ where $p_i$ is the parity of the number of vertices labelled by $o$ in $\delta \cap \left((S_A(v) \cap b_i) \cup (S_B(v)\setminus b_i)\right)$.
\item If $c(t) = 3$, then there exists $A \subseteq \{1, 2, \ldots, K(t)-1\}$ such that $f^t_v(\delta) = [(p_1, \ldots, p_{\tilde{d}})]$ where $p_i$ is the parity of the number of vertices labelled by $o$ in $\delta \cap (S_A(v) \cap b_i)$.
\end{itemize}
\end{lemma}

\begin{proof}
Since $K(t) \leq K(1-t)$, we have $\alpha_k^t = \alpha_k^{1-t} = \infty$ for each $k < K(t)$ and hence $\mathcal{D}(\tilde{H}^{K(t)-1}(v)) = \Diag_{v,K(t)-1}$. Now, for each $r \in \{0, 1, \ldots, K(t)-1\}$, fix a diagram $\delta_r \in \mathcal{D}(\tilde{H}^{K(t)-1}(v))$ having exactly one vertex $w$ labelled by $o$, with $w \in S(v,r)$ and, if $r \geq 1$, $w \in b_1$. In view of Lemma~\ref{lemma:ftrivial}, it is clear that the image of any diagram by $f^t_v$ can always be recovered from the values that $f^t_v$ takes on $\{\delta_0, \delta_1, \ldots, \delta_{K(t)-1}\}$.
\begin{itemize}
\item If $c(t) = 1$, then define $A = \{r \in \{0, 1, \ldots, K(t)-1\} \suchthat f^t_v(\delta_r) = O\}$. Then $f^t_v$ is exactly of the shape given in the statement.
\item If $c(t) = 2$, then for each $r \in \{0, \ldots, K(t)-1\}$ we have $f^t_v(\delta_r) = (p^r_1, \ldots, p^r_{\tilde{d}})$ for some $p^r_i \in \{E, O\}$. For $r \geq 1$, considering elements $\tilde{g} \in \Alt_{(i)}(B(v, K(t)))$ stabilizing the branch $b_1$ but permuting the other branches, we directly obtain using Lemma~\ref{lemma:ftrivial} that $p^r_2 = p^r_3 = \cdots = p^r_{\tilde{d}}$. We can therefore write $f^t_v(\delta_r) = (x_r, y_r, \ldots, y_r)$ with $x_r, y_r \in \{E, O\}$. For $r = 0$, we obtain in the same way that $p^0_1 = p^0_2 = \cdots = p^0_{\tilde{d}}$, and we write $f^t_v(\delta_0) = (y_0, \ldots, y_0)$. Now if we define $A = \{r \in \{1, 2, \ldots, K(t)-1\} \suchthat x_r = O\}$ and $B = \{r \in \{0, 1, \ldots, K(t)-1\} \suchthat y_r = O\}$, then we exactly get the shape given in the statement.
\item If $c(t) = 3$, then the same reasoning as in the previous case works but it must be remembered that the values are taken modulo $(O,\ldots,O)$. We thus get $f^t_v(\delta_r) = [(x_r, y_r, \ldots, y_r)]$ for $r \geq 1$ and $f^t_v(\delta_0) = [(y_0, \ldots, y_0)]$, but it can be assumed that all the $y_r$ are equal to $E$. Defining $A = \{r \in \{1, 2, \ldots, K(t)-1\} \suchthat x_r = O\}$, we obtain the shape given in the statement.\qedhere
\end{itemize}
\end{proof}

When $c(t) = 2$, we also have $c(1-t) = 2$ and $K(0) = K(1)$ by Lemma~\ref{lemma:2-2}. In this case, Lemma~\ref{lemma:forms} can be applied with $t = 0$ and $t = 1$. It is however important to note the following result. Remark that, as $\Alt_{(i)}(T)^+$ is transitive on $V_0(T)$ and $V_1(T)$, the sets $A$ and $B$ given by Lemma~\ref{lemma:forms} depend on $t \in \{0,1\}$ but not on $v$.


\begin{lemma}\label{lemma:2-2:AB}
Suppose that $c(0) = c(1) = 2$ and $K(0) = K(1) =: K$. For each $t \in \{0,1\}$, let $A_t$ and $B_t$ be the sets given by Lemma~\ref{lemma:forms}. For each $t \in \{0,1\}$, we have $K-1 \in B_t$ and, if $r \in \{0, \ldots, K-2\}$, then $r \in B_t$ if and only if $r+1 \in A_{1-t}$.
\end{lemma}

\begin{proof}
Let $t \in \{0,1\}$ and let $v$ and $w$ be adjacent vertices with $v$ of type $(t+K) \bmod 2$.

We first assume for a contradiction that $K-1 \not \in B_t$. Let $a$ be a vertex of $S(w,K-1)$ that is not in the branch of $v$ (see Figure~\ref{picture:2-2}). Since $\alpha_{K-1}^t = \infty$, there exists $h \in H$ such that $a$ is the only vertex labelled by $o$ in the diagram of the image of $h$ in $\tilde{H}^{K-1}(w)$. Now if we look at the image $\tilde{h}$ of $h$ in $\tilde{H}^{K-1}(v)$, Lemma~\ref{lemma:forms} and the fact that $K-1 \not \in B_t$ imply that $f^t_v(\mathcal{D}(\tilde{h})) = (E, *, \ldots, *)$, where the first branch $b_1$ is the branch of $w$. This means that the number of vertices labelled by $o$ in $S(v,K) \cap b_1$ should be even, but this is a contradiction with the fact that $a$ is the only vertex of $S(v,K) \cap b_1$ labelled by $o$.

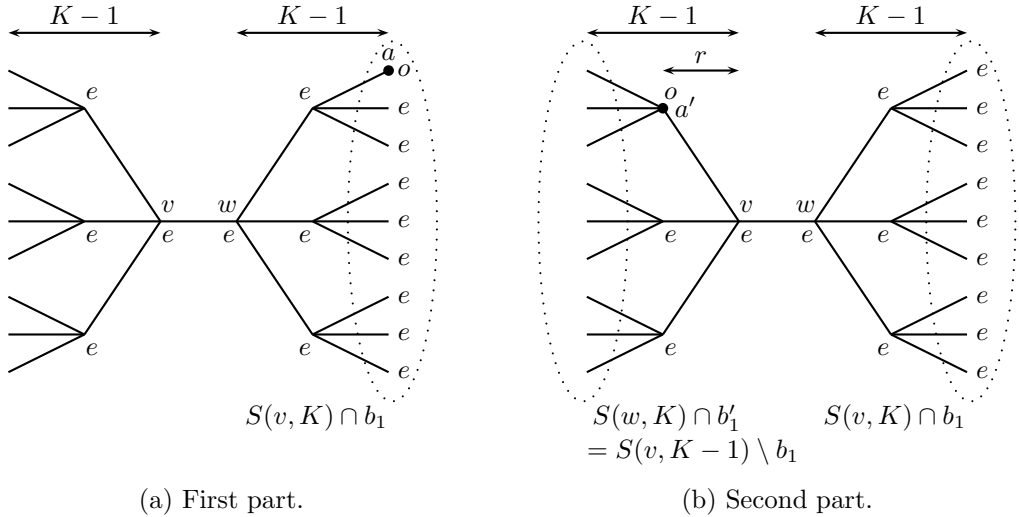
\begin{figure}[b!]
\begin{subfigure}{.49\textwidth}
\centering
\begin{pspicture*}(-2.55,-3.3)(3.15,2.9)
\fontsize{10pt}{10pt}\selectfont
\psset{unit=1cm}

\psline(0.5,0)(-0.5,0)
\rput(0.4,-0.2){$e$}
\rput(-0.4,-0.2){$e$}

\psline(0.5,0)(1.5,1.5)
\psline(0.5,0)(1.5,0)
\psline(0.5,0)(1.5,-1.5)
\rput(1.4,1.7){$e$}
\rput(1.4,-0.2){$e$}
\rput(1.4,-1.7){$e$}

\psline(-0.5,0)(-1.5,1.5)
\psline(-0.5,0)(-1.5,0)
\psline(-0.5,0)(-1.5,-1.5)
\rput(-1.4,1.7){$e$}
\rput(-1.4,-0.2){$e$}
\rput(-1.4,-1.7){$e$}

\psline(1.5,1.5)(2.5,2)
\psline(1.5,1.5)(2.5,1.5)
\psline(1.5,1.5)(2.5,1)
\psline(1.5,0)(2.5,0.5)
\psline(1.5,0)(2.5,0)
\psline(1.5,0)(2.5,-0.5)
\psline(1.5,-1.5)(2.5,-1)
\psline(1.5,-1.5)(2.5,-1.5)
\psline(1.5,-1.5)(2.5,-2)
\rput(2.5,2.22){$a$}
\psdot[dotsize=4pt](2.5,2)
\rput(2.7,2){$o$}
\rput(2.7,1.5){$e$}
\rput(2.7,1){$e$}
\rput(2.7,0.5){$e$}
\rput(2.7,0){$e$}
\rput(2.7,-0.5){$e$}
\rput(2.7,-1){$e$}
\rput(2.7,-1.5){$e$}
\rput(2.7,-2){$e$}

\psline(-1.5,1.5)(-2.5,2)
\psline(-1.5,1.5)(-2.5,1.5)
\psline(-1.5,1.5)(-2.5,1)
\psline(-1.5,0)(-2.5,0.5)
\psline(-1.5,0)(-2.5,0)
\psline(-1.5,0)(-2.5,-0.5)
\psline(-1.5,-1.5)(-2.5,-1)
\psline(-1.5,-1.5)(-2.5,-1.5)
\psline(-1.5,-1.5)(-2.5,-2)

\psellipse[linestyle=dotted](2.55,0)(0.6,2.4)
\rput(1.55,-2.6){$S(v,K)\cap b_1$}

\rput(-0.4,0.2){$v$}
\rput(0.38,0.2){$w$}

\psline[arrows=<->](0.5,2.5)(2.5,2.5)
\rput(1.5,2.72){$K-1$}
\psline[arrows=<->](-0.5,2.5)(-2.5,2.5)
\rput(-1.5,2.72){$K-1$}
\end{pspicture*}
\caption{First part.}\label{picture:2-2}
\end{subfigure}
\begin{subfigure}{.49\textwidth}
\centering
\begin{pspicture*}(-3.15,-3.3)(3.15,2.9)
\fontsize{10pt}{10pt}\selectfont
\psset{unit=1cm}

\psline(0.5,0)(-0.5,0)
\rput(0.4,-0.2){$e$}
\rput(-0.4,-0.2){$e$}

\psline(0.5,0)(1.5,1.5)
\psline(0.5,0)(1.5,0)
\psline(0.5,0)(1.5,-1.5)
\rput(1.4,1.7){$e$}
\rput(1.4,-0.2){$e$}
\rput(1.4,-1.7){$e$}

\psline(-0.5,0)(-1.5,1.5)
\psline(-0.5,0)(-1.5,0)
\psline(-0.5,0)(-1.5,-1.5)
\rput(-1.2,1.52){$a'$}
\psdot[dotsize=4pt](-1.5,1.5)
\rput(-1.4,1.7){$o$}
\rput(-1.4,-0.2){$e$}
\rput(-1.4,-1.7){$e$}

\psline(1.5,1.5)(2.5,2)
\psline(1.5,1.5)(2.5,1.5)
\psline(1.5,1.5)(2.5,1)
\psline(1.5,0)(2.5,0.5)
\psline(1.5,0)(2.5,0)
\psline(1.5,0)(2.5,-0.5)
\psline(1.5,-1.5)(2.5,-1)
\psline(1.5,-1.5)(2.5,-1.5)
\psline(1.5,-1.5)(2.5,-2)
\rput(2.7,2){$e$}
\rput(2.7,1.5){$e$}
\rput(2.7,1){$e$}
\rput(2.7,0.5){$e$}
\rput(2.7,0){$e$}
\rput(2.7,-0.5){$e$}
\rput(2.7,-1){$e$}
\rput(2.7,-1.5){$e$}
\rput(2.7,-2){$e$}

\psline(-1.5,1.5)(-2.5,2)
\psline(-1.5,1.5)(-2.5,1.5)
\psline(-1.5,1.5)(-2.5,1)
\psline(-1.5,0)(-2.5,0.5)
\psline(-1.5,0)(-2.5,0)
\psline(-1.5,0)(-2.5,-0.5)
\psline(-1.5,-1.5)(-2.5,-1)
\psline(-1.5,-1.5)(-2.5,-1.5)
\psline(-1.5,-1.5)(-2.5,-2)

\psellipse[linestyle=dotted](2.55,0)(0.6,2.4)
\rput(1.55,-2.6){$S(v,K)\cap b_1$}

\psellipse[linestyle=dotted](-2.55,0)(0.6,2.4)
\rput(-1.45,-2.6){$S(w,K)\cap b'_1$}
\rput(-1.1,-3.05){$=S(v,K-1)\setminus b_1$}

\rput(-0.4,0.2){$v$}
\rput(0.38,0.2){$w$}

\psline[arrows=<->](0.5,2.5)(2.5,2.5)
\rput(1.5,2.72){$K-1$}
\psline[arrows=<->](-0.5,2.5)(-2.5,2.5)
\rput(-1.5,2.72){$K-1$}
\psline[arrows=<->](-0.5,2)(-1.5,2)
\rput(-1,2.17){$r$}
\end{pspicture*}
\caption{Second part.}\label{picture:2-2bis}
\end{subfigure}
\caption{Illustration of Lemma~\ref{lemma:2-2:AB}.}
\end{figure}

We now show the second part of the statement. Let $r \in \{0, \ldots, K-2\}$ and let $a'$ be a vertex of $S(w,r+1)$ in the branch of $v$, which we denote by $b'_1$ (see Figure~\ref{picture:2-2bis}). Since $K(0) = K(1) = K$, there exists $h \in H$ such that $a'$ is the only vertex labelled by $o$ in the diagram of the image of $h$ in $\tilde{H}^{K-1}(w)$. By Lemma~\ref{lemma:forms} (with $1-t$ instead of $t$), the number of vertices labelled by $o$ in $S(w,K) \cap b'_1$ is odd if and only if $r+1 \in A_{1-t}$. Now we observe the diagram of the image $\tilde{h}$ of $h$ in $\tilde{H}^{K-1}(v)$. Lemma~\ref{lemma:forms} tells us that $f^t_v(\mathcal{D}(\tilde{h})) = (p_1, *, \ldots, *)$ where $p_1$ is the parity of the number of vertices labelled by $o$ in $(S_{A_t}(v) \cap b_1) \cup (S_{B_t}(v) \setminus b_1)$. But all the vertices of $S(v,K) \cap b_1$ are labelled by $e$, so $p_1 = E$. Hence, there is an even number of vertices labelled by $o$ in $(S_{A_t}(v) \cap b_1) \cup (S_{B_t}(v) \setminus b_1)$. As $K-1 \in B_t$ and $a'$ is the only vertex of $B(v,K-2)$ labelled by $o$, this means that the number of vertices labelled by $o$ in $S(v,K-1) \setminus b_1$ is odd if and only if $r \in B_t$. Since $S(w,K) \cap b'_1 = S(v,K-1) \setminus b_1$, we obtained that $r+1 \in A_{1-t}$ if and only if $r \in B_t$.
\end{proof}


\subsection{Possible shapes for \texorpdfstring{$f^t_v$}{f^t_v} when \texorpdfstring{$K(t) > K(1-t)$}{K(t) > K(1-t)}}

In the case where $c(0),c(1) \in \{1,3\}$, it can happen that $K(t) > K(1-t)$ and Lemma~\ref{lemma:forms} cannot be applied. Indeed, $\mathcal{D}(\tilde{H}^{K(t)-1}(v))$ does not contain all the diagrams, which prevents us from using the diagrams $\delta_r$ as above.
To deal with this case, we therefore need to better understand $\mathcal{D}(\tilde{H}^{K(t)-1}(v))$. This is the subject of the following result, which is illustrated in Figure~\ref{picture:all4}.

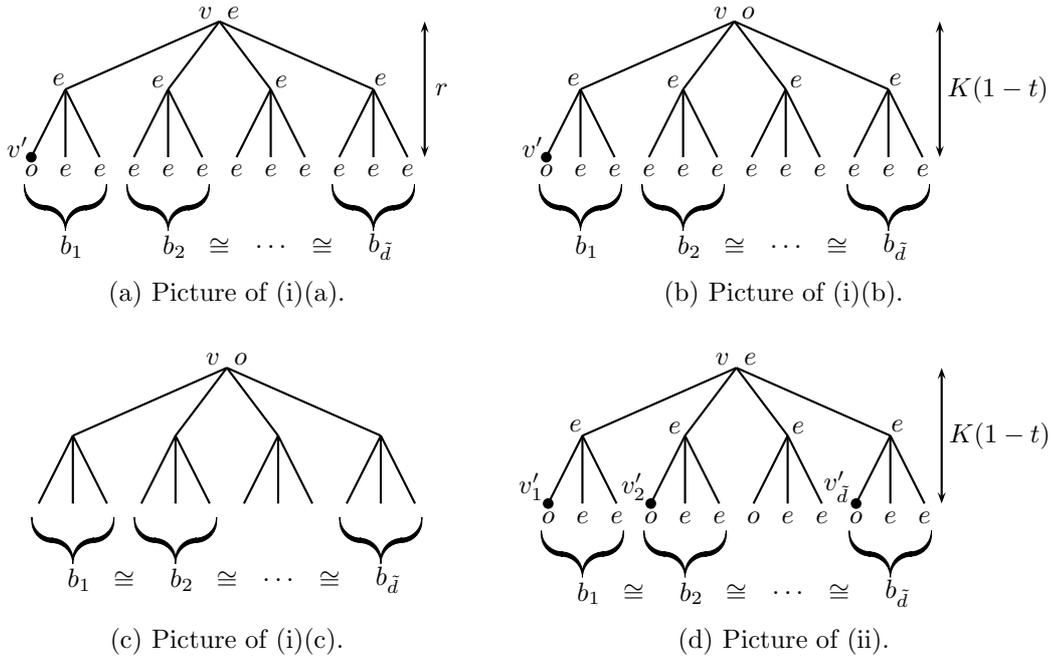
\begin{figure}[b!]
\begin{subfigure}{.5\textwidth}
\centering
\begin{pspicture*}(-2.8,-3.2)(3,0.25)
\fontsize{10pt}{10pt}\selectfont
\psset{unit=0.9cm}

\rput(-0.2,0.13){$v$}
\rput(0.2,0.13){$e$}
\psline(0,0)(-2.25,-1) \rput(-2.35,-0.85){$e$}
\psline(0,0)(-0.75,-1) \rput(-0.9,-0.88){$e$}
\psline(0,0)(0.75,-1) \rput(0.9,-0.88){$e$}
\psline(0,0)(2.25,-1) \rput(2.35,-0.85){$e$}

\psline(-2.25,-1)(-2.75,-2) \rput(-2.75,-2.2){$o$}
\psdot[dotsize=4pt](-2.75,-2) \rput(-2.95,-1.85){$v'$}
\psline(-2.25,-1)(-2.25,-2) \rput(-2.25,-2.2){$e$}
\psline(-2.25,-1)(-1.75,-2) \rput(-1.75,-2.2){$e$}
\psline(-0.75,-1)(-1.25,-2) \rput(-1.25,-2.2){$e$}
\psline(-0.75,-1)(-0.75,-2) \rput(-0.75,-2.2){$e$}
\psline(-0.75,-1)(-0.25,-2) \rput(-0.25,-2.2){$e$}
\psline(0.75,-1)(0.25,-2) \rput(0.25,-2.2){$e$}
\psline(0.75,-1)(0.75,-2) \rput(0.75,-2.2){$e$}
\psline(0.75,-1)(1.25,-2) \rput(1.25,-2.2){$e$}
\psline(2.25,-1)(1.75,-2) \rput(1.75,-2.2){$e$}
\psline(2.25,-1)(2.25,-2) \rput(2.25,-2.2){$e$}
\psline(2.25,-1)(2.75,-2) \rput(2.75,-2.2){$e$}

\psline[arrows=<->](3,0)(3,-2)
\rput(3.25,-1){$r$}
\psbrace[rot=90,ref=lC,nodesepA=-2pt,nodesepB=6pt](-2.85,-2.4)(-1.65,-2.4){$b_1$}
\psbrace[rot=90,ref=lC,nodesepA=-2pt,nodesepB=6pt](-1.35,-2.4)(-0.15,-2.4){$b_2$}
\rput(0.75,-3.32){$\ldots$}
\psbrace[rot=90,ref=lC,nodesepA=-2pt,nodesepB=6pt](1.65,-2.4)(2.85,-2.4){$b_{\tilde{d}}$}

\rput(0,-3.32){$\cong$}
\rput(1.5,-3.32){$\cong$}
\end{pspicture*}
\caption{Picture of (i)(a).}\label{picture:special1a}
\end{subfigure}%
\begin{subfigure}{.5\textwidth}
\centering
\begin{pspicture*}(-2.8,-3.2)(4.1,0.25)
\fontsize{10pt}{10pt}\selectfont
\psset{unit=0.9cm}

\rput(-0.2,0.13){$v$}
\rput(0.2,0.13){$o$}
\psline(0,0)(-2.25,-1) \rput(-2.35,-0.85){$e$}
\psline(0,0)(-0.75,-1) \rput(-0.9,-0.88){$e$}
\psline(0,0)(0.75,-1) \rput(0.9,-0.88){$e$}
\psline(0,0)(2.25,-1) \rput(2.35,-0.85){$e$}

\psline(-2.25,-1)(-2.75,-2) \rput(-2.75,-2.2){$o$}
\psdot[dotsize=4pt](-2.75,-2) \rput(-2.95,-1.85){$v'$}
\psline(-2.25,-1)(-2.25,-2) \rput(-2.25,-2.2){$e$}
\psline(-2.25,-1)(-1.75,-2) \rput(-1.75,-2.2){$e$}
\psline(-0.75,-1)(-1.25,-2) \rput(-1.25,-2.2){$e$}
\psline(-0.75,-1)(-0.75,-2) \rput(-0.75,-2.2){$e$}
\psline(-0.75,-1)(-0.25,-2) \rput(-0.25,-2.2){$e$}
\psline(0.75,-1)(0.25,-2) \rput(0.25,-2.2){$e$}
\psline(0.75,-1)(0.75,-2) \rput(0.75,-2.2){$e$}
\psline(0.75,-1)(1.25,-2) \rput(1.25,-2.2){$e$}
\psline(2.25,-1)(1.75,-2) \rput(1.75,-2.2){$e$}
\psline(2.25,-1)(2.25,-2) \rput(2.25,-2.2){$e$}
\psline(2.25,-1)(2.75,-2) \rput(2.75,-2.2){$e$}

\psline[arrows=<->](3,0)(3,-2)
\rput(3.85,-1){$K(1-t)$}
\psbrace[rot=90,ref=lC,nodesepA=-2pt,nodesepB=6pt](-2.85,-2.4)(-1.65,-2.4){$b_1$}
\psbrace[rot=90,ref=lC,nodesepA=-2pt,nodesepB=6pt](-1.35,-2.4)(-0.15,-2.4){$b_2$}
\rput(0.75,-3.32){$\ldots$}
\psbrace[rot=90,ref=lC,nodesepA=-2pt,nodesepB=6pt](1.65,-2.4)(2.85,-2.4){$b_{\tilde{d}}$}

\rput(0,-3.32){$\cong$}
\rput(1.5,-3.32){$\cong$}
\end{pspicture*}
\caption{Picture of (i)(b).}\label{picture:special1b}
\end{subfigure}
\vskip\baselineskip
\begin{subfigure}{.5\textwidth}
\centering
\begin{pspicture*}(-2.7,-3.2)(2.7,0.25)
\fontsize{10pt}{10pt}\selectfont
\psset{unit=0.9cm}

\rput(-0.2,0.13){$v$}
\rput(0.2,0.13){$o$}
\psline(0,0)(-2.25,-1)
\psline(0,0)(-0.75,-1)
\psline(0,0)(0.75,-1)
\psline(0,0)(2.25,-1)

\psline(-2.25,-1)(-2.75,-2)
\psline(-2.25,-1)(-2.25,-2)
\psline(-2.25,-1)(-1.75,-2)
\psline(-0.75,-1)(-1.25,-2)
\psline(-0.75,-1)(-0.75,-2)
\psline(-0.75,-1)(-0.25,-2)
\psline(0.75,-1)(0.25,-2)
\psline(0.75,-1)(0.75,-2)
\psline(0.75,-1)(1.25,-2)
\psline(2.25,-1)(1.75,-2)
\psline(2.25,-1)(2.25,-2)
\psline(2.25,-1)(2.75,-2)

\psbrace[rot=90,ref=lC,nodesepA=-2pt,nodesepB=6pt](-2.85,-2.2)(-1.65,-2.2){$b_1$}
\psbrace[rot=90,ref=lC,nodesepA=-2pt,nodesepB=6pt](-1.35,-2.2)(-0.15,-2.2){$b_2$}
\rput(0.75,-3.12){$\ldots$}
\psbrace[rot=90,ref=lC,nodesepA=-2pt,nodesepB=6pt](1.65,-2.2)(2.85,-2.2){$b_{\tilde{d}}$}

\rput(-1.5,-3.12){$\cong$}
\rput(0,-3.12){$\cong$}
\rput(1.5,-3.12){$\cong$}
\end{pspicture*}
\caption{Picture of (i)(c).}\label{picture:special1c}
\end{subfigure}%
\begin{subfigure}{.5\textwidth}
\centering
\begin{pspicture*}(-2.85,-3.2)(4.1,0.25)
\fontsize{10pt}{10pt}\selectfont
\psset{unit=0.9cm}

\rput(-0.2,0.13){$v$}
\rput(0.2,0.13){$e$}
\psline(0,0)(-2.25,-1) \rput(-2.35,-0.85){$e$}
\psline(0,0)(-0.75,-1) \rput(-0.9,-0.88){$e$}
\psline(0,0)(0.75,-1) \rput(0.9,-0.88){$e$}
\psline(0,0)(2.25,-1) \rput(2.35,-0.85){$e$}

\psline(-2.25,-1)(-2.75,-2) \rput(-2.75,-2.2){$o$}
\psdot[dotsize=4pt](-2.75,-2) \rput(-3,-1.75){$v'_1$}
\psline(-2.25,-1)(-2.25,-2) \rput(-2.25,-2.2){$e$}
\psline(-2.25,-1)(-1.75,-2) \rput(-1.75,-2.2){$e$}
\psline(-0.75,-1)(-1.25,-2) \rput(-1.25,-2.2){$o$}
\psdot[dotsize=4pt](-1.25,-2) \rput(-1.5,-1.75){$v'_2$}
\psline(-0.75,-1)(-0.75,-2) \rput(-0.75,-2.2){$e$}
\psline(-0.75,-1)(-0.25,-2) \rput(-0.25,-2.2){$e$}
\psline(0.75,-1)(0.25,-2) \rput(0.25,-2.2){$o$}
\psline(0.75,-1)(0.75,-2) \rput(0.75,-2.2){$e$}
\psline(0.75,-1)(1.25,-2) \rput(1.25,-2.2){$e$}
\psline(2.25,-1)(1.75,-2) \rput(1.75,-2.2){$o$}
\psdot[dotsize=4pt](1.75,-2) \rput(1.46,-1.75){$v'_{\tilde{d}}$}
\psline(2.25,-1)(2.25,-2) \rput(2.25,-2.2){$e$}
\psline(2.25,-1)(2.75,-2) \rput(2.75,-2.2){$e$}

\psline[arrows=<->](3,0)(3,-2)
\rput(3.85,-1){$K(1-t)$}
\psbrace[rot=90,ref=lC,nodesepA=-2pt,nodesepB=6pt](-2.85,-2.4)(-1.65,-2.4){$b_1$}
\psbrace[rot=90,ref=lC,nodesepA=-2pt,nodesepB=6pt](-1.35,-2.4)(-0.15,-2.4){$b_2$}
\rput(0.75,-3.32){$\ldots$}
\psbrace[rot=90,ref=lC,nodesepA=-2pt,nodesepB=6pt](1.65,-2.4)(2.85,-2.4){$b_{\tilde{d}}$}

\rput(-1.5,-3.32){$\cong$}
\rput(0,-3.32){$\cong$}
\rput(1.5,-3.32){$\cong$}
\end{pspicture*}
\caption{Picture of (ii).}\label{picture:special2}
\end{subfigure}
\caption{Illustration of Lemma~\ref{lemma:special}.}\label{picture:all4}
\end{figure}


\begin{lemma}\label{lemma:special}
Suppose that $c(t) \neq 0$ and $K(t) > K(1-t)$ and let $v$ be a vertex of type $(t+K(t))\bmod 2$. Denote by $b_1, \ldots, b_{\tilde{d}}$ the branches of the vertices adjacent to $v$. For each $j \in \{1, \ldots, \tilde{d}-1\}$, fix $\tilde{\gamma}_j \in \Alt_{(i)}(B(v,K(t)))$ sending $b_j$ to $b_{\tilde{d}}$.
\begin{enumerate}[(i)]
\item Let $r \in \{0, \ldots, K(t)-1\}$ be such that $r < K(1-t)$ or $r \equiv K(t) \bmod 2$. 
\begin{enumerate}[(a)]
\item If $r \geq 1$, then there exist a diagram $\delta \in \mathcal{D}(\tilde{H}^{K(t)-1}(v))$ and a vertex $v' \in S(v,r) \cap b_1$ such that $v'$ is the only vertex labelled by $o$ in $\delta \cap B(v,r)$ and, for each $j \in \{2, \ldots, \tilde{d}-1\}$ and each $w \in b_j$, $\tilde{\gamma}_j(w)$ has the same label as $w$.
\item If $r = 0$, $c(1-t) = 1$, $K(t) \not \equiv K(1-t) \bmod 2$ and the set $A$ associated to $1-t$ in Lemma~\ref{lemma:forms} contains $0$, then there exist a diagram $\delta \in \mathcal{D}(\tilde{H}^{K(t)-1}(v))$ and a vertex $v' \in S(v,K(1-t)) \cap b_1$ such that the only vertices labelled by $o$ in $\delta \cap B(v,K(1-t))$ are $v$ and $v'$ and, for each $j \in \{2, \ldots, \tilde{d}-1\}$ and each $w \in b_j$, $\tilde{\gamma}_j(w)$ has the same label as $w$.
\item If $r = 0$ and we are not in (b), then there exists a diagram $\delta \in \mathcal{D}(\tilde{H}^{K(t)-1}(v))$ in which $v$ is labelled by $o$ and such that, for each $j \in \{1, \ldots, \tilde{d}-1\}$ and each $w \in b_j$, $\tilde{\gamma}_j(w)$ has the same label as $w$.
\end{enumerate}
\item Suppose that $c(1-t)=3$ and $K(t) \not \equiv K(1-t) \bmod 2$. Then there exist a diagram $\sigma \in \mathcal{D}(\tilde{H}^{K(t)-1}(v))$ and a vertex $v'_j \in S(v,K(1-t)) \cap b_j$ for each $j \in \{1,\ldots,\tilde{d}\}$ such that the only vertices labelled by $o$ in $\sigma \cap B(v,K(1-t))$ are $v'_1,\ldots,v'_{\tilde{d}}$ and, for each $j \in \{1, \ldots, \tilde{d}-1\}$ and each $w \in b_j$, $\tilde{\gamma}_j(w)$ has the same label as $w$.
\end{enumerate}
\end{lemma}

\begin{proof}
We prove all three cases of (i) simultaneously, the proof of (ii) being similar. Let us say that a diagram $\delta \in \Diag_{v,i}$ (with $0 \leq i \leq K(t)-1$) is \textit{suitable} if $\delta \in \mathcal{D}(\tilde{H}^i(v))$ and if $\delta$ satisfies the conditions of the statement which concern the ball $B(v,i)$.

First remark that there exists a suitable diagram $\delta \in \Diag_{v,r}$. Indeed, it suffices to label all the vertices of $\delta \cap B(v,r-1)$ by $e$ and then to label exactly one vertex of $S(v,r)$ by $o$, placed in $b_1$ if $r \geq 1$. This always gives a diagram in $\mathcal{D}(\tilde{H}^r(v))$ because the assumption on $r$ is made so that $\alpha_r(v) = \infty$.

We now prove that, for each $r+1 \leq i \leq K(t)-1$, if $\delta \in \Diag_{v,i-1}$ is suitable then there exists $\hat{\delta} \in \Diag_{v,i}$ suitable and extending $\delta$. We obviously start by defining $\hat{\delta} \cap B(v,i-1) = \delta$, and there remains to give the labels of the vertices in $S(v,i)$. If $i < K(1-t)$ or $i \equiv K(t) \bmod 2$, we have $\alpha_i(v) = \infty$ and by Lemma~\ref{lemma:crucial} (i) we can simply label all the vertices of $S(v,i)$ by $e$. Now if $i \geq K(1-t)$ and $i \not \equiv K(t) \bmod 2$, then $v$ is of type $(1-t+i) \bmod 2$ and we know that a diagram $\hat{\delta}$ with $\hat{\delta} \cap B(v,i-1) = \delta$ is contained in $\mathcal{D}(\tilde{H}^i(v))$ if and only if $\hat{\delta} \cap B(w,K(1-t)) \in \mathcal{D}(\tilde{H}^{K(1-t)}(w))$ for each $w$ at distance $i-K(1-t)$ from $v$ (see Lemma~\ref{lemma:describe} with $1-t$). In other words, the only restrictions for being in $\mathcal{D}(\tilde{H}^i(v))$ are given by the maps $f^{1-t}_w$. These are always restrictions on the parity of the number of vertices labelled by $o$ in the $K'(1-t)$-branches. When $K'(1-t) < i$, these branches are smaller than the whole ball $B(v,i)$. In this case, since $\delta$ is suitable, a labelling of the vertices of $S(v,i) \cap b_{\tilde{d}}$ satisfying the restrictions which concern them can be pulled back by $\tilde{\gamma}_j$ in a labelling of the vertices of $S(v,i) \cap b_j$ also satisfying the restrictions (for each $j \in \{2,\ldots,\tilde{d}-1\}$ in cases (a) and (b) and for each $j \in \{1,\ldots,\tilde{d}-1\}$ in case (c)). The only case where $K'(1-t) = i$ is the case where $i = K(1-t)$, $c(1-t) = 1$, and $K(1-t) \not\equiv K(t) \bmod 2$. If $r \geq 1$, we want to prove (a) and there is no problem: we can label all the vertices of $S(v,i) \setminus b_1$ with $e$ and adapt the labelling of $S(v,i) \cap b_1$. If $r = 0$, and if the set $A$ associated to $1-t$ in Lemma~\ref{lemma:forms} does not contain $0$, then since all the vertices of $B(v,i-1) \setminus  \{v\}$ are labelled by $e$ all the vertices of $S(v,i)$ can be labelled by $e$. If $A$ contains $0$, then the restriction given by $f^{1-t}_v$ imposes the number of vertices of $S(v,i)$ labelled by $o$ to be odd. As we want to prove (b), we label one vertex of $S(v,i) \cap b_1$ by $o$ and the other vertices of $S(v,i)$ by $e$. In any cases, the diagram $\hat{\delta} \in \Diag_{v,i}$ constructed in this way is suitable.
\end{proof}

We can now look at the shapes that $f^t_v$ can take when $K(t) > K(1-t)$. It is actually sufficient for our classification to count how many shapes are possible.


\begin{lemma}\label{lemma:forms2}
Let $t \in \{0,1\}$. Fix $c(0), c(1) \in \{1,3\}$ and $K(t) > K(1-t)$ and let $v$ be a vertex of type $(t+K(t))\bmod 2$. Let $N$ be the number of maps $f^t_v$ that can be observed for at least one $H \in \mathcal{H}_T^+$ satisfying $H \supseteq \Alt_{(i)}(T)^+$ and with these invariants $c(0)$, $c(1)$, $K(0)$ and $K(1)$. Define
$$R = \{r \in \{0,1,\ldots,K(t)-1\} \suchthat r < K(1-t) \text{ or } r \equiv K(t) \bmod 2\}.$$
\begin{itemize}
\item If $c(t) = 1$, then $N \leq 2^{\lvert R\rvert +\varepsilon}$ where $\varepsilon = 1$ if $c(1-t) = 3$ and $K(t) \not \equiv K(1-t) \bmod 2$, and $\varepsilon = 0$ otherwise.
\item If $c(t) = 3$, then $N \leq 2^{\lvert R\setminus\{0\}\rvert +\varepsilon'}$ where $\varepsilon' = 1$ if $c(1-t) = 1$, $K(1-t) \neq 0$ and $K(t) \not \equiv K(1-t) \bmod 2$, and $\varepsilon' = 0$ otherwise.
\end{itemize}
\end{lemma}

\begin{proof}
We separate the case $c(t) = 1$ from the case $c(t) = 3$.
\begin{itemize}
\item Suppose that $c(t) = 1$.
To find an upper bound on $N$, we first give a set of diagrams of $\mathcal{D}(\tilde{H}^{K(t)-1}(v))$ generating (via Lemma~\ref{lemma:ftrivial}) all the diagrams of $\mathcal{D}(\tilde{H}^{K(t)-1}(v))$. \\
For $r \in R$, take $\delta_r \in \mathcal{D}(\tilde{H}^{K(t)-1}(v))$ as in Lemma~\ref{lemma:special} (i). In this case ($c(t) = 1$), we actually do not care of the condition with the elements $\tilde{\gamma}_j$. \\
For $r \in \{0, 1, \ldots, K(t)-1\} \setminus R$, an element $\delta_r$ with all vertices of $B(v,r-1)$ labelled by $e$ and exactly one vertex labelled by $o$ in $S(v,r)$ does not exist. Instead, and if $K'(1-t) > 0$, we consider an element $\rho_r \in \mathcal{D}(\tilde{H}^{K(t)-1}(v))$ with all vertices of $B(v,r-1)$ labelled by $e$ and exactly two vertices labelled by $o$ in $S(v,r)$, placed such that the minimal branch containing them is a $K'(1-t)$-branch. This diagram can be used to generate, via Lemma~\ref{lemma:ftrivial}, all the possible labellings of $S(v,r)$ with an even number of vertices labelled by $o$ in each $K'(1-t)$-branch.\\
In the particular case where $c(1-t) = 3$ and $r = K(1-t)$, as $\alpha_r(v) = (K(1-t)-1)^*$ we also need to add a diagram $\sigma \in \mathcal{D}(\tilde{H}^{K(t)-1}(v))$ as in Lemma~\ref{lemma:special} (ii). Note that this element $\sigma$ is considered if and only if $c(1-t) = 3$ and $K(t) \not \equiv K(1-t) \bmod 2$ (so that $K(1-t) \not \in R$). We write $\varepsilon = 1$ in this case and $\varepsilon = 0$ otherwise.\\
By construction, $\mathcal{D}(\tilde{H}^{K(t)-1}(v))$ can be generated using $\delta_r$, $\rho_r$ and $\sigma$ (if $\varepsilon = 1$).\\
It is however not hard to convince oneself that the diagrams $\rho_r$ can be chosen so that $f^t_v(\rho_r)$ must always be equal to $E$. Indeed, take $\rho_r$ as above with vertices $a, b \in S(v,r)$ labelled by $o$ and let $\tilde{h}$ be an element realizing this diagram. Let $\tilde{\tau} \in \Alt_{(i)}(B(v,K(t)))$ be an element that stabilizes $a$ while sending the $(K(1-t)-1)$-branch containing $b$ to another branch. In this way, $\tilde{h}\tilde{\tau}\tilde{h}^{-1}$ has a diagram $\rho'_r$ satisfying the same property as $\rho_r$ but it is now sure by Lemma~\ref{lemma:ftrivial} that $f^t_v(\rho'_r) = E$. Hence, a map $f^t_v$ is fully characterized by its values $f^t_v(\delta_r)$ (for each $r \in R$) and $f^t_v(\sigma)$ (if $\varepsilon = 1$), which leaves at most $2^{\lvert R\rvert +\varepsilon}$ possibilities for $f^t_v$.

\item For $c(t) = 3$, the idea is exactly the same. The only difference is that $f^t_v$ takes its values in $\bigslant{\{E, O\}^{\tilde{d}}}{\langle(O,\ldots,O)\rangle}$. The diagrams $\delta_r$, $\rho_r$ and $\sigma$ with the same properties as above once again generate all the diagrams. Denote by $b_1, \ldots, b_{\tilde{d}}$ the branches of the vertices adjacent to $v$, as for the definition of $J^t$.\\
This time, we fix $\tilde{\gamma}_1, \ldots, \tilde{\gamma}_{\tilde{d}-1}$ as in Lemma~\ref{lemma:special} and really want $\delta_r$ to satisfy the conditions given in this same lemma. In this way, for $r > 0$ we obtain using Lemma~\ref{lemma:ftrivial} that $f^t_v(\delta_r)$ is of the form $[(x_r, y_r, \ldots, y_r)]$, and we can assume that $y_r = E$. For $r = 0$, if $c(1-t) = 1$, $K(t) \not \equiv K(1-t) \bmod 2$ and if the set $A$ associated to $1-t$ in Lemma~\ref{lemma:forms} contains $0$, then we define $\varepsilon = 1$ and also get $f^t_v(\delta_0) = [(x_0, y_0, \ldots, y_0)]$ (and can assume that $y_0 = E$). Otherwise, we set $\varepsilon = 0$ and find that $f^t_v(\delta_0) = [(E, \ldots, E)]$.\\
For $\rho_r$, as in the first case they can generally be chosen so that $f^t_v(\rho_r)$ must always be equal to $[(E,\ldots,E)]$. Indeed, if there exist $\tilde{h}$ and $\tilde{\tau}$ as above and stabilizing the branches $b_1, \ldots, b_{\tilde{d}}$ then the same reasoning works. This is always possible, unless $c(1-t) = 1$ and $r = K(1-t) \neq 0$. This only happens when $c(1-t) = 1$, $K(t) \not \equiv K(1-t) \bmod 2$ and $K(1-t) \neq 0$, in which case we set $\varepsilon' = 1$. Otherwise, set $\varepsilon' = 0$. If $\varepsilon' = 1$, then the diagram $\rho_{K(1-t)}$ has two vertices in $S(v,K(1-t))$ labelled by $o$ and they are in different branches, say $b_1$ and $b_2$. Let $\tilde{h}$ be an element realizing this diagram and let $\tilde{\tau} \in \Alt_{(i)}(B(v,K(t)))$ be an element interchanging $b_2$ and $b_3$, interchanging $b_4$ and $b_5$, and fixing $b_1, b_6, \ldots, b_{\tilde{d}}$. In this way, $\tilde{h}' = \tilde{h} \tilde{\tau} \tilde{h}^{-1}$ has a diagram $\rho'_{K(1-t)}$ with two vertices in $S(v,K(1-t))$ labelled by $o$: one in $b_2$ and one in $b_3$. Moreover, we know that $f^t_v(\rho'_{K(1-t)}) = [(E, x, x, y, y, E \ldots, E)]$. Now let $\tilde{\tau}' \in \Alt_{(i)}(B(v,K(t)))$ be an element interchanging $b_1$ and $b_3$, interchanging $b_4$ and $b_5$, and fixing $b_2, b_6, \ldots, b_{\tilde{d}}$. Then $\tilde{h}'' = \tilde{h}' \tilde{\tau}' \tilde{h}^{\prime -1}$ has a diagram $\rho''_{K(1-t)}$ with two vertices in $S(v,K(1-t))$ labelled by $o$: one in $b_1$ and one in $b_2$. This time, we know that $f^t_v(\rho''_{K(1-t)}) = [(x, x, E, \ldots, E)]$.\\
Concerning $\sigma$ (if it must be considered), take it as in Lemma~\ref{lemma:special} (ii) so that all its branches $b_1, \ldots, b_{\tilde{d}}$ are identical (via $\tilde{\gamma}_1, \ldots, \tilde{\gamma}_{\tilde{d}-1}$). We obtain $f^t_v(\sigma) = [(E,\ldots,E)]$.\\
In total, there are at most $2^{\lvert R \setminus \{0\}\rvert + \varepsilon + \varepsilon'}$ possibilities for $f^t_v$. However, having $\varepsilon = 1$ also implies $\varepsilon' = 1$ and the diagram $\delta_0$ which we chose when $\varepsilon = 1$ can be used to generate an element with the properties of $\rho''_{K(1-t)}$. We can therefore forget $\rho''_{K(1-t)}$ when $\varepsilon = 1$, which leaves at most $2^{\lvert R \setminus \{0\}\rvert + \varepsilon'}$ possibilities for $f^t_v$.\qedhere
\end{itemize}
\end{proof}


\subsection{Upper bound on the number of groups \texorpdfstring{$H \in \mathcal{H}_T^+$ with $H \supseteq \Alt_{(i)}(T)^+$}{H}}

For each fixed values of $c(0)$, $c(1)$, $K'(0)$, $K'(1)$, we can now compute an upper bound on the number of groups $H \in \mathcal{H}_T^+$ satisfying $H \supseteq \Alt_{(i)}(T)^+$ and with these invariants. Recall that $a \boxplus b := a + b - \left\lceil\frac{|a-b|}{2}\right\rceil$ for $a,b \in \N$.


\begin{proposition}\label{proposition:upperbound}
Fix $c(0), c(1) \in \{0,1,2,3\}$ such that $c(0) = 2$ if and only if $c(1) = 2$ and fix $K'(0), K'(1) \in \N \cup \{\infty\}$ such that $K'(t) = \infty$ if and only if $c(t) = 0$, and $K'(0) = K'(1)$ if $c(0) = c(1) = 2$. Let $N$ be the number of groups $H \in \mathcal{H}_T^+$ satisfying $H \supseteq \Alt_{(i)}(T)^+$ and with these invariants $c(0)$, $c(1)$, $K'(0)$ and $K'(1)$.
\begin{enumerate}[(i)]
\item If $c(0) = c(1) = 0$ then $N \leq 1$.
\item If $c(t) \in \{1,3\}$ and $c(1-t) = 0$ then $N \leq 2^{K'(t)}$.
\item If $c(t) = 1$, $c(1-t) = 3$ and $K'(t) > K'(1-t)$, then $N \leq 2 \cdot 2^{K'(0) \boxplus K'(1)}$.
\item If $c(0) \neq 0$, $c(1) \neq 0$ and we are not in (iii), then $N \leq 2^{K'(0) \boxplus K'(1)}$.
\end{enumerate}
\end{proposition}

\begin{proof}
In view of Theorem~\ref{maintheorem:completeinvariants'}, we simply need to give in each case an upper bound on the number of ordered pairs $(f^0_{v_0}, f^1_{v_1})$ that can be observed, where $v_0$ and $v_1$ are fixed (when $c(t) = 0$, we say that there is only one possibility for $f^t_{v_t}$ (which was not defined)). Recall that the values of $c(t)$ and $K'(t)$ completely determine the value of $K(t)$. 
\begin{itemize}
\item If $c(0) = c(1) = 0$ then we trivially have $N \leq 1$.
\item If $c(t) = 1$ and $c(1-t) = 0$, then we get by Lemma~\ref{lemma:forms} that $N \leq 2^{K(t)} = 2^{K'(t)}$, because $2^{K(t)}$ is the number of subsets of $\{0,\ldots,K(t)-1\}$.
\item If $c(t) = 3$ and $c(1-t) = 0$, then we get by Lemma~\ref{lemma:forms} that $N \leq 2^{K(t)-1} = 2^{K'(t)}$, because $2^{K(t)-1}$ is the number of subsets of $\{1,\ldots,K(t)-1\}$.
\item If $c(0) \neq 0$ and $c(1) \neq 0$ then we must distinguish some cases.
\begin{itemize}
\item If $c(0) = c(1) = 2$ then by Lemma~\ref{lemma:2-2:AB} the shape of $f^1_{v_1}$ is fully determined by the shape of $f^0_{v_0}$. Let $A$ and $B$ be the sets given by Lemma~\ref{lemma:forms} for $t = 0$. There are $2^{K'(0)}$ possibilities for $A$ and $2^{K'(1)}$ possibilities for $B$ (since $K(1)-1$ must always be contained in $B$ by Lemma~\ref{lemma:2-2:AB}). Hence, $N \leq 2^{K'(0) + K'(1)} = 2^{K'(0) \boxplus K'(1)}$ (recall that $K'(0) = K'(1)$ when $c(0)=c(1)=2$).
\item If $c(0), c(1) \in \{1,3\}$ and $K(0) = K(1)$, then Lemma~\ref{lemma:forms} can be applied twice to get $N \leq 2^{K'(0) + K'(1)}$. If $c(0) = c(1)$ then $K'(0) = K'(1)$ so $2^{K'(0)+K'(1)} = 2^{K'(0) \boxplus K'(1)}$. If $c(0) \neq c(1)$ then $\lvert K'(0)-K'(1)\rvert = 1$ and $2^{K'(0)+K'(1)} = 2^{(K'(0) \boxplus K'(1)) + 1}$.
\item If $c(0), c(1) \in \{1,3\}$ and $K(t) > K(1-t)$ (for some $t \in \{0,1\}$), then by Lemma~\ref{lemma:forms} there are at most $2^{K'(1-t)}$ possibilities for $f^{1-t}_{v_{1-t}}$. The number of possibilities for $f^t_{v_t}$ is given by Lemma~\ref{lemma:forms2}. Remark that $\lvert R\rvert = K(t) - \left\lceil\frac{K(t)-K(1-t)}{2}\right\rceil$ (where $R$ is defined as in Lemma~\ref{lemma:forms2}) and that $0$ does not belong to $R$ if and only if $K(1-t) = 0$ and $K(t)$ is odd.
\begin{itemize}
\item If $c(t)=c(1-t)=1$, there there are at most $2^{|R|}$ possibilities for $f^t_{v_t}$ and we directly get $N \leq 2^{K(0) \boxplus K(1)} = 2^{K'(0) \boxplus K'(1)}$.
\item If $c(t)=c(1-t)=3$, then $K(1-t) \neq 0$ and $0$ is never contained in $R$, so there are at most $2^{|R|-1}$ possibilities for $f^t_{v_t}$ and $N \leq 2^{K'(0) \boxplus K'(1)}$.
\item If $c(t)=1$ and $c(1-t)=3$, then there are at most $2^{\lvert R \rvert+\varepsilon}$ possibilities for $f^t_{v_t}$ where $\varepsilon=1$ if $K(t) \not \equiv K(1-t) \bmod 2$ and $\varepsilon=0$ otherwise. As $K'(t) = K(t)$ and $K'(1-t)=K(1-t)-1$, we see that $\varepsilon$ is exactly equal to $1+\left\lceil\frac{K(t)-K(1-t)}{2}\right\rceil - \left\lceil\frac{K'(t)-K'(1-t)}{2}\right\rceil$, so that $N \leq 2^{(K'(0) \boxplus K'(1)) + 1}$.
\item If $c(t)=3$ and $c(1-t)=1$, then there are at most $2^{\lvert R\setminus\{0\}\rvert +\varepsilon'}$ possibilities for $f^t_{v_t}$ where $\varepsilon' = 1$ if $K(1-t)\neq 0$ and $K(t) \not \equiv K(1-t) \bmod 2$, and $\varepsilon'=0$ otherwise. Moreover, the number $K'(1-t) + \lvert R\setminus\{0\}\rvert +\varepsilon'$ is equal to
$K'(1-t)+K(t)-\left\lceil\frac{K(t)-K(1-t)}{2}\right\rceil - 1 + \eta + \varepsilon'$,
where $\eta = 1$ if $0 \not \in R$, i.e. if $K(1-t)=0$ and $K(t)$ is odd, and $\eta = 0$ otherwise. By definition of $\eta$ and $\varepsilon'$, we see that $\eta + \varepsilon' = 1$ if $K(t)$ and $K(1-t)$ have a different parity and $\eta + \varepsilon' = 0$ otherwise. Hence, $\eta + \varepsilon'$ is exactly equal to $\left\lceil\frac{K(t)-K(1-t)}{2}\right\rceil - \left\lceil\frac{K'(t)-K'(1-t)}{2}\right\rceil$, so that we obtain $N \leq 2^{K'(0) \boxplus K'(1)}$. \qedhere
\end{itemize}
\end{itemize}
\end{itemize}
\end{proof}


\subsection{The classification theorem}

The following main theorem readily follows from the previous results.


\begin{theorem}\label{theorem:classification}
Let $T$ be the $(d_0,d_1)$-semiregular tree with $d_0,d_1 \geq 6$ and let $i$ be a legal coloring of $T$. Let $H \in \mathcal{H}_T^+$ be such that $H \supseteq \Alt_{(i)}(T)^+$. Then $H$ belongs to $\underline{\mathcal{G}}_{(i)}$.
\end{theorem}

\begin{proof}
This comes from the fact that the upper bounds given in Proposition~\ref{proposition:upperbound} are all reached by the members of $\underline{\mathcal{G}}_{(i)}$ (see Proposition~\ref{proposition:table}). Remark that, in Table~\ref{table:invariants}, the lines $7$ and $11$ (as well as the lines $8$ and $12$) give the same invariants $c(0)$ and $c(1)$, so their total number add up, thereby matching the factor $2$ in the upper bound given by Proposition~\ref{proposition:upperbound} (iii).
\end{proof}

We can now prove the next explicit formulation of Theorem~\ref{maintheorem:classification}.


\begin{primetheorem}{maintheorem:classification}[Classification]\label{maintheorem:classification'}
Let $T$ be the $(d_0,d_1)$-semiregular tree with $d_0,d_1 \geq 4$ and let $i$ be a legal coloring of $T$.
\begin{enumerate}[(i)]
\item Two groups $H,H' \in \underline{\mathcal{G}}_{(i)}$ are conjugate in $\Aut(T)$ if and only if $H = H'$ or $d_0 = d_1$ and either $H = G_{(i)}^+(Y_0,Y_1)$ and $H' = G_{(i)}^+(Y'_0,Y'_1)$ with $(Y_0,Y_1) = (Y'_1,Y'_0)$ or $H = G_{(i)}^+(X_0,X_1)^*$ and $H' = G_{(i)}^+(X'_0,X'_1)^*$ with $(X_0,X_1) = (X'_1,X'_0)$.
\item Suppose that $d_0,d_1 \geq 6$. Let $H \in \mathcal{H}_T^+$ be such that  $\underline{H}(x) \cong F_0 \geq \Alt(d_0)$ for each $x \in V_0(T)$ and $\underline{H}(y) \cong F_1 \geq \Alt(d_1)$ for each $y \in V_1(T)$. Then $H$ is conjugate in $\Aut(T)^+$ to a group belonging to $\underline{\mathcal{G}}_{(i)}$.
\end{enumerate}
\end{primetheorem}

\begin{proof}
\hspace{1cm}
\begin{enumerate}[(i)]
\item It is a direct consequence of Lemma~\ref{lemma:normalizer} (i) that two different groups in $\underline{\mathcal{G}}_{(i)}$ can never be conjugate in $\Aut(T)^+$. If $d_0\neq d_1$ then $\Aut(T) = \Aut(T)^+$ and we are done. Now suppose that $d_0 = d_1$. Then there exists $\nu \in \Aut(T) \setminus \Aut(T)^+$ not preserving the types but preserving the colors, i.e. such that $i \circ \nu = i$. Every automorphism $\tau \in \Aut(T) \setminus \Aut(T)^+$ can be written as $\tau = \mu \circ \nu$ with $\mu \in \Aut(T)^+$. The statement then follows from the fact that $\nu G_{(i)}^+(Y_0,Y_1) \nu^{-1} = G_{(i)}^+(Y_1,Y_0)$ and $\nu G_{(i)}^+(X_0,X_1)^* \nu^{-1} = G_{(i)}^+(X_1,X_0)^*$.
\item By Theorem~\ref{maintheorem:Alt'}, there exists a legal coloring $i'$ of $T$ such that $H \supseteq \Alt_{(i')}(T)^+$. Hence, by Theorem~\ref{theorem:classification}, $H$ is equal to a group belonging to $\underline{\mathcal{G}}_{(i')}$. But $\Aut(T)^+$ is transitive on the set of legal colorings of $T$, so each member of $\underline{\mathcal{G}}_{(i')}$ is conjugate in $\Aut(T)^+$ to its counterpart in $\underline{\mathcal{G}}_{(i)}$ and the conclusion follows. \qedhere
\end{enumerate}
\end{proof}


\subsection{Proofs of the corollaries}

We can now prove the different corollaries mentioned in the introduction. We actually give, for each one, a more precise formulation than its version in the introduction. For the definition of the set $\mathcal{G}'_{(i)}$, see Definition~\ref{definition:groups}.


\begin{primecorollary}{maincorollary:classification}\label{maincorollary:classification'}
Let $T$ be the $d$-regular tree with $d \geq 4$ and let $i$ be a legal coloring of~$T$.
\begin{enumerate}[(i)]
\item The members of $\mathcal{G}'_{(i)}$ are pairwise non-conjugate in $\Aut(T)$.
\item Suppose that $d \geq 6$. Let $H \in \mathcal{H}_T \setminus \mathcal{H}_T^+$ be such that $\underline{H}(v) \cong F \geq \Alt(d)$ for each $v \in V(T)$. Then $H$ is conjugate in $\Aut(T)^+$ to a group belonging to $\mathcal{G}'_{(i)}$.
\end{enumerate}
\end{primecorollary}

\begin{proof}
We start by proving (ii).
\begin{enumerate}
\item[(ii)] Clearly, $H^+ := H \cap \Aut(T)^+$ is a subgroup of index $2$ of $H$. Since $H^+(v) = H(v)$ for each $v \in V(T)$, we deduce by Lemma~\ref{lemma:2-transitive} that $H^+$ is also $2$-transitive on $\partial T$, i.e. $H^+ \in \mathcal{H}_T^+$. Moreover, $\underline{H}^+(v) \cong F \geq \Alt(d)$ for each $v \in V(T)$, so Theorem~\ref{maintheorem:classification'} can be applied to find the shapes that $H^+$ can take. It is however important to note that, if $\nu \in H \setminus H^+$, then $\nu H^+ \nu^{-1} = H^+$ while $\nu$ does not preserve the types. This means that in $H^+$ the situation must be the same for the vertices of type $0$ and the vertices of type $1$. As a consequence, $H^+$ can only be conjugate in $\Aut(T)^+$ to one of the groups $G_{(i)}^+(Y, Y)$ with $Y \in \{\varnothing, X, X^*\}$ and $G_{(i)}^+(X,X)^*$ (with $X \subset_f \N$). In other words, there exists a legal coloring $i'$ of $T$ such that $H^+$ is equal to one of the groups $G_{(i')}^+(Y, Y)$ and $G_{(i')}^+(X,X)^*$.

Since $H^+$ is normal in $H$, $H$ is contained in the normalizer of $H^+$ in $\Aut(T)$. By Lemma~\ref{lemma:normalizer} (ii), the normalizer in $\Aut(T)$ of $G_{(i')}^+(\varnothing,\varnothing)$ (resp. $G_{(i')}^+(X,X)$, $G_{(i')}^+(X^*,X^*)$ and $G_{(i')}^+(X,X)^*$) is $G_{(i')}(\varnothing,\varnothing)$ (resp. $G_{(i')}(X^*,X^*)$, $G_{(i')}(X^*,X^*)$ and $G_{(i')}(X^*,X^*)$). Using the fact that $H^+$ is a subgroup of index $2$ of $H$, we directly get that $H$ is equal to $G_{(i')}(\varnothing,\varnothing)$ when $H^+ = G_{(i')}^+(\varnothing,\varnothing)$ and that $H$ is equal to $G_{(i')}(X^*,X^*)$ when $H^+ = G_{(i')}^+(X^*,X^*)$. For the other cases, we have:
\begin{itemize}
\item If $H^+ = G_{(i')}^+(X,X)$, then the normalizer of $H^+$ is $G_{(i')}(X^*,X^*)$. To get $H$, we must observe the extensions $H^+(\nu)$ of $H^+$ by an element $\nu \in G_{(i')}(X^*,X^*)$ that does not preserve the types and such that $\nu^2 \in H^+$. There are two possibilities: either $\Sgn_{(i')}(\nu, S_X(v)) = 1$ for each $v \in V(T)$ or $\Sgn_{(i')}(\nu, S_X(v)) = -1$ for each $v \in V(T)$ (we cannot have $\Sgn_{(i')}(\nu, S_X(v)) = 1$ for each $v \in V_0(T)$ and $\Sgn_{(i')}(\nu, S_X(v)) = -1$ for each $v \in V_1(T)$ since this would imply that $\nu^2 \not \in H^+$). In the first case we get $H^+(\nu) = G_{(i)}(X,X)$. In the second case, define a new legal coloring $i''$ by $i''|_{V_0(T)} = i'|_{V_0(T)}$ and $i''|_{V_1(T)} = \tau \circ i'|_{V_1(T)}$ where $\tau \in \Sym(d)$ is an odd permutation. In this way, $H^+ = G_{(i')}^+(X,X) = G_{(i'')}^+(X,X)$ and $H^+(\nu) = G_{(i'')}(X,X)$.
\item If $H^+ = G_{(i')}^+(X,X)^*$, then the normalizer of $H^+$ is $G_{(i')}(X^*,X^*)$. Here also, we observe the extensions $H^+(\nu)$. In this case, all $\Sgn_{(i')}(\nu, S_X(v))$ with $v \in V_0(T)$ must be equal and all $\Sgn_{(i')}(\nu, S_X(v))$ with $v \in V_1(T)$ must be equal, but there is no additional condition since each such $\nu$ satisfies $\nu^2 \in H^+$. Replacing $i'$ by $i''$ as above if necessary, we can assume that $\Sgn_{(i')}(\nu, S_X(v)) = 1$ for each $v \in V_0(T)$. Then, if $\Sgn_{(i')}(\nu, S_X(v)) = 1$ for each $v \in V_1(T)$ we obtain $H^+(\nu) = G_{(i')}(X,X)^*$. On the contrary, if $\Sgn_{(i')}(\nu, S_X(v)) = -1$ for each $v \in V_1(T)$, then we get $H^+(\nu) = G_{(i')}'(X,X)^*$.
\end{itemize}
In each case, $H$ is conjugate in $\Aut(T)^+$ to a group belonging to $\mathcal{G}'_{(i)}$.
\item[(i)] Suppose that there exist two different groups $H, H' \in \mathcal{G}_{(i)}'$ which are conjugate in $\Aut(T)$. Then $H^+$ and $H^{\prime +}$ are also conjugate, and by Theorem~\ref{maintheorem:classification'} (i) this implies that $H^+ = H^{\prime +}$. Since the groups in $\underline{\mathcal{G}}_{(i)}$ are pairwise distinct (Proposition~\ref{proposition:different}), the only possibility is to have $H = G_{(i)}(X,X)^*$ and $H' = G_{(i)}'(X,X)^*$ (or the contrary) for some $X \subset_f \N$. However, $G_{(i)}(X,X)^*$ and $G_{(i)}'(X,X)^*$ are not conjugate in $\Aut(T)$. Indeed, if $H^{(\infty)}$ denotes the intersection of all normal cocompact closed subgroups of $H$, then we have $(G_{(i)}(X,X)^*)^{(\infty)} = (G_{(i)}'(X,X)^*)^{(\infty)} = G_{(i)}^+(X,X)$ but $\bigslant{G_{(i)}(X,X)^*}{G_{(i)}^+(X,X)} \cong (\C_2)^2$ while $\bigslant{G_{(i)}'(X,X)^*}{G_{(i)}^+(X,X)} \cong \C_4$.
\qedhere
\end{enumerate}
\end{proof}

Before proving Corollary~\ref{maincorollary:Theta'}, recall that $\Theta \subset \Nz$ is the set of integers $m \geq 6$ such that each finite $2$-transitive group on $\{1,\ldots,m\}$ contains $\Alt(m)$.


\begin{primecorollary}{maincorollary:Theta}\label{maincorollary:Theta'}
Let $T$ be the $(d_0,d_1)$-semiregular tree with $d_0,d_1 \in \Theta$, let $i$ be a legal coloring of $T$ and let $H \in \mathcal{H}_T$. Then $H$ is conjugate in $\Aut(T)^+$ to a group belonging to $\underline{\mathcal{G}}_{(i)}$ or $\mathcal{G}'_{(i)}$ (when $d_0=d_1$).
\end{primecorollary}

\begin{proof}
Since $H$ is $2$-transitive on $\partial T$, $H(v)$ is $2$-transitive on $E(v)$ for each $v \in V(T)$ (see Lemma~\ref{corollary:2-transitive}). By definition of $\Theta$, this implies that $\underline{H}(x) \cong F_0 \geq \Alt(d_0)$ for each $x \in V_0(T)$ and $\underline{H}(y) \cong F_1 \geq \Alt(d_1)$ for each $y \in V_1(T)$. The conclusion follows from Theorem~\ref{maintheorem:classification'} (ii) and Corollary~\ref{maincorollary:classification'} (ii).
\end{proof}


\begin{primecorollary}{maincorollary:vertextransitive}
Let $T$ be the $d$-regular tree with $d \geq 6$, let $i$ be a legal coloring of $T$ and let $H$ be a vertex-transitive closed subgroup of $\Aut(T)$. If $\underline{H}(v) \cong F \geq \Alt(d)$ for each $v \in V(T)$, then $H$ is discrete or $H$ is conjugate in $\Aut(T)^+$ to a group belonging to $\mathcal{G}'_{(i)}$.
\end{primecorollary}

\begin{proof}
By~\cite{Burger}*{Propositions~3.3.1 and~3.3.2}, the hypotheses directly imply that $H$ is discrete or $2$-transitive on $\partial T$. The conclusion follows from Corollary~\ref{maincorollary:classification'} (ii).
\end{proof}


\section{Another example on the \texorpdfstring{$(4,d_1)$}{(4,d_1)}-semiregular tree}\label{section:counterexample}
Let $T$ be the $(4,d_1)$-semiregular tree with $d_1 \geq 4$. In this section, we construct a (non-linear) group $G \in \mathcal{H}^+_T$ for which there is no legal coloring $i$ of $T$ such that $G \supseteq \Alt_{(i)}(T)^+$. This group will therefore be different from all the groups defined in Section~\ref{section:simple}.

To avoid any confusion, we use the letter $j$ for the legal colorings of trees that will help to construct our group and the letter $i$ will only be used for other legal colorings appearing in the results.

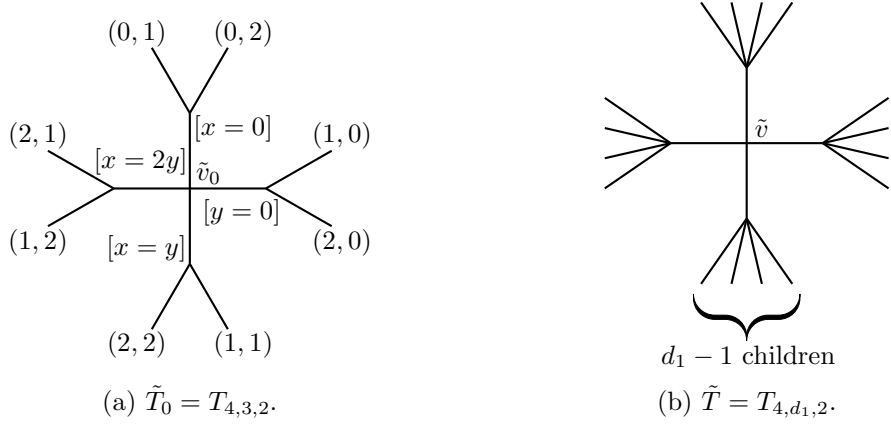
\begin{figure}
\begin{subfigure}{.5\textwidth}
\centering
\begin{pspicture*}(-2.4,-2.4)(2.4,2.5)
\fontsize{10pt}{10pt}\selectfont
\psset{unit=1cm}

\rput(0.25,0.2){$\tilde{v}_0$}

\psline(0,0)(1,0) \rput(0.68,-0.32){$[y = 0]$}
\psline(0,0)(-1,0) \rput(-0.65,0.35){$[x = 2y]$}
\psline(0,0)(0,1) \rput(0.57,0.8){$[x = 0]$}
\psline(0,0)(0,-1) \rput(-0.57,-0.8){$[x = y]$}

\psline(1,0)(1.866,0.5) \rput(2,0.7){$(1,0)$}
\psline(1,0)(1.866,-0.5) \rput(2,-0.72){$(2,0)$}
\psline(-1,0)(-1.866,0.5) \rput(-2,0.7){$(2,1)$}
\psline(-1,0)(-1.866,-0.5) \rput(-2,-0.72){$(1,2)$}
\psline(0,1)(0.5,1.866) \rput(0.7,2.05){$(0,2)$}
\psline(0,1)(-0.5,1.866) \rput(-0.7,2.05){$(0,1)$}
\psline(0,-1)(0.5,-1.866) \rput(0.7,-2.07){$(1,1)$}
\psline(0,-1)(-0.5,-1.866) \rput(-0.7,-2.07){$(2,2)$}
\end{pspicture*}
\caption{$\tilde{T}_0 = T_{4,3,2}$.}\label{picture:firstT}
\end{subfigure}%
\begin{subfigure}{.5\textwidth}
\centering
\begin{pspicture*}(-1.9,-3)(1.9,1.9)
\fontsize{10pt}{10pt}\selectfont
\psset{unit=1cm}

\rput(0.2,0.2){$\tilde{v}$}

\psline(0,0)(1,0)
\psline(0,0)(-1,0)
\psline(0,0)(0,1)
\psline(0,0)(0,-1)

\psline(1,0)(1.866,0.2) 
\psline(1,0)(1.866,-0.2)
\psline(1,0)(1.866,0.6) 
\psline(1,0)(1.866,-0.6)
\psline(-1,0)(-1.866,0.2)
\psline(-1,0)(-1.866,-0.2)
\psline(-1,0)(-1.866,0.6)
\psline(-1,0)(-1.866,-0.6)
\psline(0,1)(0.2,1.866)
\psline(0,1)(-0.2,1.866)
\psline(0,1)(0.6,1.866)
\psline(0,1)(-0.6,1.866)
\psline(0,-1)(0.2,-1.866) 
\psline(0,-1)(-0.2,-1.866)
\psline(0,-1)(0.6,-1.866) 
\psline(0,-1)(-0.6,-1.866)

\psbrace[rot=90,ref=lC,nodesepB=6pt,nodesepA=-32pt](-0.7,-2)(0.7,-2){$d_1-1$ children}

\end{pspicture*}
\caption{$\tilde{T} = T_{4,d_1,2}$.}\label{picture:secondT}
\end{subfigure}
\caption{Construction of the group $\tilde{G}$.}
\end{figure}

First consider the rooted tree $\tilde{T}_0 = T_{4,3,2}$, i.e. the rooted tree of depth~$2$ where the root $v_0$ has $4$ children and each child of $v_0$ has $3-1=2$ children. Let $\psi$ be a bijection between $(\F_3)^2 \setminus \{(0,0)\}$ and the set of eight vertices of $\tilde{T}_0$ at distance $2$ from $v_0$, such that two such vertices $x$ and $y$ have the same parent if and only if $\psi^{-1}(x)$ and $\psi^{-1}(y)$ are a multiple of one another (see Figure~\ref{picture:firstT}). The four children of $v_0$ thus correspond to the four lines of $(\F_3)^2$, or in other words to the four elements of the projective line over $\F_3$. The natural action of $\SL(2,\F_3)$ on $(\F_3)^2 \setminus \{(0,0)\}$ induces via $\psi$ an action of $\SL(2,\F_3)$ on $\tilde{T}_0$. Let $\tilde{G}_0$ be the image of $\SL(2,\F_3)$ in $\Aut(\tilde{T}_0)$ defined in this way. It is clear that the pointwise stabilizer of $B(v_0,1)$ in $\tilde{G}_0$ corresponds to the two matrices $\left(\begin{smallmatrix}1 & 0 \\ 0 & 1\end{smallmatrix}\right)$ and $\left(\begin{smallmatrix} -1 & 0 \\ 0 & -1\end{smallmatrix}\right)$. Hence, $\underline{\tilde{G}_0}(v_0) \cong \PSL(2,\F_3)$ which is in turn isomorphic to $\Alt(4)$.

Now consider the rooted tree $\tilde{T} = T_{4,d_1,2}$, i.e. the rooted tree of depth~$2$ where the root $v$ has $4$ children and each child of $v$ has $d_1-1$ children (see Figure~\ref{picture:secondT}). Fix a legal coloring $\tilde{j}$ of $\tilde{T}$ and a legal coloring $\tilde{j}_0$ of $\tilde{T}_0$ with $\tilde{j}_0(v_0) = \tilde{j}(v)$ and let $\alpha$ be the bijection from $B(v,1)$ to $B(v_0,1)$ preserving the colorings. We define the map $f \colon \Aut(\tilde{T}) \to \Aut(\tilde{T}_0)$ by saying that, if $g \in \Aut(\tilde{T})$, then $f(g) \in \Aut(\tilde{T}_0)$ is the unique automorphism of $\tilde{T}_0$ such that $f(g)(\alpha(x)) = \alpha(g(x))$ for each $x \in B(v,1)$ and $\sigma_{(\tilde{j}_0)}(f(g),\alpha(x))$ has the same parity as $\sigma_{(\tilde{j})}(g,x)$ for each $x \in S(v,1)$. Then consider $\tilde{G} = f^{-1}(\tilde{G}_0) \leq \Aut(\tilde{T})$.

It is clear from the definition of $\tilde{G}$ that $\underline{\tilde{G}}(v) \cong \Alt(4)$, and the next lemma shows that $\tilde{G}$ never contains $\Alt_{(\tilde{i})}(\tilde{T})$ for a legal coloring $\tilde{i}$ of $\tilde{T}$.


\begin{lemma}\label{lemma:G}
There does not exist a legal coloring $\tilde{i}$ of $\tilde{T}$ such that $\tilde{G} \supseteq \Alt_{(\tilde{i})}(\tilde{T})$.
\end{lemma}

\begin{proof}
By contradiction, assume that such a coloring exists. From this one, we can construct a legal coloring $\tilde{i}_0$ of $\tilde{T}_0$ such that  $\tilde{G}_0 \supseteq \Alt_{(\tilde{i}_0)}(\tilde{T}_0)$. Indeed, it suffices to set $\tilde{i}_0|_{B(v_0,1)} = \tilde{j}_0|_{B(v_0,1)}$ and then, for each $x \in S(v_0,1)$, to define $\tilde{i}_0$ on $S(x,1) \setminus \{v_0\}$ such that $\tilde{i}_0 \tilde{j}_0|_{S(x,1)}^{-1} \in \Sym(3)$ has the same parity as $\tilde{i} \tilde{j}|_{S(\alpha^{-1}(x),1)}^{-1} \in \Sym(d_1)$. In this way, $f(\Alt_{(\tilde{i})}(\tilde{T})) = \Alt_{(\tilde{i}_0)}(\tilde{T}_0)$ and thus $\tilde{G}_0 = f(\tilde{G}) \supseteq f(\Alt_{(\tilde{i})}(\tilde{T})) = \Alt_{(\tilde{i}_0)}(\tilde{T}_0)$.

Let us name each vertex of $S(v_0,1)$ with the corresponding line in $(\F_3)^2 \setminus \{0,0\}$, i.e. with $[x=0]$, $[y=0]$, $[x=y]$ or $[x=2y]$. Let $g \in \Alt_{(\tilde{i}_0)}(\tilde{T}_0)$ be such that $g$ interchanges $[x=0]$ and $[y=0]$ and interchanges $[x=y]$ and $[x=2y]$. Since $g \in \Alt_{(\tilde{i}_0)}(\tilde{T}_0) \subseteq \tilde{G}_0$, $g$ acts as $\left(\begin{smallmatrix}0 & 1 \\ -1 & 0\end{smallmatrix}\right)$ or as $\left(\begin{smallmatrix}0 & -1 \\ 1 & 0\end{smallmatrix}\right)$ on $\tilde{T}_0$ (these are the only elements of $\tilde{G}_0$ acting as $g$ on $B(v_0,1)$). In both cases, $g^2$ acts as $\left(\begin{smallmatrix}-1 & 0 \\ 0 & -1\end{smallmatrix}\right)$ on $\tilde{T}_0$. But $\left(\begin{smallmatrix}-1 & 0 \\ 0 & -1\end{smallmatrix}\right)$ fixes $B(v_0,1)$ and acts as a transposition at each vertex of $S(v_0,1)$, so it cannot be contained in $\Alt_{(\tilde{i}_0)}(\tilde{T}_0)$. This is a contradiction with the fact that $g^2 \in \Alt_{(\tilde{i}_0)}(\tilde{T}_0)$.
\end{proof}

Using the group $\tilde{G} \leq \Aut(\tilde{T})$, we can now construct a group $G \in \mathcal{H}_T^+$ which acts locally as $\tilde{G}$. For each $x \in V_0(T)$, fix a map $J_x \colon B(x,2) \to \tilde{T}$. In this way, for each $x \in V_0(T)$ and each $g \in \Aut(T)^+$, we can define
$$\Sigma_{(J)}(g, x) = J_{g(x)} \circ g \circ J_{x}^{-1} \in \Aut(\tilde{T}).$$
This allows us to define $G \leq_{cl} \Aut(T)^+$ by
$$G := \{g \in \Aut(T)^+ \suchthat \Sigma_{(J)}(g,x) \in \tilde{G} \text{ for each $x \in V_0(T)$}\}.$$
The fact that $G$ is $2$-transitive on $\partial T$ is not completely immediate.


\begin{lemma}
The group $G$ belongs to $\mathcal{H}_T^+$.
\end{lemma}

\begin{proof}
By Lemma~\ref{lemma:2-transitive}, it suffices to prove that $G(v)$ is transitive on $\partial T$ for each $v \in V(T)$. As $G$ is closed in $\Aut(T)$, we can simply show that the fixator in $G$ of a geodesic $(v,w)$ with $v, w \in V(T)$ always acts transitively on $E(w) \setminus \{e\}$, where $e$ is the edge of $(v,w)$ adjacent to $w$. If $x$ and $y$ are two vertices adjacent to $w$ but not on $(v,w)$, then we must find $g \in G$ fixing $(v,w)$ and such that $g(x) = y$. We can simply construct such an element $g$ by defining $g(x)=y$ and $g(e)=e$, and then by extending $g$ to larger and larger balls, so that $g$ fixes $(v,w)$ and $\Sigma_{(J)}(g,z) \in \tilde{G}$ for each $z \in V_0(T)$. One easily checks, using the fact that $d_1 \geq 4$, that there is sufficient freedom in $\tilde{G}$ to do so.
\end{proof}

Finally, as a corollary of Lemma~\ref{lemma:G} we find that $G$ is indeed not isomorphic to any group defined in Section~\ref{section:simple}.


\begin{proposition}
We have $\underline{G}(x) \cong \Alt(4)$ for each $x \in V_0(T)$ and $\underline{G}(y) \cong \Sym(d_1)$ for each $y \in V_1(T)$, but there does not exist a legal coloring $i$ of $T$ such that $G \supseteq \Alt_{(i)}(T)^+$.
\end{proposition}

\begin{proof}
The fact that $\underline{G}(x) \cong \Alt(4)$ for $x \in V_0(T)$ and $\underline{G}(y) \cong \Sym(d_1)$ for $y \in V_1(T)$ readily follows from the definition of $G$. Now consider a legal coloring $i$ of $T$. We show that $G \not \supseteq \Alt_{(i)}(T)^+$. Fix $x \in V_0(T)$ and consider the legal coloring $\tilde{i} = i \circ J_x^{-1}$ of $\tilde{T}$. By Lemma~\ref{lemma:G}, there exists $\tilde{g} \in \Alt_{(\tilde{i})}(\tilde{T})$ such that $\tilde{g} \not \in \tilde{G}$. One can then construct an element $g \in \Alt_{(i)}(T)^+$ fixing $x$ and with $\Sigma_{(J)}(g,x) = \tilde{g}$, which is therefore such that $g \in \Alt_{(i)}(T)^+ \setminus G$. Hence, we have $G \not \supseteq \Alt_{(i)}(T)^+$.
\end{proof}

\begin{appendix}


\section{Topologically isomorphic groups in \texorpdfstring{$\mathcal{H}_T$}{H_T}}\label{appendix:isomorphic}

We show in this appendix the following proposition, stating that two groups in $\mathcal{H}_T$ are topologically isomorphic if and only if they are conjugate in $\Aut(T)$. This is a folklore result but, because of the lack in finding a suitable reference and as suggested by the referee, we give its full proof here.


\begin{proposition}\label{proposition:isomorphic}
Let $T$ be the $(d_0,d_1)$-semiregular tree with $d_0, d_1 \geq 3$ and let $H, H' \in \mathcal{H}_T$ be isomorphic as topological groups. Then $H$ and $H'$ are conjugate in $\Aut(T)$.
\end{proposition}

\begin{proof}
Since $H$ acts edge-transitively on $T$ (see Lemma~\ref{corollary:2-transitive}), the vertex stabilizers $H_v$ and edge stabilizers $H_e$ in $H$ are all pairwise distinct. Moreover, $H_v$ is a maximal compact subgroup of $H$ for each $v \in V(T)$, $H_e$ is a maximal compact subgroup of $H$ for each $e \in E(T)$ if and only if $H \not \in \mathcal{H}_T^+$, and these are the only maximal compact subgroups of $H$ (see~\cite{HarmonicAnalysis}*{Theorem~5.2}). Given the group $H$ and all its compact maximal subgroups, one can also recognize which of them must be vertex stabilizers. Indeed, if $H \not \in \mathcal{H}_T^+$ and $K = H_e$ is an edge stabilizer in $H$ then there exists another maximal compact subgroup $K'$ of $H$ such that $[K : K \cap K'] = 2$ (namely $K' = H_v$ where $v$ is a vertex of $e$). On the contrary, if $K = H_v$ is a vertex stabilizer in $H$ (we do not suppose $H \not \in \mathcal{H}_T^+$ here), then there exists no maximal compact subgroup $K'$ of $H$ such that $[K : K \cap K'] = 2$, because $d_0, d_1 \geq 3$ and $H$ is edge-transitive. The vertex stabilizers in $H$ can thus be exactly identified among the subgroups of $H$, without knowing anything about the action of $H$ on $T$. The same is true for $H'$.

Now let $\varphi \colon H \to H'$ be an isomorphism of topological groups. For each $v \in V(T)$, the previous discussion shows that there is a unique vertex $\tau(v) \in V(T)$ such that $\varphi(H_v) = H'_{\tau(v)}$ and that the map $\tau \colon V(T) \to V(T)$ is a bijection. Moreover, two vertices $v, v' \in V(T)$ are neighbors in $T$ if and only if $[H_v : H_{v'} \cap H_v] \leq [H_v : H_w \cap H_v]$ for all vertices $w \neq v$. This indeed follows easily from Lemma~\ref{corollary:2-transitive}. In view of the definition of $\tau$, this implies that $v$ and $v'$ are adjacent if and only if $\tau(v)$ and $\tau(v')$ are adjacent. In other words, $\tau$ is an automorphism of $T$. We finally claim that $\tau h \tau^{-1} = \varphi(h)$ for all $h \in H$. Indeed, we have successively
\[H'_{\tau h \tau^{-1}(v)} = \varphi(H_{h \tau^{-1}(v)}) = \varphi(h H_{\tau^{-1}(v)} h^{-1}) = \varphi(h) H'_v \varphi(h)^{-1} = H'_{\varphi(h)(v)}. \qedhere\]
\end{proof}


\section{Asymptotic density of the set \texorpdfstring{$\Theta$}{Theta}}\label{appendix:theta}

In this appendix, we give an explicit expression for $\Theta$ coming from the classification of finite $2$-transitive groups and deduce that $\Theta$ is asymptotically dense in $\Nz$.


\begin{proposition}\label{proposition:theta}
The set $\Theta$ is equal to
\begin{dmath*}
\Theta = {\left\{m \in \Nz \suchthat m \geq 6\right\}} \setminus \left({\left\{p^d \suchthat p \text{ prime, } d \geq 1\right\}} \cup {\left\{\frac{p^{dr}-1}{p^d-1} \suchthat p \text{ prime, } d \geq 1, r \geq 2\right\}} \cup {\left\{2^{2d-1}\pm 2^{d-1} \suchthat d \geq 3 \right\}} \cup \{22,176,276\}\right)
\end{dmath*}
\end{proposition}

\begin{proof}
This is a consequence of the classification of finite $2$-transitive groups, see~\cite{Cameron}*{Tables~7.3 and~7.4}. Note that there exist some sporadic $2$-transitive groups with $m \not \in \{22,176,276\}$, but we did not write these values for $m$ since they are already contained in at least one of the infinite families.
\end{proof}


\begin{corollary}\label{corollary:dense}
The asymptotic density $D(\Theta)$ of $\Theta$ in $\Nz$ is equal to $1$, i.e.
$$\lim_{n \to +\infty}\frac{|\Theta \cap \{1,\ldots, n\}|}{n} = 1.$$
\end{corollary}

\begin{proof}
It suffices to prove that the asymptotic density of each of the three infinite families is equal to $0$. First, we have
\begin{align*}
\left|\left\{2^{2d-1}\pm 2^{d-1} \suchthat d \geq 3 \right\} \cap \{1,\ldots, n\}\right|
& \leq 2\cdot\left|\left\{d \geq 3 \suchthat 2^{2d-1} - 2^{d-1} \leq n\right\}\right|\\
& \leq 2\cdot\left|\left\{d \geq 3 \suchthat 2^{2d-2} \leq n\right\}\right|\\
& = 2\cdot\left|\left\{d \geq 3 \suchthat 2d-2 \leq \log_2(n)\right\}\right|\\
& \leq 2\cdot \log_2(n)
\end{align*}
This directly implies that $D\left(\left\{2^{2d-1}\pm 2^{d-1} \suchthat d \geq 3 \right\}\right) = 0$.

We now show that the density of $\left\{\frac{p^{dr}-1}{p^d-1} \suchthat p \text{ prime, } d \geq 1, r \geq 2\right\}$ is zero. The proof that the density of $\left\{p^d \suchthat p \text{ prime, } d \geq 1\right\}$ is zero is similar and even easier. To simplify the notation, define $R(n) := \left|\left\{\frac{p^{dr}-1}{p^d-1} \suchthat p \text{ prime, } d \geq 1, r \geq 2\right\} \cap \{1,\ldots, n\}\right|$ so that the density we must compute is $\lim_{n \to \infty}\frac{R(n)}{n}$. Since $\frac{p^{dr}-1}{p^d-1} \geq p^{d(r-1)}$, we have
\begin{align*}
R(n)
& \leq \left|\left\{(p,d,r) \suchthat p \text{ prime, } d \geq 1, r \geq 2, p^{d(r-1)} \leq n\right\}\right| \\
& \leq \sum_{d=1}^\infty \sum_{r=2}^\infty \left|\left\{p \text{ prime} \suchthat p \leq n^{\frac{1}{d(r-1)}} \right\}\right| \\
& = \sum_{d=1}^\infty \sum_{r=2}^\infty \pi(n^{\frac{1}{d(r-1)}})
\end{align*}
where $\pi(x)$ is the number of prime numbers less or equal to $x$. When $d(r-1) > \log_2(n)$, we have $n^{\frac{1}{d(r-1)}} < 2$ and hence $\pi(n^{\frac{1}{d(r-1)}}) = 0$. If $L(n) := \lfloor\log_2(n)\rfloor$, we therefore have
$$R(n) \leq \sum_{d=1}^{L(n)} \sum_{r=2}^{L(n)+1} \pi(n^{\frac{1}{d(r-1)}})$$
By the prime number theorem, we have $\displaystyle\lim_{x \to \infty}\frac{\pi(x)\ln(x)}{x} = 1$, so there exists $C > 0$ such that $\pi(x) \leq C \displaystyle\frac{x}{\ln(x)}$ for all $x > 0$. We therefore get
\begin{align*}
R(n)
& \leq C \sum_{d=1}^{L(n)} \sum_{r=2}^{L(n)+1} \frac{n^{\frac{1}{d(r-1)}}}{\ln(n^{\frac{1}{d(r-1)}})} \\
& \leq \frac{C}{\ln(n)} \sum_{d=1}^{L(n)} \sum_{r=2}^{L(n)+1} d(r-1) \cdot n^{\frac{1}{d(r-1)}}
\end{align*}
Separating the case $(d,r) = (1,2)$ from the $(L(n)^2-1)$ other cases gives
$$R(n) \leq \frac{C}{\ln(n)} \left(n + (L(n)^2-1)L(n)^2 \cdot n^{\frac{1}{2}}\right)$$
Hence,
\[\frac{R(n)}{n} \leq \frac{C}{\ln(n)} \left(1 + \frac{L(n)^4}{\sqrt{n}}\right) \to 0 \qedhere \]
\end{proof}

\end{appendix}



\begin{bibdiv}
\begin{biblist}

\bib{Banks}{article}{
author = {Banks, Christopher C.},
author = {Elder, Murray},
author = {Willis, George A.},
title = {Simple groups of automorphisms of trees determined by their actions on finite subtrees},
journal = {J. Group Theory},
volume = {18},
number = {2},
year = {2015},
pages = {235--261}
}

\bib{Burger}{article}{
author = {Burger, Marc},
author = {Mozes, Shahar},
title = {Groups acting on trees: from local to global structure},
journal = {Inst. Hautes Études Sci. Publ. Math. },
volume = {92},
year = {2000},
pages = {113--150}
}

\bib{Burger2}{article}{
author = {Burger, Marc},
author = {Mozes, Shahar},
title = {Lattices in products of trees},
journal = {Inst. Hautes Études Sci. Publ. Math. },
volume = {92},
year = {2000},
pages = {151--194}
}

\bib{Cameron}{book}{
author = {Cameron, Peter J.},
title = {Permutation Groups},
series = {London Math. Soc. Stud. Texts},
volume = {45},
publisher = {Cambridge University Press},
year = {1999},
place = {Cambridge}
}

\bib{CRW}{article}{
author = {Caprace, Pierre-Emmanuel},
author = {Reid, Colin D.},
author = {Willis, George A.},
title = {Locally normal subgroups of totally disconnected groups. Part II: Compactly generated simple groups},
note = {Preprint: \href{http://arxiv.org/abs/1401.3142}{arXiv:1401.3142}},
year = {2014}
}

\bib{Stulemeijer}{article}{
author = {Caprace, Pierre-Emmanuel},
author = {Stulemeijer, Thierry},
title = {Totally disconnected locally compact groups with a linear open subgroup},
journal = {Int. Math. Res. Not.},
volume = {24},
year = {2015},
pages = {13800–13829}
}

\bib{Dixon}{book}{
author = {Dixon, John D.},
author = {Mortimer, Brian},
title = {Permutation Groups},
series = {Grad. Texts in Math.},
volume = {163},
publisher = {Springer-Verlag},
year = {1996},
place = {New York}
}

\bib{HarmonicAnalysis}{book}{
author = {Fig\'a-Talamanca, Alessandro},
author = {Nebbia, Claudio},
title = {Harmonic analysis and representation theory for groups acting on homogeneous trees},
series = {London Math. Soc. Lecture Note Ser.},
volume = {162},
publisher = {Cambridge University Press},
year = {1991},
place = {Cambridge},
}

\bib{Goursat}{article}{
author = {Goursat, \'Edouard},
title = {Sur les substitutions orthogonales et les divisions régulières de l'espace},
journal = {Ann. Sci. \'Ec. Norm. Supér. (3)},
volume = {6},
year = {1889},
language={French},
pages = {9--102}
}

\bib{Lang}{book}{
author = {Lang, Serge},
title = {Algebra. Revised third edition},
series = {Grad. Texts in Math.},
volume = {211},
publisher = {Springer-Verlag},
year = {2002},
place = {Cambridge},
}

\bib{Smith}{article}{
author = {Smith, Simon M.},
title = {A product for permutation groups and topological groups},
note = {Preprint: \href{http://arxiv.org/abs/1407.5697}{arXiv:1407.5697}},
year = {2014}
}

\bib{Titsarbres}{article}{
author={Tits, Jacques},
title={Sur le groupe des automorphismes d'un arbre},
conference={
title={Essays on topology and related topics (Mémoires dédiés à Georges de Rham)},
},
book={
publisher={Springer, New York},
},
date={1970},
language={French},
pages={188--211}
}

\bib{Trofimov}{article}{
author = {Trofimov, Vladimir I.},
title = {Vertex stabilizers of graphs and tracks, I},
journal = {European J. Combin.},
volume = {28},
number = {2},
year = {2007},
pages = {613--640}
}


\bib{Willis}{article}{
author = {Willis, George A.},
title = {The scale and tidy subgroups for endomorphisms of totally disconnected locally compact groups},
journal = {Math. Ann.},
volume = {361},
year = {2015},
pages = {403-442}
}

\bib{Zassenhaus}{article}{
author = {Zassenhaus, Hans},
title = {Über endliche Fastkörper},
journal = {Abh. Math. Sem. Hamburg},
volume = {11},
number = {1},
year = {1935},
language={German},
pages = {187--220}
}

\end{biblist}
\end{bibdiv}
\end{document}